\documentclass[11pt]{article}
\usepackage{color}

\usepackage{amssymb}   
\usepackage{amsthm}    
\usepackage{amsmath}   
\usepackage{stmaryrd}  
\usepackage{titletoc}  
\usepackage{mathrsfs}  
\usepackage{graphicx}

\newlength{\defbaselineskip}
\newcommand{\setlinespacing}[1]%
           {\setlength{\baselineskip}{#1 \defbaselineskip}}

\theoremstyle{plain}
\newtheorem{thm}{Theorem}[section]

\newtheorem{lem}[thm]{Lemma}
\newtheorem{prop}[thm]{Proposition}

\theoremstyle{definition}
\newtheorem{defn}{Definition}[section]

\newtheorem{rmk}{Remark}[section]

\newcommand{\eps}{\varepsilon}

\DeclareMathOperator*{\esssup}{esssup}
\DeclareMathOperator*{\essinf}{essinf}

\newcommand{\cO}{\mathcal{O}}

\newcommand{\cL}{\mathcal{L}}
\newcommand{\cT}{\mathcal{T}}
\newcommand{\cM}{\mathcal{M}}
\newcommand{\cB}{\mathcal{B}}
\newcommand{\cA}{\mathcal{A}}
\newcommand{\cS}{\mathcal{S}}
\newcommand{\cE}{\mathcal{E}}
\newcommand{\cG}{\mathcal{G}}

\newcommand{\bH}{\mathbb{H}}
\newcommand{\bP}{\mathbb{P}}
\newcommand{\bR}{\mathbb{R}}
\newcommand{\bN}{\mathbb{N}}

\newcommand{\sF}{\mathscr{F}}
\newcommand{\sP}{\mathscr{P}}

\newcommand{\sC}{\mathscr{C}}

\makeatletter\@addtoreset{equation}{section} \makeatother
 \allowdisplaybreaks
\begin{document}

\title{Stochastic Differential Games with Random Coefficients and Stochastic Hamilton-Jacobi-Bellman-Isaacs Equations}

\author{Jinniao Qiu\footnotemark[1]  \and  Jing Zhang\footnotemark[2] }
\footnotetext[1]{Department of Mathematics \& Statistics, University of Calgary, 2500 University Drive NW, Calgary, AB T2N 1N4, Canada. \textit{E-mail}: \texttt{jinniao.qiu@ucalgary.ca}. J. Qiu was partially supported by the National Science and Engineering Research Council of Canada and by the start-up funds from the University of Calgary. }
%
\footnotetext[2]{Department of Finance and Control Sciences, School of Mathematical Sciences, Fudan University, Shanghai 200433, China.
\textit{E-mail}: \texttt{zhang\_jing@fudan.edu.cn}. J. Zhang is partially supported by National Key R\&D Program of China 2018YFA0703900, National Natural Science Foundation of China (12031009, 11701404, 11401108) and Key Laboratory of Mathematics for Nonlinear Sciences (Fudan University), Ministry of Education.}
%

\maketitle

\begin{abstract}
In this paper, we study a class of zero-sum two-player stochastic differential games with the controlled stochastic differential equations and the payoff/cost functionals of recursive type. As opposed to the pioneering work by Fleming and  Souganidis [Indiana Univ. Math. J., 38 (1989), pp.~293--314] and the seminal work by Buckdahn and Li [SIAM J. Control Optim., 47 (2008), pp.~444--475], the involved coefficients may be random, going beyond the Markovian framework and leading to the random upper and lower value functions. We first prove the dynamic programming principle for the game, and then under the standard Lipschitz continuity assumptions on the coefficients, the upper and lower value functions are shown to be the viscosity solutions of the upper and the lower  fully nonlinear stochastic Hamilton-Jacobi-Bellman-Isaacs (HJBI) equations, respectively.  A stability property of viscosity solutions is also proved. Under certain additional regularity assumptions on the diffusion coefficient, the uniqueness of the viscosity solution is addressed as well.

\end{abstract}

{\bf Mathematics Subject Classification (2010):}  49L20, 49L25, 93E20, 35D40, 60H15

{\bf Keywords:} stochastic Hamilton-Jacobi-Bellman-Isaacs equation, stochastic differential game, backward stochastic partial differential equation, viscosity solution

\section{Introduction}
Let $(\Omega,\sF,\{\sF_t\}_{t\geq0},\bP)$ be a complete filtered probability space carrying an $m$-dimensional Wiener process $W=\{W_t:t\in[0,\infty)\}$ such that $\{\sF_t\}_{t\geq0}$ is the natural filtration generated by $W$ and augmented by all the
$\bP$-null sets in $\sF$. We denote by $\sP$ the $\sigma$-algebra of the predictable sets on $\Omega\times[0,T]$ associated with $\{\sF_t\}_{t\geq0}$, and for each $t\geq 0$, $E_{\sF_t}[\,\cdot\,]$ represents the conditional expectation with respect to $\sF_t$.

We consider the following controlled stochastic differential equation (SDE):
\begin{equation}\label{state-proces-control}
\left\{
\begin{split}
  &dX_s= b(s,X_s ,\theta_s,\gamma_s)\,ds
  +\sigma(s,X_s ,\theta_s,\gamma_s)\,dW_s,\quad 0 \leq s\leq T;\\
 & X_0= x\in \bR^d.
\end{split}
\right.
\end{equation}
Here and throughout the paper, $T\in (0,\infty)$ is a fixed deterministic terminal time. Let $\Theta_0$ and $\Gamma_0 $ be two nonempty compact sets in $\bR^n$, and denote by $\Theta$ (resp. $\Gamma$) the set of all the $\Theta_0$-valued (resp. $\Gamma_0$-valued) and $\{\sF_t\}_{t\geq0}$-adapted processes.   The process $\{X_t\}_{t\in[0,T]}$ is the {\sl state process}. It is governed by the {\sl controls} $\theta\in\Theta$ and $\gamma\in\Gamma$. We sometimes write $X^{r,x;\theta,\gamma}_t$ for $0\leq r\leq t\leq T$ to indicate the dependence of the state process on the controls $\theta$ and $\gamma$, the initial time $r$, and the initial state $x\in \mathbb{R}^d$.
As the payoff for player I  and the cost for player II, the functional
\begin{align}
J(t,x;\theta,\gamma):=Y_t^{t,x;\theta,\gamma} \label{dynamic-cost}
\end{align}
is given in terms of the solution $(Y^{t,x;\theta,\gamma},Z^{t,x;\theta,\gamma})$ to the backward stochastic differential equation (BSDE for short)
\begin{equation}\label{BSDE}
  \left\{
  \begin{split}
  -dY_s^{t,x;\theta,\gamma}&=\,f(s,X_s^{t,x;\theta,\gamma},Y_s^{t,x;\theta,\gamma},Z_{s}^{t,x;\theta,\gamma},\theta_s,\gamma_s)\,ds
  -Z_{s}^{t,x;\theta,\gamma}\,dW_s,\quad s\in[t, T];\\
 Y_T^{t,x;\theta,\gamma}&=\,\Phi(X_T^{t,x;\theta,\gamma}).
    \end{split}
  \right.
\end{equation}

As standard in the literature of stochastic differential game, the players cannot just play controls against controls for the existence of game value, but they may play strategies versus controls: Given one player's control, the other player chooses accordingly a \textit{nonanticipative} strategy from an admissible set. A nonanticipative strategy admissible for player I is a mapping $\alpha: \Gamma\mapsto \Theta$ such that for any stopping time $\tau\leq T$ and any $\gamma^1,\gamma^2\in\Gamma$ with $\gamma^1\equiv \gamma^2$ on $[0,\tau]$, there holds $\alpha(\gamma^1)=\alpha(\gamma^2)$ on $[0,\tau]$. For player II , the nonanticipative strategies $\mu:\Theta\mapsto \Gamma$ are defined analogously. Denote by $\cA$ (resp. $\cM$) the set of all the nonanticipative strategies admissible for player I (resp. player II). We define the lower value function of our stochastic differential game
\begin{align}
V(t,x)=\essinf_{\mu\in\cM}\esssup_{\theta\in\Theta}J(t,x;\theta,\mu(\theta)),\quad t\in[0,T],
\label{eq-value-func-low}
\end{align}
and the upper value function is given by 
\begin{align}
U(t,x)=\esssup_{\alpha\in\cA}\essinf_{\gamma\in\Gamma}J(t,x;\alpha(\gamma),\gamma),\quad t\in[0,T].
\label{eq-value-func-up}
\end{align}

Unlike the standard literature, we consider the general non-Markovian cases where the coefficients $b,\sigma,f$, and $\Phi$ depend not only on time, space and controls but also \textit{explicitly} on $\omega\in\Omega$. With the generalized dynamic programming principle, it is shown that the value functions $V$ and $U$ are \textit{random} fields and satisfy the stochastic Hamilton-Jacobi-Bellman-Isaacs (HJBI) equations
\begin{equation}\label{SHJBI-low}
  \left\{
  \begin{split}
  -dV(t,x)=\,& 
 \mathbb{H}_-(t,x,D^2V,D\psi,DV,V,\psi) 
 \,dt -\psi(t,x)\, dW_{t}, \quad
                     (t,x)\in Q:=[0,T)\times\bR^d;\\
    V(T,x)=\, &\Phi(x), \quad x\in\bR^d,
    \end{split}
  \right.
\end{equation}
and 
\begin{equation}\label{SHJBI-up}
  \left\{
  \begin{split}
  -dU(t,x)=\,& 
 \mathbb{H}_+(t,x,D^2U,D\zeta,DU,U,\zeta) 
 \,dt -\zeta(t,x)\, dW_{t}, \quad
                     (t,x)\in Q;\\
    U(T,x)=\, &\Phi(x), \quad x\in\bR^d,
    \end{split}
  \right.
\end{equation}
respectively, with
\begin{align*}
\mathbb{H}_-(t,x,A,B,p,y,z)
= \esssup_{\theta\in\Theta_0}\essinf_{\gamma\in \Gamma_0} \bigg\{&
\text{tr}\left(\frac{1}{2}\sigma \sigma'(t,x,\theta,\gamma) A+\sigma(t,x,\theta,\gamma) B\right)
       +b'(t,x,\theta,\gamma)p \\
       &+f(t,x,y,z+\sigma'(t,x,\theta,\gamma)p,\theta,\gamma)
                \bigg\}, \\
 \mathbb{H}_+(t,x,A,B,p,y,z)
= \essinf_{\gamma\in \Gamma_0}\esssup_{\theta\in\Theta_0} \bigg\{&
\text{tr}\left(\frac{1}{2}\sigma \sigma'(t,x,\theta,\gamma) A+\sigma(t,x,\theta,\gamma) B\right)
       +b'(t,x,\theta,\gamma)p \\
       &+f(t,x,y,z+\sigma'(t,x,\theta,\gamma)p,\theta,\gamma)
                \bigg\},               
\end{align*}
for $(t,x,A,B,p,y,z)\in [0,T]\times\bR^d\times\bR^{d\times d}\times \bR^{m\times d} \times   \bR^d\times \bR \times \bR^m$, where the pairs of random fields $(V,\psi)$ and $(U,\zeta)$ are  unknown. 

The stochastic HJBI equations \eqref{SHJBI-low} and \eqref{SHJBI-up} are a new class of backward stochastic partial differential equations (BSPDEs) of which some special cases have been studied since  about forty years ago (see \cite{Pardoux1979}). Indeed, the linear, semilinear and even quasilinear BSPDEs have been extensively studied; we refer to \cite{Bayraktar-Qiu_2017,carmona2016mean,DuQiuTang10,Horst-Qiu-Zhang-14,Hu_Ma_Yong02,ma2012non,QiuTangMPBSPDE11,Tang-Wei-2013} among many others. In particular,
the so-called fully nonlinear stochastic Hamilton-Jacobi-Bellman (HJB) equations proposed by Peng \cite{Peng_92} for stochastic optimal control problem with controlled SDEs may be regarded as particular cases of our concerned stochastic HJBI equations like \eqref{SHJBI-low} and \eqref{SHJBI-up}; we refer to \cite{cardaliaguet2015master,Qiu2014weak,qiu2017viscosity,qiu2019uniqueness} for the recent study on the wellposedness of such fully nonlinear stochastic HJB equations which was claimed as  an open problem in Peng's plenary lecture of ICM 2010 (see \cite{peng2011backward}). However, the general  fully nonlinear stochastic HJBI equations have never been studied in the literature, mainly due to the full nonlinearity and non-convexity  of the Hamiltonian functions $\bH_{\pm}$ and the dependence of function $f$ on unknown variables.  

Inspired by the viscosity solutions for the fully nonlinear stochastic HJB equations (see \cite{qiu2017viscosity}), the concerned random fields $V$ and $U$ may be confined to the stochastic differential equations (SDEs) of the form:
\begin{align}
u(t,x)=u(T,x)-\int_{t}^T\mathfrak{d}_{s} u(s,x)\,ds-\int_t^T\mathfrak{d}_{w}u(s,x)\,dW_s,\quad (t,x)\in[0,T]\times\bR^d. \label{SDE-u}
\end{align}
The uniqueness of the pair $(\mathfrak{d}_tu,\,\mathfrak{d}_{\omega}u)$ may be concluded from the Doob-Meyer decomposition theorem, and this makes sense of the linear operators $\mathfrak{d}_t$ and $\mathfrak{d}_{\omega}$ which actually coincide with the differential operators discussed in \cite[Theorem 4.3]{Leao-etal-2018} and \cite[Section 5.2]{cont2013-founctional-aop}. Then the stochastic HJBI equations \eqref{SHJBI-low} and \eqref{SHJBI-up} may be written equivalently as
{\small
\begin{equation}\label{SHJB-low-eqiv}
  \left\{
  \begin{split}
  -\mathfrak{d}_tV(t,x)
- \mathbb{H}_-(t,x,D^2V(t,x),D\mathfrak{d}_{\omega}V(t,x),DV(t,x),V(t,x),\mathfrak{d}_{\omega}V(t,x))&=0,  \quad
                     (t,x)\in Q;\\
    V(T,x)&= \Phi(x), \quad x\in\bR^d,
    \end{split}
  \right.
\end{equation}
}
and 
{\small
\begin{equation}\label{SHJB-up-eqiv}
  \left\{
  \begin{split}
  -\mathfrak{d}_t U(t,x)
- \mathbb{H}_+(t,x,D^2U(t,x),D\mathfrak{d}_{\omega}U(t,x),DU(t,x),U(t,x),\mathfrak{d}_{\omega}U(t,x))&=0,  \quad
                     (t,x)\in Q;\\
    U(T,x)&= \Phi(x), \quad x\in\bR^d,
    \end{split}
  \right.
\end{equation}
}respectively. Solving \eqref{SHJBI-low}  and \eqref{SHJBI-up} for pairs $(V,\psi)$ and $(U,\zeta)$ is equivalent to searching for $V$ and $U$ (of form \eqref{SDE-u}) satisfying \eqref{SHJB-low-eqiv} and \eqref{SHJB-up-eqiv}, respectively. 

Due to the non-convexity  of the game and the nonlinear dependence of function $f$ on unknown variables, 
we define the viscosity solutions to \eqref{SHJBI-low}  and \eqref{SHJBI-up} with finer test functions than in \cite{qiu2017viscosity}, by introducing a class of sublinear functionals via BSDEs. In contrast with \cite{qiu2017viscosity}, some new techniques are used and more properties of  viscosity solutions are addressed, including particularly a stability result. The time-space continuity of the value function is proved via regular approximations to the coefficients and to prove the dynamic programming principle, the method of backward semigroups (see Peng \cite{Peng-DPP-1997}) is adopted. For the existence of the viscosity solution, the approach mixes some BSDE techniques and the proved dynamic programming principle, whereas for the uniqueness, we first prove a comparison result and then under additional assumptions on the controlled diffusion coefficients, the value function is verified (via approximations) to be the unique viscosity solution on the basis of the established comparison results. 

The zero-sum two-player stochastic differential games have been extensively studied. When all the involved coefficients are just deterministic functions of time, state and controls, the games are of Markovian type, the associated value functions are deterministic, and the dynamic programming methods as well as the viscosity solution theory of deterministic HJBI equations are widely used; see \cite{buckdahn2004nash,buckdahnLi-2008-SDG-HJBI,fleming1989existence,friedman1972stochastic} for instance. When the stochastic differential games are non-Markovian, there are two different formulations: one is based on path-dependence and the other one allows for general random coefficients. When the coefficients are deterministic functions of time $t$, controls $(\theta,\gamma)$ and \textit{the paths} of $X$ and $W$, the concerned games are beyond the classical Markovian framework. Nevertheless, if one thinks of the $X$ and $W$ as state processes valued in the path space, the Markovian property may be restored,  the value functions are deterministic and they may be characterized by  path-dependent PDEs on the (infinite-dimensional) path spaces. In this way, the viscosity solution theory of path-dependent PDEs developed in \cite{ekren2014viscosity,ekren2016viscosity-1} is generalized to study the stochastic differential games; see \cite{pham2014two,possamai2018zero,zhang2017existence} for instance. When the involved coefficients are generally random and may be just measurable w.r.t. $\omega\in \Omega$,  the games are typically non-Markovian. When the diffusion coefficients are uncontrolled, Elliott \cite{elliott1976existence} and Elliott and Davis \cite{elliott1981optimal} used the methods of Girsanov transformations to study the value functions and the Nash equilibrium, while Hamadene \cite{hamadene1998backward} and Hamadene, Lepeltier and Peng \cite{hamadene1997bsdes} established the stochastic maximum principle and used the forward-backward SDEs to study the open-loop Nash equilibrium points. Along this line, it becomes of great interests in this paper to study such non-Markovian games with controlled random diffusion coefficients, especially with the dynamic programming methods.

The rest of this paper is organized as follows. In the next section, we set the notations and assumptions. Section 3 is devoted to some regular properties of the value functions and the proof of the dynamic programming principle (DPP). In Section 4, we define the viscosity solution and prove a stability result. Subsequently, in Section 5, the value functions $V$ and $U$ are proved to be viscosity solutions of the associated stochastic HJBI equations with the help of the DPP and Peng's backward semigroups. Then the uniqueness of viscosity solution is discussed in Section 6. Finally, we recall a measurable selection theorem and some results on BSDEs in Appendix A and Appendix B, respectively. Appendix C collects the proofs of Lemmas \ref{lemdifferenceofY} and \ref{estimateYZ}.


\section{Preliminaries}

Denote by $|\cdot|$ the norm in  Euclidean spaces.
 Define the parabolic distance in $\bR^{1+d}$ as follows:
$$\delta(X,Y):=\max\{ |t-s|^{1/2},|x-y| \},$$
 for $X:=(t,x)$ and $Y:=(s,y)\in \bR^{1+d}$. Denote by $Q^+_r(X)$ the hemisphere of radius $r>0$ and center $X:=(t,x)\in \bR^{1+d}$ with $x\in \bR^d$:
\begin{equation*}
  \begin{split}
    Q^+_r(X):=\,  [t,t+r^2)\times B_r(x), \quad
    B_r(x):=\, \{ y\in\bR^n:|y-x|<r \},
  \end{split}
\end{equation*}
and by $|Q^+_r(X)|$ the volume. Throughout this paper, we write $(s,y)\rightarrow (t^+,x)$, meaning that $s \downarrow t$ and $y\rightarrow x$, and for a function $g$, $g^+=\max\{0,g\}$, while $g^-=\max\{0,-g\}$.

Let $\mathbb B $ be a Banach space equipped with norm $\|\cdot\|_{\mathbb B }$. For $p\in[1,\infty]$, $\cS ^p ({\mathbb B })$ is the set of all the ${\mathbb B }$-valued,
 $\sP$-measurable continuous processes $\{\mathcal X_{t}\}_{t\in [0,T]}$ such
 that
{\small$$\|\mathcal X\|_{\cS ^p({\mathbb B })}:= \left\|\sup_{t\in [0,T]} \|\mathcal X_t\|_{\mathbb B }\right\|_{L^p(\Omega,\sF,\bP)}< \infty.$$}
Denote by $\mathcal{L}^{p,2}({\mathbb B })$ the totality of all  the ${\mathbb B }$-valued,
  $\sP$-measurable processes $\{\mathcal X_{t}\}_{t\in [0,T]}$ such
 that
{\small$$\|\mathcal X\|_{\mathcal{L}^{p,2}({\mathbb B })}:=\left\| \left(\int_0^T \left\|\mathcal X_t\right\|^2_{\mathbb B}\,dt\right)^{1/2} \right\|_{L^p(\Omega,\sF_T,\bP)}< \infty.$$}
Meanwhile, by  $\mathcal{L}^{p}({\mathbb B })$ we denote the space of all  the ${\mathbb B }$-valued,
$\sP$-measurable processes $\{\mathcal X_{t}\}_{t\in [0,T]}$ such
 that
 $$
 \|\mathcal X\|_{\mathcal{L}^{p}({\mathbb B })}:=\left\|    \mathcal X  \right\|_{L^p(\Omega\times[0,T],\sP,\bP(d\omega)\otimes dt;\mathbb B)}< \infty.
 $$
 Obviously, $(\cS^p({\mathbb B }),\,\|\cdot\|_{\cS^p({\mathbb B })})$, $(\mathcal{L}^{p,2}({\mathbb B }),\|\cdot\|_{\mathcal{L}^{p,2}({\mathbb B })})$, and $(\mathcal{L}^p({\mathbb B }),\|\cdot\|_{\mathcal{L}^p({\mathbb B })})$
are Banach spaces. When the processes are defined on time intervals $[s,t]\subset [0,T]$, $s<t$, we define spaces $(\cS^p([s,t];{\mathbb B }),\,\|\cdot\|_{\cS^p([s,t];{\mathbb B })})$, $(\mathcal{L}^{p,2}([s,t]; {\mathbb B }),\|\cdot\|_{\mathcal{L}^{p,2}([s,t];{\mathbb B })})$, and $(\mathcal{L}^p([s,t]; {\mathbb B }),\|\cdot\|_{\mathcal{L}^p([s,t];{\mathbb B })})$ analogously. Throughout this paper, we define for $p\in [1,\infty]$, $\cL^p_{\text{loc}}([0,T);\mathbb B )=\cap_{T_0 \in (0,T)} \cL^{p}([0,T_0]; \mathbb B)$,
$$
\cS^p_{\text{loc}}([0,T);\mathbb B )=\cap_{T_0 \in (0,T)} \cS^{p}([0,T_0]; \mathbb B)\quad \text{and}\quad
\cL^{p,2}_{\text{loc}}([0,T);\mathbb B )=\cap_{T_0 \in (0,T)} \cL^{p,2}([0,T_0]; \mathbb B).
$$


For each $(k,l,q)\in \mathbb{N}_0\times \bN_+ \times [1,\infty]$,
 we define  the $k$-th Sobolev space $(H^{k,q}(\bR^l),\|\cdot\|_{k,q})$ as usual, and for
 each domain $\cO\subset \bR^l$, denote by $C_b^k(\cO)$ the space of functions with the up to $k$-th order derivatives being bounded and continuous on $\cO$;
 the space $C_b^k(\cO)$ is equipped with the norm $\|\cdot\|_{C_b^k(\cO)}$ as usual. For $\delta\in(0,1)$,  the H\"{o}lder space $C_b^{k+{\delta}}(\cO)$ is also defined as usual with the norm
 $$\|h\|_{C_b^{k+{\delta}}(\cO)}
 :=
 \|h\|_{C_b^k(\cO)} + \sup_{x,y\in\cO;x\neq y} \frac{|h(x)-h(y)|}{|x-y|^{\delta}}.
 $$
When $k=0$, write simply $C_b(\cO) $ and $C_b^{\delta}(\cO)$. 


\medskip

 We define the following assumption.

\bigskip \noindent
   $\textbf{({A}1)}$ (i) $\Phi\in L^{\infty}(\Omega,\sF_T;H^{1,\infty}(\mathbb{R}^d))$;
   \\
   (ii) for the coefficients $\tilde g=b^i,\sigma^{ij}, f(\cdot,\cdot,y,z,\cdot,\cdot)$, $1\leq i \leq d,\,1\leq j\leq m$, $(y,z)\in\bR\times\bR^m$, \\
 $\tilde g:~\Omega\times[0,T]\times\bR^d\times \Theta_0\times\Gamma_0 \rightarrow\bR$  {is}
$\sP\otimes\cB(\bR^d)\otimes\cB(\Theta_0)\otimes\cB(\Gamma_0)\text{-measurable}$
;\\
(iii) there exists $L>0$ such that 
$ \|\Phi\|_{L^{\infty}(\Omega,\sF_T;H^{1,\infty}(\mathbb{R}^d))}\leq L$ and for all $(\omega,t)\in\Omega\times [0,T]$ and any 
$$(x,y,z,\theta,\gamma), (\bar x,\bar y,\bar z, \bar \theta,\bar\gamma)\in \bR^d\times \bR\times\bR^m\times\Theta_0\times\Gamma_0,$$
there hold
\begin{align*}
|(b,\sigma)(t,x,\theta,\gamma)|+|f(t,x,y,z,\theta,\gamma)| \leq L,\\
\left|(b,\sigma)(t,x,\theta,\gamma)-(b,\sigma)\left(t,\bar x,\bar \theta,\bar \gamma\right)\right|
+\left| f(t,x,y,z,\theta,\gamma)-f\left(t,\bar x,\bar y,\bar z,\bar\theta,\bar \gamma\right)  \right| \quad\quad\\
\leq L\left(
	\left|  x-\bar x  \right|
	+\left|  y-\bar y  \right|
	+ \left|  z-\bar z  \right|
	+\left|  \theta-\bar \theta  \right|
	+\left|  \gamma-\bar \gamma  \right|
	\right).
\end{align*}

A standard application of density arguments yields the following approximations. 
\begin{lem}\label{lem-approx} 
 Let $\textbf{(A1)}$ hold. For each $\eps>0$, there exist partition $0=t_0<t_1<\cdots<t_{N-1}<t_N=T$ for some $N>3$ and functions  $\Phi_N\in  C_b^{3}(\bR^{m\times N+d})$,
$$b^i_N,\sigma^{ij}_N\in    C_b\left([0,T]\times \Theta_0\times\Gamma_0;C_b^3(\bR^{m \times N+ d} )\right),\quad 1\leq i\leq d,\  1\leq j\leq m,$$
and
 $$f_N\, \in    C_b\left([0,T]\times \Theta_0\times\Gamma_0;C_b^3(\bR^{m\times N+d+1+m} )\right),
 $$
such that, for $t\in [0,T]$,
\begin{align*}
&
	\Phi^{\eps}:=\esssup_{x\in\bR^d} \left|\Phi_N( W_{t_1},\cdots, W_{t_N},x)-\Phi(x) \right|,
\\
&
 	b^{\eps}_t :=\esssup_{(x,\theta,\gamma)\in\bR^d\times \Theta_0\times\Gamma_0}
		\left|b_N( W_{t_1\wedge t},\cdots, W_{t_N\wedge t},t,x,\theta,\gamma)-b(t,x,\theta,\gamma)\right|,
\\
&\sigma^{\eps}_t :=\esssup_{(x,\theta,\gamma)\in\bR^d\times \Theta_0\times\Gamma_0}
		\left|\sigma_N( W_{t_1\wedge t},\cdots, W_{t_N\wedge t},t,x,\theta,\gamma)-\sigma(t,x,\theta,\gamma)\right|,
\\
& 
	f^{\eps}_t :=\esssup_{(x,y,z,\theta,\gamma)\in\bR^d\times \bR\times \bR^m \times \Theta_0\times\Gamma_0}
		\left|f_N( W_{t_1\wedge t},\cdots, W_{t_N\wedge t},t,x,y,z,\theta,\gamma)-f(t,x,y,z,\theta,\gamma)\right|
\end{align*}
 are $\{\sF_t\}_{t\geq0}$-adapted with
\begin{align}
	\left\| \Phi^{\eps}  \right\|_{L^4(\Omega, \sF_T;\bR)} + \left\| f^{\eps}  \right\|_{\cL^4(\bR)}   
		+ \left\| b^{\eps}  \right\|_{\cL^4(\bR^d)}    
		+\left\| \sigma^{\eps}  \right\|_{\cL^4(\bR^{d\times m})}    
			&<\eps, \label{approx-coeficients}
\end{align}
and the sequence of continuous functions $\{\Phi_N, b_N,\sigma_N,f_N\}_{N\geq 4}$ are uniformly bounded by $L+1$ where the constant $L$ is from $\textbf{(A1)}$, and they are also uniformly Lipschitz-continuous in the variables $(x,y,z)$ with an identical Lipschitz-constant $L_c$  being independent of $N$ and $\eps$.
\end{lem}
The proof of Lemma \ref{lem-approx}  is so similar to that of \cite[Lemma 4.5]{qiu2019uniqueness} that it is omitted. 

\bigskip
Then, we recall the identity approximation and use it to define some smooth functions. Let
 \begin{equation}\label{bump-func}
 \rho(x)=
 \begin{cases}
 \tilde{c} \,  e^{\frac{1}{|x|^2-1}}&\quad \text{if } |x| < 1;\\
 0&\quad \text{otherwise};
 \end{cases}
 \quad \mbox{with}\quad \tilde{c}:=\left(     \int_{|x| < 1} e^{\frac{1}{x^2-1}}\,dx    \right)^{-1},
 \end{equation}
and	 set
\begin{align}
g(x)&=    \int_{\bR^d}  \int_{\bR^d} 1_{\{|y|>2 \}} \left(  |y|-2\right) \rho(z-y)\rho(x-z)\,dydz,\quad x\in\bR^d; \label{g-defined} \\
\phi_{\delta}(x)
&= \int_{\bR^d}  \phi(z) \rho\left(\frac{x-z}{\delta}\right)\cdot \frac{1}{\delta^d}\,dz,\quad (x, \delta)\in\bR^d\times (0,\infty), \quad \phi \in C_b(\bR^d). \label{I-approx}
\end{align}
%
Then the function $g$ is convex and continuously infinitely differentiable with $g({0})=0$, $g(x)>0$ whenever $|x|>0$, and there exists a constant $\alpha_0$ such that
\begin{align}
g(x)> |x|-3,\quad|Dg(x) |+ |D^2g(x)|\leq \alpha_0,\ \forall\,\, x\in\bR^d.\label{h-linear-growth}
\end{align}
The theory of identity approximations indicates that for each $\delta >0$, the function $\phi_{\delta}$ is smooth with $\|\phi_{\delta}\|_{C_b(\bR^d)} \leq \|\phi\|_{C_b(\bR^d)}$ and $\phi_{\delta}$ converges to $\phi$ uniformly on any compact subset of $\bR^d$.

Finally, in what follows, $C>0$ is a constant whose value may vary from line to line and by $C(a_1,a_2,\cdots)$ we denote a constant depending on the parameters $a_1$, $a_2$, $\dots$.



\section{Some properties of the value function and dynamic programming principle}

We first recall some standard properties of the strong solutions for SDEs (see \cite[Theorems 6.3 \& 6.16]{yong-zhou1999stochastic} for instance). 
\begin{lem}\label{lem-SDE}
Let $\textbf{(A1)}$ hold. Given $(\theta,\gamma)\in\Theta\times\Gamma$, for the strong solution of SDE \eqref{state-proces-control}, there exists $C>0$  such that, for any $0\leq r \leq t\leq s \leq T$ and $\xi\in L^p(\Omega,\sF_r;\bR^d)$ with $p\in(1,\infty)$,\\[3pt]
(i)   the two processes $\left(X_s^{r,\xi;\theta,\gamma}\right)_{t\leq s \leq T}$ and $\left(X^{t,X_t^{r,\xi;\theta,\gamma};\theta,\gamma}_s\right)_{t\leq s\leq T}$ are indistinguishable;\\[2pt]
(ii)  $E_{\sF_r} \left[ \max_{r\leq l \leq T} \left|X^{r,\xi;\theta,\gamma}_l\right|^p \right] \leq C \left(1+ |\xi|^p\right),\ \ \text{a.s.}$;\\[2pt]
(iii) $E_{\sF_r}\left[ \left| X^{r,\xi;\theta,\gamma}_s-X^{r,\xi;\theta,\gamma}_t  \right|^p \right]
	\leq C \left(1+  |\xi|^p\right) (s-t)^{p/2},\ \ \text{a.s.}$;\\[2pt]
(iv) given another $\hat{\xi}\in L^p(\Omega,\sF_r;\bR^d)$, 
$$
E_{\sF_r} \left[ \max_{r\leq l \leq T} \left|X^{r,\xi;\theta,\gamma}_l-X^{r,\hat\xi;\theta,\gamma}_l\right|^p \right]
	\leq C |\xi-\hat\xi|^p,
\quad \text{a.s.};
$$
(v) the constant $C$ depends only on $L$, $T$, and $p$.
\end{lem}

The following assertions on BSDEs are standard; the readers are referred to \cite{El-Karoui-Peng-Quenez-2001,ParPeng_90,Peng-DPP-1997}.

\begin{lem}\label{lem-BSDE}
Let $\textbf{(A1)}$ hold. It holds that\\
(i) for each $(t,\theta,\gamma)\in [0,T)\times \Theta\times \Gamma$ and any $\xi\in L^2(\Omega,\sF_t;\bR^d)$, BSDE \eqref{BSDE} admits a unique solution 
$(Y^{t,\xi;\theta,\gamma},Z^{t,\xi;\theta,\gamma})\in\cS^2([t,T];\bR)\times\cL^2([t,T];\bR^m)$ and the solution satisfies
\begin{align*}
E_{\sF_t}\left[
\sup_{s\in[t,T]} \left| Y_s^{t,\xi;\theta,\gamma}
	\right|^2	
	+\int_t^T
	\left| Z^{t,\xi;\theta,\gamma}_s\right|^2 ds
	\right] 
&\leq 
	C\left(  1+|\xi|^2  \right),\quad \text{a.s.,}\\
\left|  
Y_t^{t,\xi;\theta,\gamma}
\right| & \leq L\left(1+T \right), \quad \text{a.s.;}
\end{align*}
(ii) given another $\bar \xi\in L^2(\Omega,\sF_t;\bR^d)$, we have 
\begin{align*}
E_{\sF_t}\left[
\sup_{s\in[t,T]} \left| Y_s^{t,\xi;\theta,\gamma}
	 -Y_s^{t,\bar\xi;\theta,\gamma}
	\right|^2	
	+\int_t^T
	\left| Z^{t,\xi;\theta,\gamma}_s
		 -Z_s^{t,\bar \xi;\theta,\gamma}
		 \right|^2 ds
	\right] 
&\leq 
	C \left|\xi-\bar \xi\right|^2,\quad \text{a.s.,}\\
	\left| Y_t^{t,\xi;\theta,\gamma}  - Y_t^{t,\bar \xi;\theta,\gamma} 
	\right| 
& \leq C \left|\xi -\bar \xi \right|, \, \quad \text{a.s.;}
\end{align*}
(iii) the constant $C$ depends only on $L$ and $T$.
\end{lem}


An immediate consequence of Lemma \ref{lem-BSDE} is the following spacial regularity of the value functions $V$ and $U$.

\begin{lem}\label{reg-value-funct}
Let $\textbf{(A1)}$ hold. We have
 $$ \esssup_{(t,x)\in[0,T]\times\bR^d}  \esssup_{(\theta,\gamma)\in \Theta\times\Gamma } \max\big\{|V(t,x)|,\,|U(t,x)|,\,|J(t,x;\theta,\gamma)| \big\}  \leq L(T+1), \quad\text{a.s.}$$
And there exists a constant $L_0>0$ such that, for all $0\leq t \leq T$, $x,\bar x\in\bR^d$,
$$
|U(t,x)-U(t,\bar x)| + |V(t,x)-V(t,\bar x) | + \esssup_{(\theta,\gamma)\in \Theta \times \Gamma} |J(t,x;\theta,\gamma)  -J(t,\bar x;\theta,\gamma)| 
\leq L_0 |x-\bar x|, \,\, \text{a.s.}
$$
\end{lem}

For the dynamic functional in \eqref{dynamic-cost}, the following lemma is an immediate application of \cite[Theorem 4.7 and Lemma 6.5]{Peng-DPP-1997}, for which a concise proof can also be found in \cite[Theorem A.2]{buckdahnLi-2008-SDG-HJBI}.
\begin{lem}\label{lem-J}
Let $\textbf{(A1)}$ hold.  For any $(t,\theta,\gamma)\in[0,T]\times\Theta\times\Gamma$ and any measurable random variable $\xi\in L^2(\Omega,\sF_t;\bR^d)$, we have 
$$
J(t,\xi;\theta,\gamma)=Y_t^{t,\xi;\theta,\gamma}, \quad \text{ a.s.}
$$
\end{lem}

In order to discuss the dynamic programming principle for our stochastic differential game, we introduce the family of (backward) semigroups associated with BSDE \eqref{BSDE} (see Peng \cite{Peng-DPP-1997}). Given $0\leq t\leq t+\delta \leq T$, $(\theta,\gamma)\in \Theta\times \Gamma$, and $\eta\in L^2(\Omega,\sF_{t+\delta};\bR)$, we set
\begin{align}\label{semigroup}
G_{s,t+\delta}^{t,x;\theta,\gamma}[\eta]:=\overline Y_s^{t,x;\theta,\gamma},\quad s\in [t,t+\delta],
\end{align}
where $\overline Y^{t,x;\theta,\gamma}$ together with $\overline Z^{t,x;\theta,\gamma}$ satisfies the following BSDE:
\begin{equation}\label{BSDE-semigroup}
  \left\{
  \begin{split}
  -d\overline Y_s^{t,x;\theta,\gamma}&=\,f(s,X_s^{t,x;\theta,\gamma},\overline Y_s^{t,x;\theta,\gamma},\overline Z_{s}^{t,x;\theta,\gamma},\theta_s,\gamma_s)\,ds
  -\overline Z_{s}^{t,x;\theta,\gamma}\,dW_s,\quad s\in[t, t+\delta];\\
 \overline Y_{t+\delta}^{t,x;\theta,\gamma}&=\eta,
    \end{split}
  \right.
\end{equation}
with $X^{t,x;\theta,\gamma}$ being the solution to SDE \eqref{state-proces-control}.

Obviously, for the solution $(Y^{t,x;\theta,\gamma},\,  Z^{t,x;\theta,\gamma})$, one has 
$$
G_{t,T}^{t,x;\theta,\gamma} \left[ \Phi(X_T^{t,x;\theta,\gamma})  \right]
	= G_{t,t+\delta}^{t,x;\theta,\gamma} \left[  Y_{t+\delta}^{t,x;\theta,\gamma} \right],\quad \text{a.s.},
$$
and thus
\begin{align*}
J(t,x;\theta,\gamma)
&
	=Y_t^{t,x;\theta,\gamma}
		=G_{t,T}^{t,x;\theta,\gamma} \left[ \Phi(X_T^{t,x;\theta,\gamma})  \right]
		=G_{t,t+\delta}^{t,x;\theta,\gamma} \left[  Y_{t+\delta}^{t,x;\theta,\gamma} \right]\\
&
	=G_{t,t+\delta}^{t,x;\theta,\gamma} \left[ J(t+\delta,X_{t+\delta}^{t,x;\theta,\gamma};\theta,\gamma) \right], \quad \text{a.s.}
\end{align*}

We first prove the continuity of the value functions.
\begin{thm}\label{thm-cont} 
Let $\textbf{(A1)}$ hold. With probability 1, $V(t,x)$, $U(t,x)$ and $J(t,x;\theta,\gamma)$ are continuous  on $[0,T]\times\bR^d$ for each $(\theta,\gamma)\in \Theta\times \Gamma$.
\end{thm}
\begin{proof}
We need only to prove the continuity of $V(t,x)$, as it will follow analogously for $U(t,x)$ and $J(t,x;\theta,\gamma)$. Due to the uniform Lipschitz-continuity of $V(t,x)$ in $x$ in Lemma \ref{reg-value-funct}, it is sufficient to prove the time-continuity. We fix an $x\in\bR^d$ in what follows.

For each fixed $\eps\in(0,1)$, select $(\Phi^{\eps},\,f^{\eps},\,b^{\eps},\sigma^{\eps})$ and $(\Phi_N,f_N,b_N,\sigma_N)$ as in Lemma \ref{lem-approx}.  Then by analogy with \eqref{semigroup}, we define the family of (backward) semigroups $G_{s,t+\delta}^{N,t,x;\theta,\gamma}[\cdot] $ associated with the generator $f_N \left( W_{t_1\wedge t},\cdots, W_{t_N\wedge t},t,x,y,z,\theta_t,\gamma_t\right) $. 
For each $(s,x)\in[0,T)\times\bR^d$,  set
\begin{align*}
{V}^{\eps}(s,x)
&=\essinf_{\mu\in\cM}\esssup_{\theta\in\Theta} G_{s,T}^{N,s,x;\theta,\mu({\theta})}
\left[\Phi_N \left( W_{t_1},\cdots,  W_{t_N},X^{s,x;\theta,\mu(\theta),N}_T\right) \right],
\end{align*}
where
$X^{s,x;\theta,\mu(\theta),N}_t$ satisfies SDE
\begin{equation*}\label{state-proces-contrl}
\left\{
\begin{split}
&dX_t^{s,x;\theta,\mu(\theta),N}=b_N( W_{t_1\wedge t},\cdots, W_{t_N\wedge t},t,X_t^{s,x;\theta,\mu(\theta),N},\theta_t,\mu(\theta)(t))dt  \\&\hspace{3cm} +\sigma_N (W_{t_1\wedge t},\cdots, W_{t_N\wedge t},t,X_t^{s,x;\theta,\mu(\theta),N},\theta_t,\mu(\theta)(t))\,dW_t; \\
& X_s^{s,x;\theta,\mu(\theta),N}=x.
\end{split}
\right.
\end{equation*}
The theory of stochastic differential games  (see \cite{buckdahnLi-2008-SDG-HJBI}) yields that when $s\in[t_{N-1},T]$, 
$$V^{\eps}(s,x)=\tilde V^{\eps}(s,x, W_{t_1},\cdots,  W_{t_{N-1}}, W_s)$$ with
\begin{align*}
\tilde{V}^{\eps}(s,x,  W_{t_1},\cdots,  W_{t_{N-1}},y)
:=&\essinf_{\mu\in\cM}\esssup_{\theta\in\Theta}  
G_{s,T}^{N,s,x;\theta,\mu({\theta})} \left[ \Phi_N\left(  W_{t_1},\cdots,  W_{t_N},X^{s,x;\theta,\mu(\theta),N}_T\right)\right] \Big|_{W_s=y},
\end{align*}
which is the lower value function of a stochastic differential game of Markovian type and thus is time-continuous (see \cite[Theorem 3.10]{buckdahnLi-2008-SDG-HJBI}). Analogously, we may obtain the time-continuity over other time intervals, $[t_{N-2},t_{N-1}]$, $\dots$, $[0,t_1]$, and it is easy to check that $V^{\eps}(\cdot,x)\in \cS^4(\bR)$.
 
  In view of the approximations in Lemma \ref{lem-approx}, using It\^o's formula, Burkholder-Davis-Gundy's inequality and Gronwall's inequality, we have through standard computations that for each $(\theta,\mu)\in\Theta\times\cM$,
\begin{align*}
&
	E_{\sF_s} \left [  \sup_{s\leq t\leq T} \left|  X^{s,x;\theta,\mu(\theta),N}_t-X^{s,x;\theta,\mu(\theta)}_t   \right| ^2  \right]
\leq
\tilde C E_{\sF_s}\left[ \int_s^T  \left(   \left|b^{\eps}_t\right|^2
		+  
		\left|\sigma^{\eps}_t\right|^2 \right)\,dt \right],
\end{align*}
with $\tilde C$ being independent of $N$, $\eps$, and $(\theta,\mu)$. As a consequence of the uniform-Lipchitz continuity of coefficients $f_N$ and $\Phi_N$, we have through standard estimates for BSDEs
\begin{align*}
&\left|V^{\eps}(s,x)-V(s,x)\right|^2    \\
&\leq C\esssup_{(\theta,\mu)\in\Theta\times\cM}E_{\sF_s}\bigg[ \int_s^T\Big( |f^{\eps}_t|^2 
+ L_c^2\Big|    X^{s,x;\theta,\mu(\theta),N}_t - X^{s,x;\theta,\mu(\theta)}_t    \Big|^2 \Big)\,dt
\\
&\quad\quad
+|\Phi^{\eps}|^2 +L_c^2\Big|  X^{s,x;\theta,\mu(\theta),N}_T  - X^{s,x;\theta,\mu(\theta)}_T\Big|^2\bigg] \\
&\leq 
C_0\,E_{\sF_s}\left[ \left|\Phi^{\eps}\right|^2 +  \int_0^T\left(\left| f^{\eps}_t\right|^2 + \left| b^{\eps}_t\right|^2+\left|\sigma^{\eps}_t\right|^2 \right)  \,dt   \right],
\end{align*}
with the constant $C_0$ being independent of $N$, $\eps$, and $(s,x)$. Taking supremum with respect to $s$ and then expectation on both sides, we obtain
\begin{align*}
\left\| V^{\eps}(\cdot,x)-V(\cdot,x)\right\|^4_{\cS^4(\bR)}
&
	\leq  C\,
		E\left[ \sup_{s\in[0,T]}
		\left(E_{\sF_s}\left[ \left|\Phi^{\eps}\right|^2 +  \int_0^T\left(\left| f^{\eps}_t\right|^2 + \left| b^{\eps}_t\right|^2+\left|\sigma^{\eps}_t\right|^2 \right)  \,dt   \right]\right)^2  \right] \\
\text{(by Doob's inequality)}
&
	\leq  C\,E\left[
	\left|\Phi^{\eps}\right|^4 +  \int_0^T\left(\left| f^{\eps}_t\right|^4 + \left| b^{\eps}_t\right|^4+\left|\sigma^{\eps}_t\right|^4 \right)  \,dt   
	\right]
  \\
  &\leq
 C\eps^4 \rightarrow 0, \quad\text{as } \eps\rightarrow 0.
\end{align*}

Hence, $V(\cdot,x)\in\cS^4(\bR)$. In particular, $V(t,x)$ is continuous in $t$. We complete the proof.
\end{proof}

We now turn to present our dynamic programming principle.
\begin{thm}\label{thm-DPP}
Let $\textbf{(A1)}$ hold. For any stopping times $\tau$ and $\hat\tau$ with $\tau\leq \hat\tau\leq T$, and any $ \xi\in L^2(\Omega,\sF_{\tau};\bR^d)$, we have
\begin{align}
V(\tau,\xi)&
	=\essinf_{\mu\in\cM}\esssup_{\theta\in\Theta} G_{\tau,\hat \tau}^{\tau,\xi;\theta,\mu(\theta)} \left[  V(\hat \tau,X_{\hat\tau}^{\tau,\xi;\theta,\mu(\theta)})  \right],
 		\quad \text{a.s.,} \label{DPP-low} \\
U(\tau,\xi)&
	=\esssup_{\alpha\in\cA}\essinf_{\gamma\in\Gamma} G_{\tau,\hat \tau}^{\tau,\xi;\alpha(\gamma),\gamma} \left[  U(\hat \tau,X_{\hat\tau}^{\tau,\xi;\alpha(\gamma),\gamma})  \right],
 		\quad \text{a.s.}\label{DPP-up} 
\end{align}
\end{thm}
\begin{proof}
We prove \eqref{DPP-low} and then \eqref{DPP-up} follows analogously.  Denote the right hand side of \eqref{DPP-low} by $\overline V(\tau,\xi)$.

\textbf{Step 1.} We prove $\overline V (\tau,\xi)\leq V(\tau,\xi)$, a.s. Fix an arbitrary $\mu\in\cM$. Then
\begin{align*}
\overline V(\tau,\xi)
	\leq \esssup_{\theta\in\Theta} 
			G_{\tau,\hat \tau}^{\tau,\xi;\theta,\mu(\theta)} \left[ V(\hat \tau,X^{\tau,\xi;\theta,\mu(\theta)}_{\hat \tau}) \right], 
			\quad \text{a.s.}
\end{align*}
Put
$$
\hat I (\tau,\xi;\theta,\gamma)
=
	G_{\tau,\hat \tau}^{\tau,\xi;\theta,\gamma} \left[ V(\hat \tau,X^{\tau,\xi;\theta,\gamma}_{\hat \tau}) \right].
$$
Notice that $\hat I (\tau,\xi;\theta,\gamma)$ depends only on the values of controls $\theta$  and $\gamma$ on $[\tau,\hat\tau]$. There exists a sequence $\{  \theta^i;\,i\geq 1 \}\subset \Theta$ such that $\theta^i_s=\theta^1_s$ a.s. for any $0\leq s < \tau$, $i\geq 1$, and
\begin{align}
\hat I(\tau,\xi;\mu): 
	& = \esssup_{\theta\in\Theta} \hat I (\tau,\xi;\theta,\mu(\theta))\nonumber    \\
	& =\sup_{i\geq 1} \hat I (\tau,\xi;\theta^i,\mu(\theta^i)),\quad \text{a.s.}  \label{DPP-neq-1}
\end{align}

For any $\eps\in(0,1)$, set $\tilde{\Omega}^i:=\{\hat I (\tau,\xi;\mu) \leq \hat I(\tau,\xi;\theta^i,\mu(\theta^i)) +\varepsilon \}\in\sF_{\tau}$ for $i=1,2,\dots,$. Then $\mathbb P(\cup_{i=1}^{\infty} \tilde \Omega^i)=1$.  Let 
$\Omega^1=\tilde\Omega^1$, $\Omega^j=\tilde\Omega^{j}\backslash \cup_{i=1}^{j-1} \Omega^i\in\sF_{\tau}$, $j\geq 2$. Then $\{\Omega^j;\,j\geq 1\}$ is an $(\Omega,\sF_{\tau})$-partition, and 
\begin{align}
\theta^{\eps}:=\theta^1 1_{[0,\tau)} + \sum_{j \geq 1} 1_{\Omega^j} \theta^j 1_{[\tau,T]}  
\end{align} 
is an admissible control in $\Theta$. Furthermore, the nonanticipativity of $\mu$ implies that $\mu(\theta^{\eps})=\mu(\theta^1)1_{[0,\tau)} + \sum_{j\geq 1} 1_{\Omega^j} \mu (\theta^j) 1_{[\tau,T]} $, and in view of the uniqueness of the solution to forward-backward SDE, we have $\hat I(\tau,\xi;\theta^{\eps},\mu(\theta^{\eps})) =\sum_{i\geq 1} 1_{\Omega^i} \hat I (\tau,\xi;\theta^{i},\mu(\theta^{i}))$ a.s. and thus,
\begin{align}
\overline V(\tau,\xi)
\leq \hat I(\tau,\xi;\mu)
	&\leq \sum_{i\geq 1}1_{\Omega^i} \hat I (\tau,\xi;\theta^{i},\mu(\theta^{i})) + \eps 
		= \hat I(\tau,\xi;\theta^{\eps},\mu(\theta^{\eps})) +\eps       \nonumber\\
	&=G_{\tau,\hat \tau}^{\tau,\xi;\theta^{\eps},\mu(\theta^{\eps}) } 
		\left[ V(\hat \tau,X^{\tau,\xi;\theta^{\eps},\mu(\theta^{\eps})}_{\hat \tau}) \right] + \eps,\quad \text{a.s.} \label{DPP-neq-2}
\end{align}

In a similar way to \eqref{DPP-neq-1}-\eqref{DPP-neq-2}, we construct the control $\hat{\theta}^{\eps}\in\Theta$ such that $\hat{\theta}^{\eps}_s =\theta^{\eps}_s $ for $0\leq s\leq \hat{\tau}$ and
\begin{align}
V(\hat{\tau}, X^{\tau,\xi;\theta^{\eps},\mu(\theta^{\eps})} )
	\leq Y^{\tau,\xi;\hat\theta^{\eps},  \mu (\hat\theta^{\eps})}_{\hat\tau} + \eps,\quad \text{a.s.}
	\label{DPP-neq-3}
\end{align}

From relations \eqref{DPP-neq-2} and \eqref{DPP-neq-3} and (iii) of Proposition \ref{prop-BSDE-comp}, it follows that
\begin{align}
\overline V(\tau,\xi)
	&\leq G_{\tau,\hat \tau}^{\tau,\xi;\theta^{\eps},\mu(\theta^{\eps}) } 
		\left[  Y^{\tau,\xi;\hat\theta^{\eps},  \mu (\hat\theta^{\eps})}_{\hat\tau} + \eps \right]  +\eps
	\nonumber \\
	&\leq 
		G_{\tau,\hat \tau}^{\tau,\xi;\theta^{\eps},\mu(\theta^{\eps}) } 
			\left[  Y^{\tau,\xi;\hat\theta^{\eps},  \mu (\hat\theta^{\eps})}_{\hat\tau}  \right]  + \left(  1+C_0 \right)   \eps 
	\nonumber \\
	&= G_{\tau,\hat \tau}^{\tau,\xi;\hat \theta^{\eps},\mu(\hat \theta^{\eps}) } 
			\left[  Y^{\tau,\xi;\hat\theta^{\eps},  \mu (\hat\theta^{\eps})}_{\hat\tau}  \right]  + \left(  1+C_0 \right)   \eps 
	\nonumber \\
	 &=Y^{\tau,\xi;\hat\theta^{\eps},  \mu (\hat\theta^{\eps})}_{\tau}  
	 			+ 	 \left(  1+C_0 \right)   \eps \label{DPP-neq-R},
\end{align}
which together with the arbitrariness of $(\mu,\eps)$ implies  $\overline V (\tau,\xi) \leq V(\tau,\xi)$, a.s.

\textbf{Step 2.} We prove $ V (\tau,\xi) \leq \overline V(\tau,\xi)$, a.s. The methodology is analogous to that in \textbf{Step 1.} 
Notice that 
$$
\overline V(\tau,\xi)=\essinf_{\mu\in\cM} \hat I(\tau,\xi;\mu),
$$
and that  $ I (\tau,\xi;\theta,\gamma)$ depends only on the values of controls $\theta$  and $\gamma$ on $[\tau,\hat\tau]$. There exists a sequence $\{\mu^i;\,i\geq 1\}\subset \cM$ such  that $\mu^i_s=\mu^1_s$ a.s. for any $0\leq s < \tau$, $i\geq 1$, and
\begin{align}
\overline V(\tau,\xi) 
	& =\inf_{i\geq 1} \hat I (\tau,\xi;\mu^i),\quad \text{a.s.}  \label{DPP-neq-1-S2}
\end{align}

For any $\eps\in(0,1)$, set $\tilde{\Lambda}^i:=\{\hat I (\tau,\xi;\mu) -\eps \leq \overline V(\tau,\xi) \}  \in\sF_{\tau}$, for $i=1,2,\dots$, and put
$\Lambda^1=\tilde\Lambda^1$, $\Lambda^j=\tilde\Lambda^{j-1}\backslash \cup_{i=1}^{j-1} \Lambda^i\in\sF_{\tau}$, $j\geq 2$. Then $\{\Lambda^j;\,j\geq 1\}$ is an $(\Omega,\sF_{\tau})$-partition, and 
\begin{align}
\mu^{\eps}:=\mu^1 1_{[0,\tau)} + \sum_{j \geq 1} 1_{\Lambda^j} \mu^j 1_{[\tau,T]}  
\end{align} 
belongs to $\cM$. The uniqueness of the solution to forward-backward SDE further yields, for all $\theta\in\Theta$, $\hat I(\tau,\xi;\theta,\mu^{\eps}(\theta)) =\sum_{i\geq 1} 1_{\Lambda^i} \hat I (\tau,\xi;\theta,\mu^{i}(\theta))$ a.s., and thus,
\begin{align}
\overline V(\tau,\xi)
	&\geq \sum_{i\geq 1}1_{\Lambda^i} \hat I (\tau,\xi;\mu^i) - \eps  \nonumber\\
	&\geq \sum_{i\geq 1}1_{\Lambda^i} \hat I (\tau,\xi;\theta,\mu^i(\theta)) - \eps  \nonumber\\
	&=G_{\tau,\hat \tau}^{\tau,\xi;\theta,\mu^{\eps}(\theta) } 
		\left[ V(\hat \tau,X^{\tau,\xi;\theta,\mu^{\eps}(\theta)}_{\hat \tau}) \right] -\eps,\quad \text{a.s. for all }\theta\in\Theta. \label{DPP-neq-2-S2}
\end{align}
Analogously to \eqref{DPP-neq-1-S2}-\eqref{DPP-neq-2-S2}, 
we may construct $\hat{\mu}^{\eps}\in\cM$ such that $\hat{\mu}^{\eps}_s =\mu^{\eps}_s $ for $0\leq s\leq \hat{\tau}$ and
\begin{align}
V(\hat{\tau}, X^{\tau,\xi;\theta,\mu^{\eps}(\theta)} )
	\geq Y^{\tau,\xi;\theta,  \hat\mu^{\eps} (\theta)}_{\hat\tau} - \eps,\quad \text{a.s.}
	\label{DPP-neq-3-S2}
\end{align}
From relations \eqref{DPP-neq-2-S2} and \eqref{DPP-neq-3-S2} and (iii) of Proposition \ref{prop-BSDE-comp}, it follows that
\begin{align}
\overline V(\tau,\xi)
	&\geq G_{\tau,\hat \tau}^{\tau,\xi;\theta,\mu^{\eps}(\theta) } 
		\left[  Y^{\tau,\xi;\theta,  \hat \mu^{\eps} (\theta)}_{\hat\tau} - \eps \right]  -\eps
	\nonumber \\
	&\geq 
		G_{\tau,\hat \tau}^{\tau,\xi;\theta,\mu^{\eps}(\theta) } 
		\left[  Y^{\tau,\xi;\theta,  \hat \mu^{\eps} (\theta)}_{\hat\tau}  \right] - \left(  1+C_0 \right)   \eps 
	\nonumber \\
	&= G_{\tau,\hat \tau}^{\tau,\xi;\theta,\hat\mu^{\eps}(\theta) } 
			\left[  Y^{\tau,\xi;\theta,  \hat \mu^{\eps} (\theta)}_{\hat\tau}  \right] - \left(  1+C_0 \right)   \eps 
	\nonumber \\
	 &=Y^{\tau,\xi;\theta,  \hat\mu^{\eps} (\theta)}_{\tau}  
	 				 -\left(  1+C_0 \right)  \eps, \quad \text{a.s. for all }\theta\in\Theta , \label{DPP-neq-L}
\end{align}
which together with the arbitrariness of $(\theta,\eps)$ implies $\overline V (\tau,\xi) \geq V(\tau,\xi)$, a.s.
\end{proof}

\begin{rmk}\label{rmk-DPP}
For any $\eps\in (0,1)$, we have actually constructed in the above proof a pair $(\hat\theta^{\eps},\hat\mu^{\eps})\in \Theta\times \cM$ such that
\begin{align*}
Y^{\tau,\xi; \hat\theta^{\eps},  \hat\mu^{\eps} (\hat\theta^{\eps})}_{\tau}  
	 			- \left(  1+C_0 \right)   \eps 
				&\leq 
					V(\tau,\xi)
				\leq 
				Y^{\tau,\xi;\hat\theta^{\eps},  \hat\mu^{\eps} (\hat\theta^{\eps})}_{\tau}  
	 			+\left(  1+C_0 \right)   \eps, \quad \text{a.s.,}
				\\
Y^{\tau,\xi; \hat\theta^{\eps},  \hat\mu^{\eps} (\hat\theta^{\eps})}_{\hat\tau}  
	 			-    \eps 
				&\leq 
					V(\hat\tau,X^{\tau,\xi;\hat\theta^{\eps},  \hat\mu^{\eps} (\hat\theta^{\eps})}_{\hat\tau})
				\leq 
				Y^{\tau,\xi;\hat\theta^{\eps},  \hat\mu^{\eps} (\hat\theta^{\eps})}_{\hat\tau}  
	 			+  \eps, \quad \text{a.s.,}
\end{align*}
where the constant $C_0$ depends only on $L$ and $T$. 
\end{rmk}


\section{Definition of viscosity solutions and a stability result}\label{sec:def}

For each stopping time $t\leq T$, denote by $\mathcal{T}^t$ the set of stopping times $\tau$ valued in $[t,T]$ and by $\mathcal{T}^t_+$ the subset of $\mathcal{T}^t$ such that $\tau>t$ for any $\tau\in \mathcal{T}^t_+$. For each $\tau\in\mathcal T^0$ and $\Omega_{\tau}\in\sF_{\tau}$, we denote by $L^0(\Omega_{\tau},\sF_{\tau};\bR^d)$ the set of $\bR^d$-valued $\sF_{\tau}$-measurable functions on $\Omega_{\tau}$. 

\subsection{A class of sublinear functionals}
A class of sublinear functionals may be defined through BSDEs (see \cite{peng2011backward} for instance). For each $K\geq 0$ and two stopping times $\tau_1,\tau_2\in \mathcal{T}^0$ with $\tau_1\leq \tau_2$ a.s., we define
$$
\overline\cE^K_{\tau_1,\tau_2}[\xi] =\overline Y_{\tau_1},\quad 
\underline\cE^K_{\tau_1,\tau_2}[\xi] =\underline Y_{\tau_1}, \quad
\text{for }\xi\in L^2(\Omega,\sF_{\tau_2},\bP),
$$  
where $(\overline Y_{s})_{\tau_1\leq s\leq \tau_2}$ and $(\underline Y_{s})_{\tau_1\leq s\leq \tau_2}$, together with another pair of processes
$(\overline Z_{s})_{\tau_1\leq s\leq \tau_2}$ and $(\underline Z_{s})_{\tau_1\leq s\leq \tau_2}$, are solutions to the following BSDEs:
 \begin{align*}
 \overline Y_{t}&=\xi+\int_{t}^{\tau_2}  K\left(\left|\overline Y_s\right|+ \left|\overline Z_s\right|\right)\, ds 
 -\int_{t}^{\tau_2}\overline Z_s dW_s,\quad \tau_1\leq t\leq \tau_2,\\
  \underline Y_{t}&=\xi-\int_{t}^{\tau_2}  K\left( \left|\underline Y_s\right| + \left|\underline Z_s\right|\right)\, ds 
 -\int_{t}^{\tau_2}\underline Z_s dW_s,\quad \tau_1\leq t\leq \tau_2,
 \end{align*}
respectively.

A straightforward application of \cite[Theorem 7.2]{El_Karoui-reflec-1997}
yields the following representations:
\begin{align}
\overline\cE^K_{\tau_1,\tau_2}[\xi]
=\esssup_{\| (h^0,h)\|_{\cL^{\infty}([\tau_1,\tau_2];\bR^{1+m})}\leq K} E_{\sF_{\tau_1}}\left[
\xi\cdot \exp\left\{\int_{\tau_1}^{\tau_2} h_s dW_s -\frac{1}{2}\int_{\tau_1}^{\tau_2}(|h_s|^2+2h^0_s)\,ds\right\}
\right],
\label{representation-u}\\
\underline\cE^K_{\tau_1,\tau_2}[\xi]
=\essinf_{\| (h^0,h)\|_{\cL^{\infty}([\tau_1,\tau_2];\bR^{1+m})}\leq K} E_{\sF_{\tau_1}}\left[
\xi\cdot \exp\left\{\int_{\tau_1}^{\tau_2} h_s dW_s -\frac{1}{2}\int_{\tau_1}^{\tau_2}(|h_s|^2 +2h^0_s)\,ds\right\}
\right].\label{representation-l}
\end{align}
From the representation \eqref{representation-l}, we have
\begin{align}
&E_{\sF_{\tau_1}}[\sqrt{|\xi|}] 
\nonumber\\
&=  \essinf_{\|(h^0, h)\|_{\cL^{\infty}([\tau_1,\tau_2];\bR^{1+m)}}\leq K} 
E_{\sF_{\tau_1}}\left[
\sqrt{|\xi|} \cdot \exp\left\{\frac{1}{2}\int_{\tau_1}^{\tau_2} h_s dW_s -\frac{1}{4}\int_{\tau_1}^{\tau_2}(|h_s|^2+2h^0_s)\,ds\right\}  \right.
\nonumber \\
&\quad\quad \quad \quad \quad\quad\quad\quad\quad\quad\quad
\left.
\cdot \exp\left\{-\frac{1}{2}\int_{\tau_1}^{\tau_2} h_s dW_s +\frac{1}{4}\int_{\tau_1}^{\tau_2}(|h_s|^2+2h^0_s)\,ds\right\}
\right]
\nonumber \\
&\leq
 \essinf_{\|(h^0, h)\|_{\cL^{\infty}([\tau_1,\tau_2];\bR^{1+m)}}\leq K}
\left( E_{\sF_{\tau_1}}\left[
{|\xi|} \cdot \exp\left\{\int_{\tau_1}^{\tau_2} h_s dW_s -\frac{1}{2}\int_{\tau_1}^{\tau_2}(|h_s|^2+2h^0_s)\,ds\right\}  \right] \right)^{1/2}
\nonumber \\
&\quad
\cdot\esssup_{\|(h^0, h)\|_{\cL^{\infty}([\tau_1,\tau_2];\bR^{1+m)}}\leq K}
  \left(E_{\sF_{\tau_1}}\left[  
 \exp\left\{
 	- \int_{\tau_1}^{\tau_2} h_s dW_s +\frac{1}{2}\int_{\tau_1}^{\tau_2}(|h_s|^2+2h^0_s)\,ds
  		\right\}
\right]\right)^{1/2}
\nonumber  \\
&\leq \sqrt{ \underline\cE_{\tau_1,\tau_2}^K \left[
  |\xi| \right] \cdot e^{(K+2)K\|\tau_2-\tau_1\|_{L^{\infty}(\Omega,\sF,\bP)}}},\quad \text{a.s.}  \label{est-inf-dominate}
\end{align}
Similarly, it holds that
\begin{align}
\overline\cE_{\tau_1,\tau_2}^K \left[
  |\xi| \right]  \leq \sqrt{ e^{(K+2)K\|\tau_2-\tau_1\|_{L^{\infty}(\Omega,\sF,\bP)}}  \cdot  E_{\mathscr{F}_{\tau_1}}\left[ |\xi|^2   \right] },\quad \text{a.s.}
  \label{est-sup-dominate}
\end{align}

On the other hand, given $\xi^1,\xi^2\in L^2(\Omega,\sF_{\tau_2},\bP)$, recall that 
\begin{align*}
G_{s,\tau_2}^{\tau_1,x;\theta,\gamma}[\xi^i]=\overline Y_s^{\tau_1,x;\theta,\gamma,\xi^i},\quad s\in [\tau_1,\tau_2], \text{ for }i=1,2,
\end{align*}
where $\overline Y^{\tau_1,x;\theta,\gamma,\xi^i}$ together with $\overline Z^{\tau_1,x;\theta,\gamma,\xi^i}$ satisfies the following BSDE:
\begin{equation*}
  \left\{
  \begin{split}
  -d\overline Y_s^{\tau_1,x;\theta,\gamma,\xi^i}&=\,f(s,X_s^{\tau_1,x;\theta,\gamma},\overline Y_s^{\tau_1,x;\theta,\gamma,\xi^i},\overline Z_{s}^{\tau_1,x;\theta,\gamma,\xi^i},\theta_s,\gamma_s)\,ds
  -\overline Z_{s}^{\tau_1,x;\theta,\gamma,\xi^i}\,dW_s,\quad s\in[\tau_1, \tau_2];\\
 \overline Y_{\tau_2}^{\tau_1,x;\theta,\gamma,\xi^i}&=\xi^i,
    \end{split}
  \right.
\end{equation*}
with $X^{\tau_1,x;\theta,\gamma}$ being the solution to SDE \eqref{state-proces-control}.  In view of the uniform Lipschitz continuity of function $f$ in Assumption $\textbf{(A1)}$ and the comparison principle of BSDEs recalled in Proposition \ref{prop-BSDE-comp}, we have
\begin{align}
\underline\cE^L_{\tau_1,\tau_2}\left[ \xi^1-\xi^2\right]
\leq
G_{s,\tau_2}^{\tau_1,x;\theta,\gamma}[\xi^1]
-
G_{s,\tau_2}^{\tau_1,x;\theta,\gamma}[\xi^2]
\leq \overline\cE^L_{\tau_1,\tau_2}\left[ \xi^1-\xi^2\right], \quad\text{a.s.}
\label{dominate-G}
\end{align}

In the following lemma, we summarize some properties of the nonlinear functionals $\overline\cE^K_{\tau_1,\tau_2}$ and $\underline\cE^K_{\tau_1,\tau_2}$. As the assertions are following straightforwardly from the standard BSDE theory (see \cite{Karoui_Peng_Quenez,El_Karoui-reflec-1997,El-Karoui-Peng-Quenez-2001} for instance), the proofs are omitted.

\begin{lem}\label{lem-nonlinear}
Let $K\geq 0$ and $\tau_1,\tau_2\in \mathcal{T}^0$ with $\tau_1\leq \tau_2$. Given $\xi,\xi_1,\xi_2 \in L^2(\Omega,\sF_{\tau_2},\bP)$, it holds that:\\[3pt]
(i) $\overline\cE^K_{\tau_1,\tau_2}[\xi] =-\underline\cE^K_{\tau_1,\tau_2}[-\xi]$ a.s. and for any bounded $\sF_{\tau_1}$-measurable $\lambda$, we have $\overline\cE^K_{\tau_1,\tau_2}[\lambda\,\xi]= \lambda \,\overline\cE^K_{\tau_1,\tau_2}[\xi]$ and $\underline\cE^K_{\tau_1,\tau_2}[\lambda\,\xi]= \lambda \, \underline\cE^K_{\tau_1,\tau_2}[\xi]$; \\[2pt]
(ii) for each $\xi_{\tau_1}\in L^2(\Omega,\sF_{\tau_1},\bP)$, we have 
\begin{align*}
\overline\cE^K_{\tau_1,\tau_2}[\xi_{\tau_1}]
&=\xi_{\tau_1}^+ \cdot E_{\sF_{\tau_1}}\left[e^{K(\tau_2-\tau_1)}\right] -\xi_{\tau_1}^- \cdot E_{\sF_{\tau_1}}\left[e^{-K(\tau_2-\tau_1)}\right] ,
\\
 \underline\cE^K_{\tau_1,\tau_2}[\xi_{\tau_1}]
&=\xi_{\tau_1}^+ \cdot E_{\sF_{\tau_1}}\left[e^{-K(\tau_2-\tau_1)}\right]  -\xi_{\tau_1}^-\cdot E_{\sF_{\tau_1}}\left[e^{K(\tau_2-\tau_1)}\right] ,
 \end{align*}
 and if we assume further that $\tau_2$ is $\sF_{\tau_1}$-measurable, then the above conditional expectations $E_{\sF_{\tau_1}}[\cdot]$ may be dropped and it holds that 
 $\underline\cE^K_{\tau_1,\tau_2}\left[\overline\cE^K_{\tau_1,\tau_2}[\xi_{\tau_1}]\right]
 =
 \overline\cE^K_{\tau_1,\tau_2}\left[\underline\cE^K_{\tau_1,\tau_2}[\xi_{\tau_1}]\right]=\xi_{\tau_1}$ a.s.;\\[2pt]
(iii) by the comparison principle of BSDEs in Proposition \ref{prop-BSDE-comp},
for $0\leq K\leq \hat K$, $\overline\cE^K_{\tau_1,\tau_2}[\xi] \leq \overline\cE^{\hat K}_{\tau_1,\tau_2}[\xi]$ and $\underline\cE^K_{\tau_1,\tau_2}[\xi] \geq \underline\cE^{\hat K}_{\tau_1,\tau_2}[\xi]$ a.s.; \\[2pt]
(iv) when $K=0$, we have   $\overline\cE^0_{\tau_1,\tau_2}[\xi] = E_{\sF_{\tau_1}}[\xi]=\underline\cE^0_{\tau_1,\tau_2}[\xi]  $;\\[2pt]
(v) with probability 1, we have the representations \eqref{representation-u} and \eqref{representation-l} and dominating relations \eqref{est-inf-dominate}, \eqref{est-sup-dominate}, and \eqref{dominate-G} hold;\\[2pt]
(vi) with probability 1,
\begin{align*}
&\underline\cE^K_{\tau_1,\tau_2}[\xi_1] + \underline\cE^K_{\tau_1,\tau_2}[\xi_2]  \leq \underline\cE^{ K}_{\tau_1,\tau_2}[\xi_1+\xi_2],
\\
&\overline\cE^K_{\tau_1,\tau_2}[\xi_1] + \overline\cE^K_{\tau_1,\tau_2}[\xi_2]  \geq \overline\cE^{ K}_{\tau_1,\tau_2}[\xi_1+\xi_2].
\end{align*}
\end{lem}

\subsection{Definition of viscosity solutions and some properties}
We then define the test function space for the viscosity solutions.
\begin{defn}\label{defn-testfunc}
For $u\in \cS^{4} (C_b^1(\bR^d))$ with $Du,\, D^2u \in\cS^4_{\text{loc}}([0,T);C_b^1(\bR^d))$, we say $u\in \mathscr C_{\sF}^3$ if 
there exists $(\mathfrak{d}_tu, \,\mathfrak{d}_{\omega}u)\in \cL^4_{\text{loc}}([0,T);C_b^1(\bR^d))\times  \cL^{4,2}_{\text{loc}}([0,T);C_b^2(\bR^d))$ with $D\mathfrak{d}_tu$,  $D\mathfrak{d}_{\omega}u$\,\, $\in$ $\cS^4_{\text{loc}}([0,T);C_b^1(\bR^d))$ such that a.s.,
\begin{align*}
u(r,x)=u(T_0,x)-\int_r^{T_0} \mathfrak{d}_su(s,x)\,ds -\int_r^{T_0}\mathfrak{d}_{\omega}u(s,x)\,dW_s,\quad \text{for all }\,0\leq r\leq T_0<T,\, x\in \bR^d.
\end{align*}
\end{defn}
For each $\xi\in L^4(\Omega,\sF_{T};\bR)$ and $h\in\cL^4([0,T];\bR)$, put
$$
Y_{t}=E_{\sF_t}\left[
\xi+\int_t^T h_s\,ds
\right], \quad\text{for }t\in[0,T].
$$
Then by the standard BSDE theory, it is easy to check that the space-invariant process $(Y_t)_{t\in[0,T]}$ belongs to $\mathscr C_{\sF}^3$.
The space $\mathscr C_{\sF}^3$ may be defined on general time subintervals like $[a,b]\subset [0,T]$ for some $a<b$, and in this case, we may write $\mathscr C_{\sF}^3([a,b])$ to specify the time interval. Here, the two linear operators $\mathfrak{d}_t$ and $\mathfrak{d}_{\omega}$ are defined on $ \mathscr C^3_{\sF}$, and they are actually consistent with the \textit{differential} operators w.r.t. the paths of Wiener process $W$ in \cite{Leao-etal-2018} and \cite[Section 5.2]{cont2013-founctional-aop}.    

We now introduce the notion of viscosity solutions. For each $K\geq 0$, $(u,\tau)\in \cS^{2}(C_b(\bR^d))\times \mathcal T^0$, $\Omega_{\tau}\in\sF_{\tau}$ with $\mathbb P(\Omega_{\tau})>0$, and $\xi\in L^0(\Omega_{\tau},\sF_{\tau};\bR^d)$, we define
{\small
\begin{align*}
\underline{\mathcal{G}}u(\tau,\xi;\Omega_{\tau},K):=\bigg\{
\phi\in\mathscr C^3_{\sF}:(\phi-u)(\tau,\xi)1_{\Omega_{\tau}}=0=\essinf_{\bar\tau\in\mathcal T^{\tau}} \underline\cE^K_{{\tau},\bar\tau\wedge \hat{\tau}}\left[\inf_{y\in \bR^d}
(\phi-u)(\bar\tau\wedge \hat{\tau},y)
\right]1_{\Omega_{\tau}}, \text{ a.s.},
&\\
\text{for some } \hat\tau \in \mathcal T^\tau_+\,&
\bigg\},\\
\overline{\mathcal{G}}u(\tau,\xi;\Omega_{\tau},K):=\bigg\{
\phi\in\mathscr C^3_{\sF}:(\phi-u)(\tau,\xi)1_{\Omega_{\tau}}=0=\esssup_{\bar\tau\in\mathcal T^{\tau}} \overline \cE^K_{{\tau},\bar\tau\wedge \hat{\tau} }\left[\sup_{y\in \bR^d}
(\phi-u)(\bar\tau\wedge \hat{\tau},y)
\right]1_{\Omega_{\tau}}, \text{ a.s.},
&\\
\text{for some } \hat\tau \in  \mathcal T^\tau_+\,&
\bigg\},
\end{align*}
}
where we call $\hat \tau$ the stopping time associated  to the relation $\phi\in \underline{\mathcal{G}}u(\tau,\xi;\Omega_{\tau},K)$ (or $\phi\in \overline{\mathcal{G}}u(\tau,\xi;\Omega_{\tau},K)$).
It is obvious that if either $\underline{\mathcal{G}}u(\tau,\xi;\Omega_{\tau},K)$ or $\overline{\mathcal{G}}u(\tau,\xi;\Omega_{\tau},K)$ is nonempty, we must have $0\leq\tau <T$ on $\Omega_{\tau}$.

\begin{lem}\label{lem-linear-growth}
Let $h$ be an infinitely differentiable function such that there exist constants $\alpha_0,\alpha_1,\alpha_2\in (0,\infty)$ such that
\begin{align*}
h(x)> \alpha_1 |x|-\alpha_2,\quad  |D^2h(x)|\leq \alpha_0,\quad  \forall
\, x\in\bR^d. 
\end{align*}
Given $(u,\tau)\in \cS^{2}(C_b(\bR^d))\times \mathcal T^0$, $\Omega_{\tau}\in\sF_{\tau}$ with $0\leq \tau<T$ on $\Omega_{\tau}$, $\mathbb P(\Omega_{\tau})>0$, and $\xi\in L^0(\Omega_{\tau},\sF_{\tau};\bR^d)$, suppose that there exists $  \phi \in \mathscr C^3_{\sF}$ such that for some $K\geq 0$,
\begin{align*}
(\phi+h-u)(\tau,\xi)1_{\Omega_{\tau}}=0=\essinf_{\bar\tau\in\mathcal T^{\tau}} \underline\cE^K_{{\tau},\bar\tau\wedge \hat{\tau}}\left[\inf_{y\in \bR^d}
(\phi+h-u)(\bar\tau\wedge \hat{\tau},y)
\right]1_{\Omega_{\tau}}, \text{ a.s.},
\text{ for some } \hat\tau \in \mathcal T^\tau_+.
\end{align*}
Then there exist $\tilde \tau \in \mathcal T^{\tau}_+$ and $\tilde \phi\in  \underline{\mathcal{G}}u(\tau,\xi;\Omega'_{\tau},K)$ with $\Omega'_{\tau}\in \sF_{\tau}$, $\Omega'_{\tau}\subset \Omega_{\tau}$, and $\bP (\Omega'_{\tau})>\frac{\bP(\Omega_{\tau})}{2}>0$ such that 
\begin{align*}
&(\mathfrak{d}_{s}\tilde\phi,D^2\tilde\phi,D\mathfrak{d}_{\omega}\tilde\phi, D\tilde\phi, \tilde\phi, \mathfrak{d}_{\omega}\tilde\phi )(s,\xi) 1_{\Omega'_{\tau}}
\\
&= (\mathfrak{d}_{s}(\phi+h),D^2(\phi+h),D\mathfrak{d}_{\omega}(\phi+h), D(\phi+h), \phi+h, \mathfrak{d}_{\omega}(\phi+h) )(s,\xi)
1_{\Omega'_{\tau}}
\end{align*}
for all $\tau \leq s\leq \tilde \tau$, a.s., and
$$\inf_{x\in\bR^d}(\tilde \phi (s,x)-u(s,x)) 1_{\Omega'_{\tau}}
=\inf_{x\in\bR^d} (\phi(s,x)+h(x)-u(s,x)) 1_{\Omega'_{\tau}}, 
	\text{ for all } 
		\tau\leq s\leq \tilde \tau, \text{ a.s.}
$$
\end{lem}
\begin{proof}
As $u\in \cS^{2}(C_b(\bR^d))$, we may choose such a big $M_0>0$ that $\bP ( \Omega'_{\tau})>\frac{\bP(\Omega_{\tau})}{2}>0$ where 
$$\Omega'_{\tau}:= \left\{ \sup_{x\in\bR^d}(|u(\tau,x)|+ |\phi (\tau,x)|)<M_0\right\} \cap \Omega_{\tau}.$$
  Set
$$
\tilde\tau=\inf\{s>\tau:\, \sup_{x\in\bR^d}(|u(s,x)|+ |\phi (s,x)|)>M_0+1\}
\wedge \hat\tau.
$$
Obviously, it holds that $\tilde \tau \in \mathcal T^{\tau}_+$ on $\Omega'_{\tau}$. Moreover, for $\tau\leq s\leq \tilde \tau$, the infimum $\inf_{x\in\bR^d} (\phi(s,x)+h(x)-u(s,x))$ must be achieved at some point $x$ (one may take $x=\xi$ when $s=\tau$) inside the ball $B_{M_1}(0)$  on $\Omega'_{\tau}$ with $M_1:=\frac{2M_0+2+\alpha_2}{\alpha_1}$. 

Let $\chi$ be an infinitely differentiable $[0,1]$-valued (cutoff) function satisfying $\chi(x)=1$ when $|x|\leq M_1$ and $\chi(x)=0$ when $|x|\geq M_1+1$. Set $\tilde \phi(s,x)=\phi(s,x)+h(x)\chi(x)+(2M_0+2) (1-\chi(x))$. Then one may straightforwardly check that such a pair $(\tilde \tau, \tilde \phi)$ satisfies the desired properties.
\end{proof}
\begin{rmk}\label{rmk-linear-growth}
 For test functions in $\overline{\mathcal{G}}u(\tau,\xi;\Omega_{\tau},K)$, we may prove the similar results as in Lemma \ref{lem-linear-growth}. Let $h$ be an infinitely differentiable  as in Lemma \ref{lem-linear-growth}.
Given $(u,\tau)\in \cS^{2}(C_b(\bR^d))\times \mathcal T^0$, $\Omega_{\tau}\in\sF_{\tau}$ with $0\leq \tau<T$ on $\Omega_{\tau}$, $\mathbb P(\Omega_{\tau})>0$, and $\xi\in L^0(\Omega_{\tau},\sF_{\tau};\bR^d)$, suppose that there exists $  \phi \in \mathscr C^3_{\sF}$ such that for some $K\geq0$,
\begin{align*}
(\phi-h-u)(\tau,\xi)1_{\Omega_{\tau}}=0=\esssup_{\bar\tau\in\mathcal T^{\tau}} \overline\cE^K_{{\tau},\bar\tau\wedge \hat{\tau}}\left[\sup_{y\in \bR^d}
(\phi-h-u)(\bar\tau\wedge \hat{\tau},y)
\right]1_{\Omega_{\tau}}, \text{ a.s.},
\text{ for some } \hat\tau \in \mathcal T^\tau_+.
\end{align*}
Then there exist $\tilde \tau \in \mathcal T^{\tau}_+$ and $\tilde \phi\in  \overline{\mathcal{G}}u(\tau,\xi;\Omega'_{\tau},K)$ with $\Omega'_{\tau}\in \sF_{\tau}$, $\Omega'_{\tau}\subset \Omega_{\tau}$, and $\bP (\Omega'_{\tau})>\frac{\bP(\Omega_{\tau})}{2}>0$ such that 
\begin{align*}
&(\mathfrak{d}_{s}\tilde\phi,D^2\tilde\phi,D\mathfrak{d}_{\omega}\tilde\phi, D\tilde\phi, \tilde\phi, \mathfrak{d}_{\omega}\tilde\phi )(s,\xi) 1_{\Omega'_{\tau}}
\\
&= (\mathfrak{d}_{s}(\phi-h),D^2(\phi-h),D\mathfrak{d}_{\omega}(\phi-h), D(\phi-h), \phi-h, \mathfrak{d}_{\omega}(\phi-h) )(s,\xi)
1_{\Omega'_{\tau}}
\end{align*}
for all $\tau \leq s\leq \tilde \tau$, a.s., and
$$\sup_{x\in\bR^d}(\tilde \phi (s,x)-u(s,x)) 1_{\Omega'_{\tau}}
=\sup_{x\in\bR^d} \left( \phi(s,x)-h(x)-u(s,x)\right) 1_{\Omega'_{\tau}},
	\text{ for all } 
		\tau\leq s\leq \tilde \tau, \text{ a.s.}
$$
\end{rmk}

Now we are ready to introduce the definition of viscosity solutions.

\begin{defn}\label{defn-viscosity}
We say $V\in \cS^2(C_b(\bR^d))$ is a viscosity subsolution (resp. supersolution) of BSPDE \eqref{SHJBI-low} (equivalently \eqref{SHJB-low-eqiv}), if $V(T,x)\leq (\text{ resp. }\geq) \Phi(x)$ for all $x\in\bR^d$ a.s., and there exists $K_0\geq 0$, such that for any $K\in [K_0,\infty)$, $\tau\in  \mathcal T^0$, $\Omega_{\tau}\in\sF_{\tau}$ with $\mathbb P(\Omega_{\tau})>0$, $\xi\in L^0(\Omega_{\tau},\sF_{\tau};\bR^d)$, and any $\phi\in \underline{\cG}V(\tau,\xi;\Omega_{\tau},K)$ (resp. $\phi\in \overline{\cG}V(\tau,\xi;\Omega_{\tau},K)$), there holds
\begin{align}
&\text{ess}\!\liminf_{(s,x)\rightarrow (\tau^+,\xi)}
	  \left\{ -\mathfrak{d}_{s}\phi-\bH_-(D^2\phi,D\mathfrak{d}_{\omega}\phi, D\phi, \phi, \mathfrak{d}_{\omega}\phi ) \right\} (s,x) \leq0,
\label{defn-vis-sub}
\end{align}
for almost all $\omega\in\Omega_{\tau}$ (resp.
\begin{align}
&\text{ess}\!\limsup_{(s,x)\rightarrow (\tau^+,\xi)} 
		\left\{ -\mathfrak{d}_{s}\phi-\bH_-(D^2\phi,D\mathfrak{d}_{\omega}\phi, D\phi, \phi, \mathfrak{d}_{\omega}\phi ) \right\}(s,x)  \geq0,
\label{defn-vis-sup}
\end{align}
for almost all $\omega\in\Omega_{\tau}$).

Equivalently, $V\in \cS^2(C_b(\bR^d))$ is a viscosity subsolution (resp. supersolution) of BSPDE \eqref{SHJBI-low}, if $V(T,x)\leq (\text{ resp. }\geq) \Phi(x)$ for all $x\in\bR^d$ a.s., and for any $K_0\in[0,\infty)$, $\tau\in  \mathcal T^0$, $\Omega_{\tau}\in\sF_{\tau}$ with $\mathbb P(\Omega_{\tau})>0$, $\xi\in L^0(\Omega_{\tau},\sF_{\tau};\bR^d)$, and any $\phi\in \mathscr C^3_{\sF}$, whenever there exist $\eps>0$, $\tilde\delta>0$ and $\Omega_{\tau}'\subset\Omega_{\tau}$ such that $\Omega'_{\tau}\in\sF_{\tau}$, $\bP(\Omega'_{\tau})>0$ and
{\small
\begin{align*}
&
	\essinf_{(s,x)\in Q^+_{\tilde\delta}(\tau,\xi) \cap Q}
	  \left\{ -\mathfrak{d}_{s}\phi-\bH_-(D^2\phi,D\mathfrak{d}_{\omega}\phi, D\phi, \phi, \mathfrak{d}_{\omega}\phi ) \right\} (s,x)\geq  \eps, \text{ a.e. in }\Omega'_{\tau}
\\
(\text{resp. }
&\esssup_{(s,x)\in Q^+_{\tilde\delta}(\tau,\xi) \cap Q}
	  \left\{ -\mathfrak{d}_{s}\phi-\bH_-(D^2\phi,D\mathfrak{d}_{\omega}\phi, D\phi, \phi, \mathfrak{d}_{\omega}\phi ) \right\}(s,x) \leq  -\eps, \text{ a.e. in }\Omega'_{\tau}),
\end{align*}
}
then $\phi\notin \underline{\cG}V(\tau,\xi;\Omega_{\tau},K)$ (resp. $\phi\notin \overline{\cG}V(\tau,\xi;\Omega_{\tau},K)$) for any $K\geq K_0$.

The function $u$ is a viscosity solution of BSPDE \eqref{SHJBI-low} (equivalently \eqref{SHJB-low-eqiv}) if it is both a viscosity subsolution and a viscosity supersolution of \eqref{SHJBI-low}. The viscosity solution is defined analogously for general BSPDEs, especially for BSPDE \eqref{SHJBI-up} (equivalently \eqref{SHJB-up-eqiv}).
\end{defn}

In Definition \ref{defn-viscosity}, we shall also call $K_0$ the number associated to the viscosity (semi)solution; we may also say viscosity $K_0$-sub/supersolution or viscosity $K_0$-solution, inspired by the viscosity solutions for path-dependent PDEs in \cite{ekren2014viscosity,ekren2016viscosity-1}.
We shall give some remarks about the viscosity solutions. First, let
\begin{align*}
 \mathbb{L}^{\theta,\gamma}(t,x,A,B,p,y,z)
=  \,&
\text{tr}\left(\frac{1}{2}\sigma \sigma'(t,x,\theta,\gamma) A+\sigma(t,x,\theta,\gamma) B\right)
       +b'(t,x,\theta,\gamma)p \\
       &\,+f(t,x,y,z+\sigma'(t,x,\theta,\gamma)p,\theta,\gamma)
               ,               
\end{align*}
for $(t,x,A,B,p,y,z,\theta,\gamma)\in [0,T]\times\bR^d\times\bR^{d\times d}\times \bR^{m\times d} \times   \bR^d\times \bR \times \bR^m\times \Theta_0\times\Gamma_0$. 
In view of the assumption $\textbf{({A}1)}$, for each $\phi\in \mathscr C_{\sF}^3$,  there is $\zeta^{\phi}\in \cL^4_{\text{loc}}([0,T); \bR)$ such that for all $x,\bar x \in \bR^d$ and a.e. $(\omega,t)\in \Omega\times [0,T)$,
\begin{align}
 \esssup_{\gamma\in \Gamma_0, \theta\in\Theta_0} \Big| \left\{ -\mathfrak{d}_{s}\phi-\mathbb{L}^{\theta,\gamma}(D^2\phi,D\mathfrak{d}_{\omega}\phi, D\phi, \phi, \mathfrak{d}_{\omega}\phi ) \right\}(t,x) \Big| \leq \zeta^{\phi}_t,&\label{H-bounded}\\ 
  \esssup_{\gamma\in \Gamma_0, \theta\in\Theta_0}  \left|
  \mathbb{L}^{\theta,\gamma}(D^2\phi,D\mathfrak{d}_{\omega}\phi, D\phi, \phi, \mathfrak{d}_{\omega}\phi ) (t,x)
  - \mathbb{L}^{\theta,\gamma}(D^2\phi,D\mathfrak{d}_{\omega}\phi, D\phi, \phi, \mathfrak{d}_{\omega}\phi )  (t,\bar x)\right| \qquad\,\,\,&
  \nonumber\\
  \leq \zeta^{\phi}_t \cdot |x-\bar x|,&
 \label{R-Lip-const}
\end{align}
where by Definition \ref{defn-testfunc} we may further have $\zeta^{\phi}\in \cS^4_{\text{loc}}([0,T); \bR)$ in \eqref{R-Lip-const}.

 

 %
\begin{rmk}\label{rmk-defn-timechange}
If $u$ is a viscosity $K_0$-subsolution (resp. $K_0$-supersolution) of BSPDE \eqref{SHJBI-low} (equivalently \eqref{SHJB-low-eqiv}) for some $K_0\geq 0$, then for each $\lambda \in(-\infty,0]$, $\tilde{u}(t,x):=e^{\lambda t} u(t,x)$ is a viscosity $K_0$-subsolution (resp. $K_0$-supersolution) of BSPDE:
{\small
\begin{equation}\label{SHJB-tilde}
  \left\{\begin{array}{l}
  \begin{split}
  -\mathfrak{d}_{t}\tilde u(t,x)=\,& 
 \mathbb{K}(t,x,D^2\tilde u(t,x),D\mathfrak{d}_{\omega}\tilde u(t,x),D\tilde u(t,x),\tilde u(t,x),\mathfrak{d}_{\omega}\tilde u (t,x)) , \quad
                     (t,x)\in Q;\\
    \tilde u(T,x)=\, &e^{\lambda T}\Phi(x), \quad x\in\bR^d,
    \end{split}
  \end{array}\right.
\end{equation}
}
where for $(t,x,A,B,p,y,z)\in [0,T]\times\bR^d\times\bR^{d\times d}\times \bR^{m\times d} \times   \bR^d\times \bR \times \bR^m$,
\begin{align*}
\mathbb{K}(t,x,A,B,p,y,z) 
:= \lambda y +e^{\lambda t}\mathbb{H}_-(t,x,e^{-\lambda t}A,e^{-\lambda t}B,e^{-\lambda t}p,e^{-\lambda t}y,e^{-\lambda t}z) .
\end{align*}
Therefore, w.l.o.g., we may assume that $\mathbb{H}_-(t,x,A,B,p,y,z)$ is decreasing in $y$.

Indeed, assume $ u$ is a viscosity $K_0$-subsolution of BSPDE \eqref{SHJBI-low}/\eqref{SHJB-low-eqiv} and take $\tilde{u}(t,x)=e^{\lambda t} u(t,x)$. Let $\tilde{\phi}\in \underline{\cG}\tilde{u}(\tau,\xi;\Omega_{\tau},K)$ for $K\geq K_0$,  $\tau\in\cT^0$, $\Omega_{\tau}\in\sF_{\tau}$, and $\xi\in L^0(\Omega_{\tau},\sF_{\tau};\bR^d)$. Let ${ \tilde \tau}\in \cT^{\tau}_+$ be a stopping time, corresponding to  $\tilde{\phi}\in \underline{\cG}\tilde{u}(\tau,\xi;\Omega_{\tau},K)$. Set
$$
\phi^{\eps}(s,x)=e^{-\lambda s}\tilde{\phi}(s,x)+E_{\sF_{s}}[\eps(s-\tau)],\quad \forall \, \eps >0.
$$
Noticing that $\lambda\leq 0$ and $\inf_{x\in\bR^d}\tilde \phi(\tau,x)-e^{\lambda \tau} u(\tau,x)=0 = \tilde \phi(\tau,\xi)-e^{\lambda \tau} u(\tau,\xi)$ for almost all $\omega\in\Omega_{\tau}$, we have for $s>\tau$ and $x\in\bR^d$,
\begin{align*}
&\phi^{\eps}(s,x)-u(s,x)-e^{-\lambda {\tau}} (\tilde{\phi}-\tilde u)(s,x)\\
&=\left(  e^{-\lambda s}- e^{-\lambda {\tau}}\right)\tilde{\phi}(s,x) +\left(e^{\lambda (s-{\tau}) }-1 \right) u(s,x) + \eps (s-\tau)\\
&
\geq
\left(  e^{-\lambda s}- e^{-\lambda {\tau}}\right)\left(\tilde{\phi}(s,x)  -\tilde{\phi}(\tau,x) \right)+\left(e^{\lambda (s-{\tau}) }-1 \right) (u(s,x)-u(\tau,x)) + \eps (s-\tau)\\
&\quad\quad + \left( e^{-\lambda (s-{\tau}) }+e^{\lambda (s-{\tau}) }  -2\right) u(\tau,x)\\
&\geq \eps(s-\tau) -C (s-\tau) \Big(  \|\tilde\phi(s,\cdot) -\tilde\phi(\tau,\cdot)  \|_{L^{\infty}(\bR^d)} 
+ \| u (s,\cdot) - u (\tau,\cdot)  \|_{L^{\infty}(\bR^d)}   \\
&\quad\quad+\|  u (\tau,\cdot)  \|_{L^{\infty}(\bR^d)} (s-\tau) \Big).
\end{align*}

Set 
\begin{align*}
\hat \tau^{\eps}= \tilde \tau \wedge \inf \bigg\{
s>\tau:\, \|\tilde\phi(s,\cdot) -\tilde\phi(\tau,\cdot)  \|_{L^{\infty}(\bR^d)} 
+ \| u (s,\cdot) - u (\tau,\cdot)  \|_{L^{\infty}(\bR^d)}   &
\\
\quad\quad\quad+\|  u (\tau,\cdot)  \|_{L^{\infty}(\bR^d)} (s-\tau) \geq \frac{\eps}{C}&
\bigg\}.
\end{align*}
Then $\hat \tau^{\eps}\in \cT^{\tau}_+$ and for any $(\bar \tau,x)\in \cT^{\tau}\times \bR^d$ with $\bar \tau\leq \hat \tau^{\eps}$, we have 
$$
\phi^{\eps}(\bar \tau,x)-u(\bar\tau,x) \geq e^{-\lambda \tau}\left(   \tilde{\phi}(\bar \tau,x)-\tilde{u}(\bar\tau,x)   \right),\quad \text{a.s.}
$$
Then, the fact $\tilde{\phi}\in \underline{\cG}\tilde{u}(\tau,\xi;\Omega_{\tau},K)$ yields for almost all $\omega\in\Omega_{\tau}$,
\begin{align*}
\underline\cE^K_{{\tau},\bar\tau}\left[\inf_{y\in \bR^d}
(\phi^{\eps}-u)(\bar\tau,y)
\right]
&\geq 
\underline \cE^K_{{\tau},\bar\tau}\left[\inf_{y\in \bR^d} e^{-\lambda \tau}
(\tilde\phi-\tilde u)(\bar\tau,y)
\right]\\
&= e^{-\lambda \tau} \underline \cE^K_{{\tau},\bar\tau}\left[\inf_{y\in \bR^d}
(\tilde\phi-\tilde u)(\bar\tau,y)
\right]\\
&\geq 0= \left(\phi^{\eps}-u\right)(\tau,\xi).
\end{align*}
This implies that $\phi^{\eps}\in \underline{\cG}u(\tau,\xi;\Omega_{\tau},K)$ and thus, 
$$
\text{ess}\!\liminf_{(s,x)\rightarrow (\tau^+,\xi)}
		 \left\{ -\mathfrak{d}_s\phi^{\eps}-\bH_-(D^2\phi^{\eps},D\mathfrak{d}_{\omega}\phi^{\eps},D\phi^{\eps},\phi^{\eps},\mathfrak{d}_{\omega}\phi^{\eps})\right\} (s,x)  \leq \,0,$$
for almost all $\omega\in\Omega_{\tau}$. As $\eps$ tends to zero, we have 
$$
\text{ess}\!\liminf_{(s,x)\rightarrow (\tau^+,\xi)} 
	 	 \left\{ -\mathfrak{d}_s\phi^0-\bH_-(D^2\phi^0,D\mathfrak{d}_{\omega}\phi^0,D\phi^0,\phi^0,\mathfrak{d}_{\omega}\phi^0)\right\} (s,x) \leq \,0,$$
for almost all $\omega\in\Omega_{\tau}$, with $\phi^0(s,x)=e^{-\lambda s}\tilde \phi(s,x)$. Straightforward calculations yield
$$
\text{ess}\!\liminf_{(s,x)\rightarrow (\tau^+,\xi)} 
		\left\{ -\mathfrak{d}_s\tilde\phi-\mathbb K(D^2\tilde\phi,D\mathfrak{d}_{\omega}\tilde\phi,D\tilde\phi,\tilde{\phi},\mathfrak{d}_\omega\tilde\phi)\right\}(s,x) \leq \,0,
$$
for almost all $\omega\in\Omega_{\tau}$, which finally implies that $\tilde u$ is a viscosity subsolution of BSPDE \eqref{SHJB-tilde}.
\end{rmk}

\begin{rmk}\label{rmk-eq-def-vs}
Here, the definition of viscosity solution in \ref{defn-viscosity} is different from that of \cite{qiu2017viscosity} as we drop the conditional expectations in relations like \eqref{defn-vis-sub} and \eqref{defn-vis-sup}, use the nonlinear expectations $\overline\cE^K$ and $\underline\cE^K$ and employ the test functions from $\mathscr C_{\sF}^3$ which have stronger regularity than those in  \cite{qiu2017viscosity}. This is  to overcome the difficulties arising from the non-convexity  of the game and the nonlinear dependence of function $f$ on unknown variables. In principle, the stronger regularity of test functions here makes the existence of viscosity solution more tractable while increasing the difficulties for the uniquness. On the other hand, these finer test functions allow us to define the viscosity solution equivalently, by replacing the relations \eqref{defn-vis-sub} and \eqref{defn-vis-sup} respectively by the following
\begin{align}
&\text{ess}\!\liminf_{s\rightarrow \tau^+}
	  \left\{ -\mathfrak{d}_{s}\phi-\bH_-(D^2\phi,D\mathfrak{d}_{\omega}\phi, D\phi, \phi, \mathfrak{d}_{\omega}\phi ) \right\} (s,\xi) \leq0,
\label{defn-vis-sub-1}
\end{align}
for almost all $\omega\in\Omega_{\tau}$, and
\begin{align}
&\text{ess}\!\limsup_{s\rightarrow \tau^+} 
		\left\{ -\mathfrak{d}_{s}\phi-\bH_-(D^2\phi,D\mathfrak{d}_{\omega}\phi, D\phi, \phi, \mathfrak{d}_{\omega}\phi ) \right\}(s,\xi)  \geq0,
\label{defn-vis-sup-1}
\end{align}
for almost all $\omega\in\Omega_{\tau}$.

In fact, the relations \eqref{defn-vis-sub-1} and \eqref{defn-vis-sup-1} straightforwardly imply \eqref{defn-vis-sub} and \eqref{defn-vis-sup} respectively, and for the equivalence, we need only to show how to derive \eqref{defn-vis-sub-1} and \eqref{defn-vis-sup-1}  from  \eqref{defn-vis-sub} and \eqref{defn-vis-sup} respectively. We may use contradiction arguments. Suppose  \eqref{defn-vis-sub} is holding, whereas \eqref{defn-vis-sub-1} is not true. Then for some $\phi\in \underline{\cG}V(\tau,\xi;\Omega_{\tau},K)$ with $K\geq 0$, there exist $\eps>0$, $\tilde\delta>0$, $T_0\in(0,T)$, and $\Omega_{\tau}'\subset\Omega_{\tau}$ such that $\Omega'_{\tau}\in\sF_{\tau}$, $\{\tau<T_0\}\subset \Omega'_{\tau}$, $\bP(\Omega'_{\tau})>0$ and
\begin{align*}
&
	\essinf_{s\in [\tau,\tau+\tilde\delta \wedge T_0 ] }
	 \left\{ -\mathfrak{d}_{s}\phi-\bH_-(D^2\phi,D\mathfrak{d}_{\omega}\phi, D\phi, \phi, \mathfrak{d}_{\omega}\phi ) \right\} (s,\xi)\geq  \eps, \text{ a.e. in }\Omega'_{\tau}.
\end{align*}
Then for each $\delta \in (0,\tilde \delta)  $,  
\begin{align*}
&\frac{1}{\delta}\int_{\tau}^{\tau+\delta \wedge T_0}
\essinf_{s\in [\tau,\tau+\tilde\delta \wedge t ] }
 \left\{ -\mathfrak{d}_{s}\phi-\bH_-(D^2\phi,D\mathfrak{d}_{\omega}\phi, D\phi, \phi, \mathfrak{d}_{\omega}\phi ) \right\} (s,\xi) \,dt
\\
 &\geq \eps\cdot \frac{(\tau+\delta \wedge T_0) -\tau}{\delta}, \text{ a.e. in }\Omega'_{\tau},
\end{align*}
and thus, recalling \eqref{R-Lip-const} (with $\zeta^{\phi}\in \cS^4_{\text{loc}}([0,T); \bR)$), we have  for each $\rho\in (0,1)$,
\begin{align*}
&
 \frac{(\tau+\delta\wedge T_0) -\tau}{\delta} \cdot
\text{ess}\!\liminf_{(s,x)\rightarrow (\tau^+,\xi)}
	  \left\{ -\mathfrak{d}_{s}\phi-\bH_-(D^2\phi,D\mathfrak{d}_{\omega}\phi, D\phi, \phi, \mathfrak{d}_{\omega}\phi ) \right\} (s,x)
	\\
&
\geq
\frac{1}{\delta}\int_{\tau}^{\tau+\delta \wedge T_0}
\essinf_{(s,x_{\rho})\in [\tau,\tau+\delta \wedge t ]\times B_{\rho}(\xi)}
  \left\{ -\mathfrak{d}_{s}\phi-\bH_-(D^2\phi,D\mathfrak{d}_{\omega}\phi, D\phi, \phi, \mathfrak{d}_{\omega}\phi ) \right\} (s,x_{\rho}) \,dt
\\
&\geq
\frac{1}{\delta}\int_{\tau}^{\tau+\delta \wedge T_0}
\essinf_{s\in [\tau,\tau+\delta \wedge t ] }
  \left\{ -\mathfrak{d}_{s}\phi-\bH_-(D^2\phi,D\mathfrak{d}_{\omega}\phi, D\phi, \phi, \mathfrak{d}_{\omega}\phi ) \right\} (s,\xi) \,dt
\\
&\quad 
-
\frac{1}{\delta}\int_{\tau}^{\tau+\delta \wedge T_0}
\esssup_{(s,x_{\rho})\in [\tau,\tau+\delta \wedge t ]\times B_{\rho}(\xi)}
 \bigg| 
 \left\{ -\mathfrak{d}_{s}\phi-\bH_-(D^2\phi,D\mathfrak{d}_{\omega}\phi, D\phi, \phi, \mathfrak{d}_{\omega}\phi ) \right\} (s,x_{\rho})
 \\
&\quad\quad\quad\quad\quad\quad
-\left\{ -\mathfrak{d}_{s}\phi-\bH_-(D^2\phi,D\mathfrak{d}_{\omega}\phi, D\phi, \phi, \mathfrak{d}_{\omega}\phi ) \right\} (s,\xi)\bigg|
  \,dt
  \\
  &\geq
  \eps \cdot \frac{(\tau+\delta\wedge T_0) -\tau}{\delta} -\frac{1}{\delta}\int_{\tau}^{\tau+\delta \wedge T_0} 
    \left[  \esssup_{s\in [0,T_0]}\left|\zeta^{\phi}_s\right| \cdot \rho  \right] ds
  \\
  &=\eps\cdot \frac{(\tau+\delta\wedge T_0) -\tau}{\delta}
  -\rho\cdot  \frac{(\tau+\delta\wedge T_0) -\tau}{\delta} \cdot   \left(  \esssup_{s\in [0,T_0]}\left|\zeta^{\phi}_s\right|  \right)
  , \text{ a.e. in }\Omega'_{\tau}.
\end{align*}
 By the assumed relation \eqref{defn-vis-sub}, taking expectations on $\Omega'_{\tau}$ on both sides gives  
 \begin{align*}
 0\geq E\left[ 1_{\Omega'_{\tau}}\eps\cdot \frac{(\tau+\delta\wedge T_0) -\tau}{\delta}
  -1_{\Omega'_{\tau}}\rho\cdot  \frac{(\tau+\delta\wedge T_0) -\tau}{\delta} \cdot   \left(  \esssup_{s\in [0,T_0]}\left|\zeta^{\phi}_s\right|  \right)
 \right],
  \end{align*}
which by the dominated convergence theorem incurs a contradiction when $\delta $ and $\rho$ are sufficiently small. In this way, we prove the equivalence between \eqref{defn-vis-sub} (resp. \eqref{defn-vis-sup}) and  \eqref{defn-vis-sub-1} (resp. \eqref{defn-vis-sup-1}).
\end{rmk}
 
 \subsection{A stability result}
 
 The following result is about the stability of the defined viscosity (semi)solutions.
 
 \begin{thm}\label{thm-stability}
 Let Assumption $\textbf{(A1)}$ hold. For each $\eps\in(0,1)$, let  $\mathbb H_{-}^{\eps}$ be a measurable function: 
 $$
 \mathbb H_{-}^{\eps}:\left(\Omega\times [0,T]\times\bR^{2d+(d+m)\times d+ m+1},\,\sP\otimes\cB(\bR^{2d+(d+m)\times d+ m+1})\right) \longrightarrow \left(\bR,\cB(\bR)\right).
 $$
 Suppose that there are a real $M>0$ and a sequence of
 nonnegative functions $\{\rho(\eps)\}_{\eps\in[0,1]} \subset L^{2}(\Omega,\sF_T,\bP;\bR)$ such that
 \begin{align}
 \esssup_{(\omega,t,x)} \left| (\mathbb H^{\eps}_{-}-\mathbb{H}_{-})(t,x,A,B,p,y,z)\right|&
 \leq \rho(\eps) (|A|+|B|+|p|+M\wedge |y|+  M\wedge |z|), \nonumber\\&
 \forall \,\, (A,B,p,y,z)\in   \bR^{d\times d}\times \bR^{m\times d} \times   \bR^d\times \bR \times \bR^m, \label{apprx-eps}
 \end{align}
 and $\lim_{\eps\rightarrow 0^+} \rho(\eps) =0$, a.s.
 Assume that for each $\eps\in(0,1)$, $u^{\eps}$ is a viscosity  $K_0$-subsolution (resp., $K_0$-supersolution) of BSPDE \eqref{SHJBI-low} associated to the nonlinear (generator) function $\mathbb H_-^{\eps}$ for some $K_0\geq 0$. If $\lim_{\eps\rightarrow 0^+}\|u^{\eps}-u\|_{\cS^2(C_b(\bR^d))}=0$, then $u$ must be a $K_0$-subsolution (resp., $K_0$-supersolution) of BSPDE \eqref{SHJBI-low} associated to the nonlinear (generator) function $\bH_-$.
  \end{thm}
  
 \begin{rmk}
  A similar stability result holds for the viscosity (semi)solutions of BSPDE \eqref{SHJBI-up}. 
  \end{rmk}
  
 \begin{proof}
 We only need to give the proof for the case of viscosity subsolution of \eqref{SHJBI-low} with the nonlinear generator $\bH_-$, as the proofs follow similarly for the other cases. 
 In what follows, we set 
 $$
 \alpha(\eps)=\sup_{t\in[0,T]} \|(u^{\eps}-u)(t,\cdot)\|_{C_b(\bR^d)}.
 $$
 Obviously, it holds that $E[\alpha(\eps)] \leq  \|u^{\eps}-u\|_{\cS^2(C_b(\bR^d))}$.

 Let $\phi\in  \underline{\cG}u(\tau,\xi;\Omega_{\tau},K) $ for some $K\geq K_0$, $\tau\in  \mathcal T^0$, $\xi\in L^0(\Omega_{\tau},\sF_{\tau};\bR^d)$, and $\Omega_{\tau}\in\sF_{\tau}$ with $\mathbb P(\Omega_{\tau})>0$. Without  any loss of generality, we assume $\tau=0$, $\xi=0$, and $\Omega_{\tau}=\Omega$, and  $\hat \tau =T$ is the stopping time associated to the relation $\phi\in  \underline{\cG}u(0,0;\Omega,K)$, i.e., we have
 $$
 (\phi-u)(0,0) =0
 =\essinf_{\bar\tau\in\mathcal T^{0}} \underline\cE_{0,\bar\tau}^K \left[\inf_{y\in \bR^d}
(\phi-u)(\bar\tau ,y)
\right] .
 $$
In view of Remark \ref{rmk-defn-timechange}, we assume that  $\mathbb{H}_{-}(t,x,A,B,p,y,z)$ is decreasing in $y$ w.l.o.g..

 For each $\delta,\tilde \delta\in(0, 1\wedge T)$,  set $\phi^{\delta,\tilde\delta}(t,x)=\phi(t,x)+\delta t+\tilde \delta g(x)$ for $(t,x)\in[0,T]\times\bR^d$, where $g$ is the differentiable nonnegative convex function defined in \eqref{g-defined}-\eqref{h-linear-growth}. It follows that for each $\bar\tau\in \mathcal T^0_+$, there holds
 \begin{align*}
  (\phi^{\delta,\tilde\delta}-u)(0,0)=
  (\phi-u)(0,0) =0
 \leq \underline\cE_{0,\bar\tau}^K \left[\inf_{y\in \bR^d}
(\phi-u)(\delta\wedge\bar\tau  ,y)
\right] <  \underline\cE_{0,\bar\tau}^K \left[\min_{y\in \bR^d}
(\phi^{\delta,\tilde\delta}-u)(\delta\wedge \bar\tau ,y)
\right].
 \end{align*}
 As $\lim_{\eps\rightarrow 0^+}\|u^{\eps}-u\|_{\cS^2(C_b(\bR^d))}=0$, there exists $\eps^{\delta,\tilde \delta}\in(0,\delta)$ such that for each $\eps\in (0,\eps^{\delta,\tilde \delta}]$, it holds that
  \begin{align}
  (\phi^{\delta,\tilde\delta}-u^{\eps})(0,0)
   <   \underline\cE_{0,\delta}^K \left[\min_{y\in \bR^d}
(\phi^{\delta,\tilde\delta}-u^{\eps})(\delta ,y)
\right] \text{ and }\ 
  \|u^{\eps}-u\|_{\cS^2(C_b(\bR^d))}^{1/4} \leq \delta.
   \label{ineq-0}
 \end{align}
 For each $s\in[0,T]$, choose an $\sF_s$-measurable variable $\xi_s$ such that  
\begin{align*}
 (\phi^{\delta,\tilde\delta}-u^{\eps})(s,\xi_s)
 = \min_{x\in\bR^d}(\phi^{\delta,\tilde\delta}-u^{\eps})(s,x),
\end{align*}  
where the existence of the minimum points is ensured by the linear growth properties of function $g$ in \eqref{h-linear-growth} as well as by the measurable selection theorem \ref{thm-MS}. Further, due to the linear growth of function $g$ in \eqref{h-linear-growth}, basic calculations yield that 
\begin{align}
\sup_s |\xi_s| \leq 3+\frac{1}{\tilde\delta} \cdot \left\{ \sup_{t\in[0,T]} \left( \|\phi(t)\|_{C_b(\bR^d)}+\|u^{\eps}(t)\|_{C_b(\bR^d)}  \right) +\delta T\right\}, \text{ a.s.} \label{bd-xi}
\end{align}

For $t\in[0,T]$, set 
$$Y_t = (\phi^{\delta,\tilde\delta}-u^{\eps})(t\wedge\delta ,\xi_{t\wedge\delta })\quad \mbox{and}\quad Z_t= \inf_{s\in\mathcal T^t} \underline\cE_{t,s}^K\left[ Y_s\right] .$$
 As $\phi,u^{\eps}\in \cS^2(C_b(\bR^d))$, it is easy to check the time-continuity of the process $(Y_s)_{s\in[0,T]}$; moreover, we have $Y\in \cS^2(\bR)$. Define $\tau_*=\inf\{s\geq 0: Y_s=Z_s\}$, and set $\Omega_{\tau_*}=\{\tau_*<\delta \}$. In view of the relation \eqref{ineq-0} and the optimal stopping theory (see \cite{El_Karoui-reflec-1997}), we have 
 \begin{align}
  \underline\cE_{0,\delta}^K [Y_{\delta }] > (\phi^{\delta,\tilde\delta}-u^{\eps})(0,0) \geq Y_0\geq Z_0= E[Y_{\tau_*}]=E[Z_{\tau_*}] , \label{ineq-tau}
 \end{align}
 and thus $\bP(\Omega_{\tau_*})>0$. In fact, combining \eqref{ineq-tau} and the following two observations
$$Y_0\leq (\phi^{\delta,\tilde\delta}-u^{\eps})(0,0) \leq E[\alpha(\eps)]$$ and 
$$
 \underline\cE_{0,\tau_*}^K [Y_{\tau_*}] \geq 
 \underline\cE_{0,\tau_*}^K\left[
\delta \tau_* -\alpha(\eps)
+  \inf_{y\in \bR^d}
(\phi-u)(\tau_* ,y)
\right]
\geq 
 \underline\cE_{0,\tau_*}^K 
 \left[
\delta \tau_* -\alpha(\eps)
\right] ,
$$
we have $ \underline\cE_{0,\tau_*}^K \left[
\delta \tau_* -\alpha(\eps)
\right] \leq E[\alpha(\eps)]$ and thus, $\underline\cE_{0,\tau_*}^K \left[
\delta \tau_* \right] \leq 2 E[\alpha(\eps)]$.
We note that $Y$, $Z$, $\tau_*$, and $\Omega_{\tau_*}$ are depending on $\eps$, $\delta$, and $\tilde \delta$.
Recalling the representation \eqref{representation-l} and relation \eqref{est-inf-dominate}, we have
\begin{align}
E[\sqrt{\tau_*}] 
&\leq \sqrt{ \underline\cE_{0,\tau_*}^K \left[
  \tau_* \right] \cdot e^{(K+2)KT}}.  \label{est-tau-star}
\end{align}
 Then, it holds that
 \begin{equation*}\begin{split}
 1-\bP(\Omega_{\tau_*})
 &\leq \frac{E[\sqrt{\tau_*}]}{\sqrt{\delta}} 
 \leq \frac{\sqrt{ \underline\cE_{0,\tau_*}^K \left[
 \tau_* \right] \cdot e^{(K+2)KT} }}{\sqrt{\delta}}
\\& \leq\frac{\sqrt{2E[\alpha(\eps)]e^{(K+2)KT}}}{{\delta}}
 \leq \frac{\sqrt{2 e^{(K+2)KT} \|u^{\eps}-u\|_{\cS^2(C_b(\bR^d))}}}{{\delta}}, 
 \end{split} \end{equation*}
 which converges to $0$ as $\delta\rightarrow 0^+$ as we take 
 $\delta\geq  \|u^{\eps}-u\|_{\cS^2(C_b(\bR^d))}^{1/4}$ in \eqref{ineq-0}, i.e., 
 $\lim_{\delta\rightarrow 0^+}\bP(\Omega_{\tau_*}) = \lim_{\delta\rightarrow 0^+}\bP(\tau_*<\delta) =1$. Recall that  $\eps<\delta$ and $\bP(\tau_*\leq \delta)=1$.   On the other hand, it is easy to see that  
 \begin{align}
 \lim_{\delta\rightarrow 0^+} Y_{\tau_*}
 &=\lim_{\delta\rightarrow 0^+} (\phi^{\delta,\tilde\delta}-u^{\eps})(\tau_*,\xi_{\tau_*})
 \nonumber\\
&= \lim_{\delta\rightarrow 0^+} 
\min_{x\in\bR^d}(\phi-u^{\eps})(\tau_*,x)+\tilde\delta g(x) +\delta\tau_*  \nonumber\\
&
=\min_{x\in\bR^d}(\phi-u)(0,x)+\tilde\delta g(x) 
= (\phi-u)(0,0)+\tilde\delta g(0) 
=0, \text{ a.s.,}\label{eq-Y-lim}
\end{align}
 which by the definition of function $g$ in \eqref{g-defined} and \eqref{h-linear-growth} also implies that $\lim_{\delta\rightarrow 0^+}\xi_{\tau_*} \rightarrow 0$ a.s.
 
 Define
 \begin{equation*}\begin{split}
 \overline\phi^{\delta,\tilde\delta} (t,x)
 =&\,\phi^{\delta,\tilde\delta} (t,x) + E_{\sF_t}\left[  \overline\cE^K_{\tau_*,t\vee \tau_*} \left[
 u^{\eps}(\tau_*,\xi_{\tau_*})
 -\phi^{\delta,\tilde\delta} (\tau_*,\xi_{\tau_*})
 \right] \right]
 \\=&\, \phi^{\delta,\tilde\delta} (t,x)+
 E_{\sF_t} \bigg[ (
 u^{\eps}(\tau_*,\xi_{\tau_*}) - \phi^{\delta,\tilde\delta} (\tau_*,\xi_{\tau_*}) )^+    e^{K(t\vee \tau_*-\tau_*)} \\&\hspace{3cm}
 -(u^{\eps}(\tau_*,\xi_{\tau_*})-\phi^{\delta,\tilde\delta} (\tau_*,\xi_{\tau_*}))^-   e^{-K(t\vee\tau_*-\tau_*)}\bigg].
 \end{split}\end{equation*}
 Straightforwardly, we may obtain
 $$( \overline\phi^{\delta,\tilde\delta}-u^{\eps})(\tau_*,\xi_{\tau_*}) 1_{\Omega_{\tau_*}}
  =0
 =\essinf_{\bar\tau\in\mathcal T^{\tau_*}} \underline\cE_{\tau_*,\bar\tau}^K \left[\inf_{y\in \bR^d}
(\overline\phi^{\delta,\tilde\delta}-u^{\eps})(\bar\tau\wedge \delta ,y)
\right] 1_{\Omega_{\tau_*}} ,\quad \text{a.s.}$$
 As $u^{\eps}$ is a viscosity $K_0$-subsolution, by Lemma \ref{lem-linear-growth} the function $\overline\phi^{\delta,\tilde\delta}$ admits a truncated version (denoted by itself) in $\underline{\cG}u^{\eps}(\tau_*,\xi_{\tau_*};\Omega'_{\tau_*},K) $ for some $\Omega'_{\tau_*}\in\sF_{\tau_*}$ satisfying $\Omega'_{\tau_*} \subset \Omega_{\tau_*}$ $\bP(\Omega'_{\tau_*})>\frac{\bP(\Omega_{\tau_*})}{2}>0$. Notice that as $\bP(\tau_*\leq \delta)=1$ and $\lim_{\delta\rightarrow 0^+}\bP(\Omega_{\tau_*})=1$. Blumenthal's zero-one law implies that 
 $\lim_{\delta\rightarrow 0^+}\bP(\Omega'_{\tau_*})=1$.

 In view of Assumptions $\textbf{(A1)}$ and \eqref{apprx-eps}, relations \eqref{h-linear-growth} and \eqref{R-Lip-const} (with $\zeta^{\phi}\in \cS^4_{\text{loc}}([0,T); \bR)$), and Remark \ref{rmk-eq-def-vs}, we have a.e. on $\Omega'_{\tau_*}$,
 \begin{align*}
 0&
 \geq
 \text{ess}\!\liminf_{s\rightarrow \tau_*^+}
	\left\{ -\mathfrak{d}_{s}\overline\phi^{\delta,\tilde\delta}-\bH^{\eps}_-\left(D^2\overline\phi^{\delta,\tilde\delta},D\mathfrak{d}_{\omega}\overline\phi^{\delta,\tilde\delta}, D\overline\phi^{\delta,\tilde\delta}, \overline\phi^{\delta,\tilde\delta}, \mathfrak{d}_{\omega}\overline\phi^{\delta,\tilde\delta} 
	\right) \right\} (s,\xi_{\tau_*})
	\\
&=	
 \text{ess}\!\liminf_{s\rightarrow \tau_*^+}
	  \left\{ -\mathfrak{d}_{s}\phi 
	-\bH^{\eps}_-\Big(D^2\phi^{\delta,\tilde\delta},D\mathfrak{d}_{\omega}\phi , D\phi^{\delta,\tilde\delta}, \overline\phi^{\delta,\tilde\delta}, \mathfrak{d}_{\omega} \phi  
	\Big) \right\} (s,\xi_{\tau_*})
	  -\delta	-|(u^{\eps}- \phi^{\delta,\tilde\delta}) (\tau_*,\xi_{\tau_*})| K
	\\
&\geq 
\text{ess}\!\liminf_{s\rightarrow \tau_*^+} 
	  \left[
	 \left\{ -\mathfrak{d}_{s}\phi 
	-\bH_-\Big(D^2\phi^{\delta,\tilde\delta},D\mathfrak{d}_{\omega}\phi , D\phi^{\delta,\tilde\delta}, \overline\phi^{\delta,\tilde\delta}, \mathfrak{d}_{\omega} \phi  
	\Big) \right\}   
\right.
	\\
&
\quad \left.
	-\rho(\eps) \left(|D^2\phi^{\delta,\tilde\delta}|+|D\mathfrak{d}_{\omega}\phi |
	+M\wedge |\overline\phi^{\delta,\tilde\delta}|+|D\phi^{\delta,\tilde\delta}|+ M\wedge |\mathfrak{d}_{\omega} \phi|\right)\right](s,\xi_{\tau_*}) 
	-\delta -| Y_{\tau_*} | K 
	\\
&\geq 
\text{ess}\!\liminf_{s\rightarrow \tau_*^+} 
	  \left[ \left\{ -\mathfrak{d}_{s}\phi 
	-\bH_-\left(D^2\phi ,D\mathfrak{d}_{\omega}\phi , D\phi ,  \phi , \mathfrak{d}_{\omega} \phi  
	\right) \right\}  
\right.
	\\
&
\quad 
	-\rho(\eps) \left(|D^2\phi^{\delta,\tilde\delta}|+|D\mathfrak{d}_{\omega}\phi |
	+M\wedge |\overline\phi^{\delta,\tilde\delta}|+|D\phi^{\delta,\tilde\delta}|+ M\wedge |\mathfrak{d}_{\omega} \phi| \right) \\
& \quad \left.
- 2L\left( |D^2\phi^{\delta,\tilde\delta}- D^2\phi | 
+|D\phi^{\delta,\tilde\delta} - D\phi|
+  1\wedge ||\overline\phi^{\delta,\tilde\delta} - \phi|
 \right)\right](s,\xi_{\tau_*})  
 	-\delta -|Y_{\tau_*}| K
\\
&= 
\text{ess}\!\liminf_{s\rightarrow \tau_*^+} 
	   \left[\left\{ -\mathfrak{d}_{s}\phi 
	-\bH_-\left(D^2\phi ,D\mathfrak{d}_{\omega}\phi , D\phi ,  \phi , \mathfrak{d}_{\omega} \phi  
	\right) \right\} 
	-\rho(\eps) \left(M\wedge |\mathfrak{d}_{\omega} \phi |\right) 
	\right] (s,\xi_{\tau_*})
		 \\
&\quad
	-\rho(\eps)
	 \Big(|D^2\phi  + \tilde\delta D^2g |
	+|D\mathfrak{d}_{\omega}\phi  | + M\wedge |  u^{\eps}   |
	+|D\phi  +\tilde \delta Dg | \Big) (\tau_*,\xi_{\tau_*}) 
\\
&\quad
	-\delta -|Y_{\tau_*}| K
-  2L\left(|\tilde \delta D^2g(\xi_{\tau_*}) | 
+ |\tilde \delta Dg(\xi_{\tau_*}) | 
+1\wedge |  Y_{\tau_*} -\delta \tau_*-\tilde \delta g(\xi_{\tau_*})|  
 \right)
 \\
&\geq
\text{ess}\!\liminf_{s\rightarrow \tau_*^+} 
	  \left[\left\{ -\mathfrak{d}_{s}\phi 
	-\bH_-\left(D^2\phi ,D\mathfrak{d}_{\omega}\phi , D\phi ,  \phi , \mathfrak{d}_{\omega} \phi  
	\right) \right\} (s,0) -\esssup_{t\in[0,\delta]}|\zeta^{\phi}_t| \cdot |\xi_{\tau_*}| -\rho(\eps) M \right] 
	\\
&\quad  
-\rho(\eps) \left(\esssup_{s\in[0,\delta]} \Big(\|\phi(s)\|_{C_b^2(\bR^d)} +  \|D\mathfrak{d}_{\omega}\phi(s)\|_{C_b(\bR^d)}\Big) + 2\tilde\delta\alpha_0+ M\right )  
\\
&\quad 
-\delta - 4L\left(\tilde\delta \alpha_0 +1\wedge (|Y_{\tau_*}| + \delta \tau_* + \tilde\delta\alpha_0)\right)
-|Y_{\tau_*}  | K
,
 \end{align*}
 where $\alpha_0$ is from \eqref{h-linear-growth}.  Therefore, it holds that a.s.,
 \begin{align*}
 0&
 \geq
 \text{ess}\!\liminf_{s\rightarrow \tau_*^+} 
	 \left[ \left\{ -\mathfrak{d}_{s}\phi 
	-\bH_-\left(D^2\phi ,D\mathfrak{d}_{\omega}\phi , D\phi ,  \phi , \mathfrak{d}_{\omega} \phi  
	\right) \right\} (s,0) 1_{\Omega_{\tau'_*}}
	-    1_{\Omega'_{\tau_*}} |\xi_{\tau_*}| \cdot \esssup_{t\in[0,\delta]}|\zeta^{\phi}_t| \right]
	 \\
&\quad 
-1_{\Omega_{\tau'_*}}\rho(\eps) \left(\esssup_{s\in[0,\delta]} (\|\phi(s)\|_{C_b^2(\bR^d)} +  \|D\mathfrak{d}_{\omega}\phi(s)\|_{C_b(\bR^d)}) + 2\tilde\delta \alpha_0 + 2M\right )  
\\
&\quad
-\delta- 1_{\Omega'_{\tau_*}} 4(L+K)\left(2\tilde\delta \alpha_0 + \delta\tau_*+  |Y_{\tau_*}|\right).
 \end{align*}
 Let $\delta \rightarrow 0^+$ (implying $\eps\rightarrow 0^+$) and then $\tilde{\delta}\rightarrow 0^+$. A standard proof by contradiction gives
 $$
 0\geq  \text{ess}\!\liminf_{s\rightarrow 0+} 
	 \left\{ -\mathfrak{d}_{s}\phi 
	-\bH_-\left(D^2\phi ,D\mathfrak{d}_{\omega}\phi , D\phi ,  \phi , \mathfrak{d}_{\omega} \phi  
	\right) \right\} (s,0)  ,
 $$
 which together with the arbitrariness of $\phi$ and $K$ implies that $u$ is a viscosity $K_0$-subsolution of BSPDE \eqref{SHJBI-low} associated to the nonlinear (generator) function $\bH_{-}$.  
 \end{proof}
 \begin{rmk}
 In Theorem \ref{thm-stability}, the sequence of viscosity (semi)solutions $\{u^{\eps}\}$ are assumed to be  $K_0$-(semi)solutions with an identical number $K_0\geq 0$ so that we obtain the estimate like \eqref{est-tau-star} uniformly with respect to $\eps$ in the above proof. On the other hand, for $\tilde g= \Phi (\cdot), b^i(t,\cdot,\theta, \gamma),\sigma^{ij}(t,\cdot,\theta, \gamma), f(t,\cdot,y,z,\theta, \gamma)$, $1\leq i \leq d,\,1\leq j\leq m$, $(t,y,z,\theta,\gamma)\in[0,T]\times \bR\times\bR^m\times\Theta_0\times\Gamma_0$, we may use the mollifiers as in \eqref{I-approx} to define
 \begin{align*}
 \tilde g_{\eps}(x)
&= \int_{\bR^d}  \tilde g(\zeta) \rho\left(\frac{x-\zeta}{\eps}\right)\cdot \frac{1}{\eps^d}\,d\zeta,\quad (x, \eps)\in\bR^d\times (0,\infty),
\end{align*}
and for each $\eps>0$, corresponding to such defined coefficients $\Phi_{\eps}$, $b_{\eps}$, $\sigma_{\eps}$ and $f_{\eps}$, the Hamiltonian functions $\bH_{\pm}^{\eps}$ may be defined in the same way as $\bH_{\pm}$. Then the relation \eqref{apprx-eps} is satisfied with $\rho(\eps)=L\eps$ where $L$ is the constant in Assumption \textbf{(A1)}. We would also note that the approximating coefficients in Lemma \ref{lem-approx} do not lead to the associated Hamiltonian functions $\bH_{\pm}^{\eps}$ satisfying  \eqref{apprx-eps} because the approximations therein is not uniform in time $t$; see \eqref{approx-coeficients}.
 \end{rmk}



\section{Existence of viscosity solutions for stochastic HJBI equations}
This section is devoted to the existence of viscosity solutions.
\begin{thm}\label{existence}
Let {\bf (A1)} hold. Then the value function $V$ (resp., $U$) defined by \eqref{eq-value-func-low} (resp., \eqref{eq-value-func-up}) is a viscosity $L$-solution of the stochastic Hamilton-Jacobi-Bellmen-Isaacs equation \eqref{SHJBI-low} (resp., \eqref{SHJBI-up}).
\end{thm}

The approaches for the proof mix some BSDE techniques and the obtained dynamic programming principle. For some arbitrarily chosen but fixed test function $\phi\in \mathscr C_{\sF}^3$, we put 
\begin{equation}
\begin{split}
F&(s,x,y,z,\theta,\gamma)=\mathfrak{d}_s\phi(s,x)+tr\left(\frac{1}{2}\sigma\sigma'(s,x,\theta,\gamma)D^2\phi(s,x)+\sigma(s,x,\theta,\gamma)D\mathfrak{d}_\omega\phi(s,x)\right)\\&+b'(s,x,\theta,\gamma)D\phi(s,x)+f(s,x,y+\phi(s,x),z+\sigma'(s,x,\theta,\gamma)D\phi(s,x)+\mathfrak{d}_\omega\phi(s,x),\theta,\gamma)\,. 
\end{split}
\end{equation}
For each fixed $T_0\in(0,T)$, $0\leq \tau \leq T_0$, $\theta\in \Theta$, $\gamma\in\Gamma$, and $\delta\in (0,1)$, we consider the following BSDE defined on the interval $[\tau,\tau+\delta\wedge T_0] $: 
\begin{equation}\label{bsdeY1}
\left\{
\begin{split}
-dY_s^{1,\theta,\gamma}&=F(s,X_s^{\tau,\xi;\theta,\gamma},Y_s^{1,\theta,\gamma},Z_s^{1,\theta,\gamma},\theta_s,\gamma_s)ds-Z_s^{1,\theta,\gamma}dW_s\,,\\
Y_{\tau+\delta\wedge T_0}^{1,\theta,\gamma}&=0\,.
\end{split}
\right.
\end{equation}
It is easy to check that $F(s,X_s^{\tau,\xi;\theta,\gamma},y,z,\theta_s,\gamma_s)$ is uniformly Lipschitz in $(y,z)$ for each $\theta\in \Theta$, $\gamma\in\Gamma$, and $F(\cdot,\cdot,0,0,\cdot,\cdot)\in\mathcal{L}^2([0,T_0];\mathbb{R})$. Thus, BSDE \eqref{bsdeY1} has a unique solution by the standard BSDE theory (see \cite{Hu_2002,ParPeng_90} for instance). On the other hand, we would note that the function $F(s,x,y,z,\theta_s,\gamma_s)$ is not uniformly Lipschitz in $x$ but there exists some $\zeta^{\phi}\in \cL^4_{\text{loc}}([0,T);
\bR)$ (recalling \eqref{R-Lip-const}) such that for all $x,\bar x \in \bR^d, y\in\bR, z\in\bR^m, $ and all $(\theta,\gamma)\in\Theta\times \Gamma$,
\begin{align}
|F(t,x,y,z,\theta_t,\gamma_t)-F(t,\bar x,y,z,\theta_t,\gamma_t)|\leq \zeta^{\phi}_t |x-\bar x|, \quad\text{for a.e. }(
\omega,t)\in \Omega\times [0,T). \label{Lip-const}
\end{align}
\begin{lem}\label{relationshipYGphi}
For any stopping time $\tau$ and any $\xi\in L^2(\Omega,\sF_\tau;\mathbb{R}^d)$, for any $s\in[\tau,\tau+\delta\wedge T_0]$, we have the following relationship: for $0\leq \tau\leq T_0$,
\begin{equation}
Y_s^{1,\theta,\gamma}=G_{s,\tau+\delta\wedge T_0}^{\tau,\xi;\theta,\gamma}[\phi(\tau+\delta\wedge T_0,X_{\tau+\delta\wedge T_0}^{\tau,\xi;\theta,\gamma})]-\phi(s,X_s^{\tau,\xi;\theta,\gamma}), \quad \text{a.s.}
\end{equation}
\end{lem}
\begin{proof}
Recalling the definition of $G_{s,\tau+\delta\wedge T_0}^{\tau,\xi;\theta,\gamma}[\phi(\tau+\delta\wedge T_0,X_{\tau+\delta\wedge T_0}^{\tau,\xi;\theta,\gamma})]$, we only have to prove that $Y_s^{\tau,\xi;\theta,\gamma}-\phi(s,X_s^{\tau,\xi;\theta,\gamma})=Y_s^{1,\theta,\gamma}$. This can be obtained by applying the It\^o-Kunita formula (see Lemma 4.1 in \cite{qiu2017viscosity}) to $\phi(s,X_s^{\tau,\xi;\theta,\gamma})$. 
\end{proof}

Now we consider the following BSDE in which the driving process $X^{\tau,\xi;\theta,\gamma}$ is replaced by the initial value $\xi$: 
\begin{equation}\label{bsdeY2}
\left\{
\begin{split}
-dY_s^{2,\theta,\gamma}&=F(s,\xi,Y_s^{2,\theta,\gamma},Z_s^{2,\theta,\gamma},\theta_s,\gamma_s)ds-Z_s^{2,\theta,\gamma}dW_s,\quad s\in[\tau,\tau+\delta\wedge T_0),\\
Y_{\tau+\delta\wedge T_0}^{2,\theta,\gamma}&=0, 
\end{split}
\right.
\end{equation}
where $\theta\in\Theta$ and $\gamma\in\Gamma$. 
For the difference of $Y^{1,\theta,\gamma}$ and $Y^{2,\theta,\gamma}$ we have the following estimate whose proof is postponed to the Appendix.
\begin{lem}\label{lemdifferenceofY}
For each $\theta\in\Theta$ and $\gamma\in\Gamma$, we have 
\begin{equation}\label{differenceofY}
\big|Y_\tau^{1,\theta,\gamma}-Y_\tau^{2,\theta,\gamma}\big|\leq  \delta^{\frac{5}{4}} \cdot C(1+|\xi|) \left( E_{\sF_\tau}\left[ \int_0^{T_0}\big| \zeta^{\phi}_t\big|^4  dt  \right]\right)^{1/4},\quad \text{a.s.},
\end{equation}
where $C$ is independent of $\delta$, $T_0$, and the control processes $\theta$ and $\gamma$. 
\end{lem}

\begin{lem}\label{lemmaonY0}
Let $(Y^0,Z^0)$ be the solution of the following BSDE: 
\begin{equation}\label{Y0}
\left\{
\begin{split}&-dY^0_s=F_0(s,\xi,Y_s^0,Z_s^0)ds-Z_s^0dW_s,\quad s\in[\tau,\tau+\delta \wedge T_0),\\
&Y^0_{\tau+\delta \wedge T_0 }=0,
\end{split}
\right.
\end{equation}
where $\xi\in L^2(\Omega,\sF_\tau;\mathbb{R}^d)$ and $F_0$ is defined as 
\begin{equation}\label{F0}
F_0(s,x,y,z)=\esssup_{\theta\in\Theta_0}\essinf_{\gamma\in\Gamma_0}F(s,x,y,z,\theta,\gamma).
\end{equation}
Then,  
\begin{equation}
\esssup_{\theta\in\Theta}\essinf_{\gamma\in\Gamma}Y_\tau^{2,\theta,\gamma}=Y^0_\tau,\quad \text{a.s.}
\end{equation}
\end{lem}
\begin{proof}
It is obvious that $F_0(s,x,y,z)$ is Lipschitz in $(y,z)$ uniformly w.r.t. $(s,x)$, a.s. Hence, BSDE \eqref{Y0} admits a unique solution. We first introduce the function 
\begin{equation}\label{F1}
F_1(s,x,y,z,\theta)=\essinf_{\gamma\in\Gamma_0}F(s,x,y,z,\theta,\gamma),\ (s,x,y,z,\theta)\in[0,T_0]\times\bR^d\times\bR\times\bR^m\times\Theta_0,
\end{equation}
and consider the BSDE
\begin{equation}\label{Y3}\left\{\begin{split}
&-dY^{3,\theta}_s=F_1(s,\xi,Y_s^{3,\theta},Z_s^{3,\theta},\theta_s)ds-Z_s^{3,\theta}dW_s,\quad s\in[\tau,\tau+\delta \wedge T_0),\\&Y_{\tau+\delta \wedge T_0}^{3,\theta}=0,\end{split}\right.
\end{equation}
for $\theta\in\Theta$. Since, for every $\theta\in\Theta$, $F_1(s,x,y,z,\theta)$ is Lipschitz in $(y,z)$, the solution $(Y^{3,\theta},Z^{3,\theta})$ uniquely exists. Moreover, 
$$Y^{3,\theta}_\tau=\essinf_{\gamma\in\Gamma}Y_\tau^{2,\theta,\gamma},\ \text{a.s.},\ \mbox{for any}\ \theta\in\Theta.$$
Indeed, from the definition of $F_1$ and Proposition \ref{prop-BSDE-comp}(ii) (comparison theorem), we have 
$$Y^{3,\theta}_\tau\leq\essinf_{\gamma\in\Gamma}Y_\tau^{2,\theta,\gamma},\ \text{a.s.},\ \mbox{for any}\ \theta\in\Theta.$$
On the other hand, by Theorem \ref{thm-MS} there exists a measurable function $\hat{\mu}:\Omega\times [\tau,\tau+\delta\wedge T_0]\times\bR^d\times\bR\times\bR^m\times\Theta\rightarrow\Gamma$ such that 
$$F_1(s,x,y,z,\theta)=F(s,x,y,z,\theta,\hat{\mu}(s,x,y,z,\theta)),\ \ \mbox{for any}\ (s,x,y,z,\theta).$$
Put $$\tilde{\mu}_s:=\hat{\mu}(s,\xi,Y_s^{3,\theta},Z_s^{3,\theta},\theta_s),\quad s\in[\tau,\tau+\delta\wedge T_0],$$
and we observe that $\tilde{\mu}\in\Gamma$ and 
$$F_1(s,\xi,Y_s^{3,\theta},Z_s^{3,\theta},\theta_s)=F(s,\xi,Y_s^{3,\theta},Z_s^{3,\theta},\theta_s,\tilde{\mu}_s),\quad s\in[\tau,\tau+\delta \wedge T _0].$$
Then, the uniqueness of solution to the BSDE yields that $(Y^{3,\theta},Z^{3,\theta})=(Y^{2,\theta,\tilde{\mu}},Z^{2,\theta,\tilde{\mu}})$, a.s., and in particular, $Y^{3,\theta}_\tau=Y^{2,\theta,\tilde{\mu}}_\tau$, a.s. for any $\theta\in\Theta$. This further implies that 
$$Y^{3,\theta}_\tau=\essinf_{\gamma\in\Gamma}Y_\tau^{2,\theta,\gamma},\ \text{a.s.},\ \mbox{for any}\ \theta\in\Theta.$$
Finally, since $F_0(s,x,y,z)=\esssup_{\theta\in\Theta_0}F_1(s,x,y,z,\theta)$, a similar argument as above gives that 
$$Y^0_\tau=\esssup_{\theta\in\Theta}Y^{3,\theta}_\tau=\esssup_{\theta\in\Theta}\essinf_{\gamma\in\Gamma}Y_\tau^{2,\theta,\gamma},\quad \text{a.s.}$$
The proof is complete. 
\end{proof}

\begin{lem}\label{estimateYZ}
For each $\theta\in\Theta$ and $\gamma\in\Gamma$, there holds
\begin{equation*}
E_{\sF_\tau}\left[\int_\tau^{\tau+\delta\wedge T_0}\big|Y_s^{2,\theta,\gamma}\big|ds\right]+E_{\sF_\tau}\left[\int_\tau^{\tau+\delta\wedge T_0}\big|Z_s^{2,\theta,\gamma}\big|ds\right]\leq  \delta^{\frac{5}{4}} \cdot C 
\left( E_{\sF_\tau}\left[ \int_0^{T_0}\big| \zeta^{\phi}_t\big|^4 dt   \right]\right)^{1/4}
\text{a.s.,}
\end{equation*}
where $C$ is a constant independent of $\delta,\,T_0,$ and the controls $\theta$ and $\gamma$. Moreover, we have
\begin{equation*}
E_{\sF_\tau}\left[\int_\tau^{\tau+\delta\wedge T_0}\big|Y_s^{0}\big|ds\right]+E_{\sF_\tau}\left[\int_\tau^{\tau+\delta\wedge T_0}\big|Z_s^{0 }\big|ds\right]\leq  \delta^{\frac{5}{4}} \cdot C 
\left( E_{\sF_\tau}\left[\int_0^{T_0}\big| \zeta^{\phi}_t\big|^4dt\right]\right)^{1/4}
\text{a.s.,}
\end{equation*}
where $C$ is independent of $\delta$ and $T_0$.
\end{lem}
The proof of Lemma \ref{estimateYZ} may be found in the appendix and now we come to the proof of Theorem \ref{existence}. 

\begin{proof}[Proof of Theorem \ref{existence}]
\textbf{Step 1.}  We show that $V$ is a viscosity {\bf supersolution}.
 First, in view of Lemma \ref{reg-value-funct}, we have $V\in\mathcal{S}^\infty(C_b(\bR^d))$. And it is obvious that $V(T,x)=\Phi(x)$ for all $x\in\mathbb{R}^d$, a.s. Let $K\geq L$ with $L$ being the constant in Assumption $\textbf{(A1)}$. For each $\phi\in\overline{\mathcal{G}}V(\tau,\xi;\Omega_\tau,K)$ with $\tau\in\mathcal{T}^0$, $\Omega_\tau\in\sF_\tau$, $P(\Omega_\tau)>0$, and $\xi\in L^0(\Omega_\tau,\sF_\tau;\bR^d)$,  by Remark \ref{rmk-eq-def-vs}, it is sufficient to verify that there holds for almost all $\omega\in\Omega_\tau$, 
$$\text{ess}\!\liminf_{s\rightarrow\tau^+}\{\mathfrak{d}_{s}\phi+\bH_-(D^2\phi,D\mathfrak{d}_{\omega}\phi, D\phi, \phi, \mathfrak{d}_{\omega}\phi )\}(s,\xi)\leq 0,$$
i.e.,
$$\text{ess}\!\liminf_{s\rightarrow\tau^+}F_0(s,\xi,0,0)\leq 0.$$
Let $\hat \tau $ be the stopping time corresponding to the fact $\phi\in \overline{\mathcal{G}}V(\tau,\xi;\Omega_\tau,K)$. 
We may choose $T_0\in (0,T)$ such that $\bP(\{\tau<T_0\}\cap \Omega_{\tau})>0$.  W.l.o.g., we assume $\hat\tau =T_0$ and $\Omega_{\tau}= \{\tau<T_0\}=\Omega$.

Thanks to Theorem \ref{thm-DPP}, we have for $\delta\in (0,1)$,
$$\phi(\tau,\xi)=V(\tau,\xi)=\essinf_{\mu\in\mathcal{M}}\esssup_{\theta\in\Theta}G_{\tau,\tau+\delta\wedge T_0 }^{\tau,\xi;\theta,\mu(\theta)}\left[V(\tau+\delta\wedge T_0 ,X_{\tau+\delta \wedge T_0 }^{\tau,\xi;\theta,\mu(\theta)})\right],\quad \text{a.s.}$$
In view of the relation  \eqref{dominate-G} and $\phi\in \overline{\mathcal{G}}V(\tau,\xi;\Omega_\tau,K)$, we have
\begin{align}
&
G_{\tau,\tau+\delta\wedge T_0 }^{\tau,\xi;\theta,\mu(\theta)}\left[
\phi(\tau+\delta\wedge T_0 ,X_{\tau+\delta\wedge T_0 }^{\tau,\xi;\theta,\mu(\theta)})
\right] 
-G_{\tau,\tau+\delta\wedge T_0 }^{\tau,\xi;\theta,\mu(\theta)}\left[V(\tau+\delta\wedge T_0 ,X_{\tau+\delta \wedge T_0 }^{\tau,\xi;\theta,\mu(\theta)})\right]
\nonumber \\
\leq&\,
\overline\cE^L_{\tau,\tau+\delta\wedge T_0 } \left[
\phi(\tau+\delta\wedge T_0 ,X_{\tau+\delta\wedge T_0 }^{\tau,\xi;\theta,\mu(\theta)})
-
V(\tau+\delta\wedge T_0 ,X_{\tau+\delta \wedge T_0 }^{\tau,\xi;\theta,\mu(\theta)})
\right]
\nonumber \\
\leq&\,
\overline\cE^L_{\tau,\tau+\delta\wedge T_0 } \left[\sup_{y\in\bR^d}
\left( \phi(\tau+\delta\wedge T_0 ,y)
-
V(\tau+\delta\wedge T_0 , y)
\right)
\right]
\nonumber \\
\leq&\,
\overline\cE^K_{\tau,\tau+\delta\wedge T_0 } \left[\sup_{y\in\bR^d}
\left( \phi(\tau+\delta\wedge T_0 ,y)
-
V(\tau+\delta\wedge T_0 , y)
\right)
\right]
\nonumber \\
\leq&\,
\phi(\tau,\xi)
-V(\tau,\xi)
\nonumber \\
=&\,0,\quad \text{a.s.} \label{compr-V-1}
\end{align}
Therefore, it holds that
\begin{align}
\essinf_{\mu\in\mathcal{M}}\esssup_{\theta\in\Theta}\big\{G_{\tau,\tau+\delta\wedge T_0 }^{\tau,\xi;\theta,\mu(\theta)}[
\phi(\tau+\delta\wedge T_0 ,X_{\tau+\delta\wedge T_0 }^{\tau,\xi;\theta,\mu(\theta)})
]-\phi(\tau,\xi)\big\}\leq 0,\quad \text{a.s.} \label{compr-phi-1}
\end{align}
Thus, by Lemma \ref{relationshipYGphi}, it follows that
$$\essinf_{\mu\in\mathcal{M}}\esssup_{\theta\in\Theta}Y_\tau^{1,\theta,\mu(\theta)}\leq 0,\quad \text{a.s.},$$
and further, by Lemma \ref{lemdifferenceofY}, we have   
$$
\essinf_{\mu\in\mathcal{M}}\esssup_{\theta\in\Theta}Y_\tau^{2,\theta,\mu(\theta)}\leq \delta^{\frac{5}{4}} \cdot C(1+|\xi|) \left( E_{\sF_\tau}\left[ \int_0^{T_0}\big| \zeta^{\phi}_t\big|^4 dt   \right]\right)^{1/4} ,\quad \text{a.s.}
$$
Since $\essinf_{\gamma\in\Gamma}Y_\tau^{2,\theta,\gamma}\leq Y_\tau^{2,\theta,\mu(\theta)}$, we obtain   
\begin{align*}
\esssup_{\theta\in\Theta}\essinf_{\gamma\in\Gamma}Y_\tau^{2,\theta,\gamma}
&
\leq \essinf_{\mu\in\mathcal{M}}\esssup_{\theta\in\Theta}Y_\tau^{2,\theta,\mu(\theta)}
\\
&\leq \delta^{\frac{5}{4}} \cdot C(1+|\xi|) \left( E_{\sF_\tau}\left[ \int_0^{T_0}\big| \zeta^{\phi}_t\big|^4 dt   \right]\right)^{1/4},\quad \text{a.s.,}
\end{align*}
and Lemma \ref{lemmaonY0} implies  that
\begin{equation}\label{EstimateY0}
Y^0_\tau
	\leq 
		\delta^{\frac{5}{4}} \cdot C(1+|\xi|) \left( E_{\sF_\tau}\left[ \int_0^{T_0}\big| \zeta^{\phi}_t\big|^4 dt   \right]\right)^{1/4},\quad \text{a.s.,}
\end{equation} 
where $(Y^0,Z^0)$ is the solution to BSDE \eqref{Y0}. Combining \eqref{EstimateY0} and the following relation $$Y^0_\tau=E_{\sF_\tau}\left[\int_\tau^{\tau+\delta \wedge T_0}F_0(s,\xi,Y_s^0,Z_s^0)\,ds\right],$$
we have  
\begin{align}
\text{ess}\!\liminf_{\delta\rightarrow 0^+}\frac{1}{\delta}
E_{\sF_\tau}
	\left[\int_\tau^{\tau+\delta \wedge T_0}
F_0(s,\xi,Y_s^0,Z_s^0)\,ds \right]
&\leq
	\lim_{\delta\rightarrow 0^+}\delta^{\frac{1}{4}} \cdot C(1+|\xi|) \left( E_{\sF_\tau}\left[ \int_0^{T_0}\big| \zeta^{\phi}_t\big|^4 dt   \right]\right)^{1/4}
		\nonumber\\
	&
		=0, \quad\text{a.s.} \label{lim-0}
\end{align}
Then, one can easily deduce that 
\begin{align}
\text{ess}\!\liminf_{s\rightarrow\tau^+} F_0(s,\xi,0,0)\leq 0,\quad \text{a.s.} 
\label{eqn-relation-supersolution}
\end{align}
If this is not true, there exists a $\Omega'_{\tau} \in \sF_{\tau}$, $\epsilon>0$, and $\delta_0\in(0,1)$, such that   $\mathbb P(\Omega'_{\tau})>0$ and for all $\delta\in (0,\delta_0]$,
$$\essinf_{s\in [\tau,\tau+\delta \wedge T_0 ]} F_0(s,\xi,0,0)
>\epsilon,\quad \text{a.e. on } \Omega'_{\tau}.$$
Then, by Lipschitz condition and Lemma \ref{estimateYZ}, we have a.e. on $\Omega'_{\tau}$,
\begin{equation*}\begin{split}
&\frac{1}{\delta}E_{\sF_\tau}\left[
\int_\tau^{\tau+\delta \wedge T_0 }F_0(s,\xi,Y^0_s,Z^0_s)\,ds \right]
\\
&> {\epsilon}\cdot E_{\sF_\tau} \left[ 
\frac{1}{\delta}\cdot\big|((\tau+\delta)\wedge T_0)-\tau\big| \right]
-\frac{C}{\delta}\cdot \left\{ E_{\sF_\tau}\left[\int_\tau^{\tau+\delta \wedge T_0 }|Y_s^0|ds\right]
+E_{\sF_\tau}\left[ \int_\tau^{\tau+\delta\wedge T_0}|Z_s^0|ds\right]\right\}
\\
&
	>{\epsilon}\cdot E_{\sF_\tau}\left[  
	\frac{1}{\delta}\cdot\big|((\tau+\delta)\wedge T_0)-\tau\big|\right]
	- \delta^{\frac{1}{4}} \cdot C \cdot\left( E_{\sF_\tau}\left[ \int_0^{T_0}\big| \zeta^{\phi}_t\big|^4 dt    \right]\right)^{1/4},
\end{split}\end{equation*}
which leads to a contradiction with \eqref{lim-0} as $\delta$ tends to zero. Therefore, we have relation \eqref{eqn-relation-supersolution} hold which together with the arbitrariness of $\phi$ implies that the function $V$ is a viscosity {\bf supersolution} of BSPDE \eqref{SHJBI-low}. 

\textbf{Step 2.} We prove that $V$ is a viscosity {\bf subsolution}: for each $\phi\in\underline{\mathcal{G}}V(\tau,\xi;\Omega_\tau,K)$ with $K\geq L$, $\tau\in\mathcal{T}^0$, $\Omega_\tau\in\sF_\tau$, $P(\Omega_\tau)>0$, and $\xi\in L^0(\Omega_\tau,\sF_\tau;\bR^d)$, 
there holds for almost all $\omega\in\Omega_\tau$, 
\begin{align}
\text{ess}\!\limsup_{s\rightarrow\tau^+} 
\{\mathfrak{d}_{s}\phi+\bH_-(D^2\phi,D\mathfrak{d}_{\omega}\phi, D\phi, \phi, \mathfrak{d}_{\omega}\phi )\}(s,\xi)\geq 0. \label{ineq-subvis}
\end{align}
 Let $\hat \tau $ be the stopping time corresponding to $\phi\in \underline{\mathcal{G}}V(\tau,\xi;\Omega_\tau,K)$.   
 
Let us suppose that the relation \eqref{ineq-subvis}  is not true. 
Then there exists some $k>0$, $\Omega'_{\tau} \in \sF_{\tau}$, and $\delta_0\in (0,1)$, such that $\Omega'_{\tau} \subset \Omega_{\tau}$, $\mathbb P(\Omega'_{\tau})>0$, and there holds a.e. on $\Omega_{\tau}'$,
\begin{equation}
\esssup_{s\in [\tau,\tau+\delta_0 \wedge T]}    F_0(s,\xi,0,0)  
= \esssup_{s\in [\tau,\tau+\delta_0\wedge T]} \left[
\esssup_{\theta\in\Theta}\essinf_{\gamma\in\Gamma}
F(s,\xi,0,0,\theta,\gamma)
\right]
\leq -k<0.
\end{equation}
Again,  we may choose $T_0\in (0,T)$ such that $\bP(\{\tau<T_0\}\cap \Omega'_{\tau})>0$.  W.l.o.g., we assume $\hat\tau =T_0$ and $\Omega_{\tau}=\Omega_{\tau}'=\Omega$. The measurable selection theorem (see Theorem \ref{thm-MS}) allows us to find a measurable function $\psi: \Theta\rightarrow\Gamma$ such that
\begin{align}
F(s,\xi,0,0,\theta,\psi(\theta))
\leq -\frac{k}{2},\quad \text{a.e. on }  [\tau,\tau+\delta_0\wedge T_0],\ \ \mbox{for\ all\ }\theta\in\Theta.\label{F-ineq}
\end{align}
For each $\delta\in (0,\delta_0)$, by Theorem \ref{thm-DPP}, we have
$$
\phi(\tau,\xi)=V(\tau,\xi)=\essinf_{\mu\in\mathcal{M}}\esssup_{\theta\in\Theta}G_{\tau,\tau+\delta \wedge T_0 }^{\tau,\xi;\theta,\mu(\theta)}[V(\tau+\delta \wedge T_0,X_{\tau+\delta \wedge T_0}^{\tau,\xi;\theta,\mu(\theta)})], \text{ a.s.}
$$
Meanwhile, by the relation  \eqref{dominate-G} and $\phi\in \underline{\mathcal{G}}V(\tau,\xi;\Omega_\tau,K)$, there holds 
\begin{align}
&
G_{\tau,\tau+\delta\wedge T_0 }^{\tau,\xi;\theta,\mu(\theta)}\left[
\phi(\tau+\delta\wedge T_0 ,X_{\tau+\delta\wedge T_0 }^{\tau,\xi;\theta,\mu(\theta)})
\right] 
-G_{\tau,\tau+\delta\wedge T_0 }^{\tau,\xi;\theta,\mu(\theta)}\left[V(\tau+\delta\wedge T_0 ,X_{\tau+\delta \wedge T_0 }^{\tau,\xi;\theta,\mu(\theta)})\right]
\nonumber \\
\geq&\,
\underline\cE^L_{\tau,\tau+\delta\wedge T_0 } \left[
\phi(\tau+\delta\wedge T_0 ,X_{\tau+\delta\wedge T_0 }^{\tau,\xi;\theta,\mu(\theta)})
-
V(\tau+\delta\wedge T_0 ,X_{\tau+\delta \wedge T_0 }^{\tau,\xi;\theta,\mu(\theta)})
\right]
\nonumber \\
\geq&\,
\underline\cE^L_{\tau,\tau+\delta\wedge T_0 } \left[\inf_{y\in\bR^d}
\left( \phi(\tau+\delta\wedge T_0 ,y)
-
V(\tau+\delta\wedge T_0 , y)
\right)
\right]
\nonumber \\
\geq&\,
\underline\cE^K_{\tau,\tau+\delta\wedge T_0 } \left[\inf_{y\in\bR^d}
\left( \phi(\tau+\delta\wedge T_0 ,y)
-
V(\tau+\delta\wedge T_0 , y)
\right)
\right]
\nonumber \\
\geq&\, 
\phi(\tau,\xi)
-V(\tau,\xi)
\nonumber \\
=&\,0,\quad \text{a.s.} \label{compr-V-2}
\end{align}
Therefore, we have
\begin{align}
\essinf_{\mu\in\mathcal{M}}\esssup_{\theta\in\Theta}\big\{G_{\tau,\tau+\delta \wedge T_0}^{\tau,\xi;\theta,\mu(\theta)}[\phi(\tau+\delta \wedge T_0,X_{\tau+\delta \wedge T_0}^{\tau,\xi;\theta,\mu(\theta)})]-\phi(\tau,\xi)\big\}\geq 0,\quad \text{a.s.} \label{compr-phi-2}
\end{align}
Thus, by Lemma \ref{relationshipYGphi}, it follows that
$$\essinf_{\mu\in\mathcal{M}}\esssup_{\theta\in\Theta}Y_\tau^{1,\theta,\mu(\theta)}\geq 0, \text{a.s.,}$$
and in particular, for each $\psi\in \mathcal M$ there holds
$$\esssup_{\theta\in\Theta}Y_\tau^{1,\theta,\psi(\theta)}\geq 0,\quad \text{a.s.}$$
Given an arbitrary $\epsilon>0$, we can choose $\theta^\epsilon\in\Theta$ such that $Y_\tau^{1,\theta^\epsilon,\psi(\theta^\epsilon)}\geq-\epsilon\delta.$
From Lemma \ref{lemdifferenceofY}, we further have
\begin{equation}\label{Y2geqCdelta}
Y_\tau^{2,\theta^\epsilon,\psi(\theta^\epsilon)}\geq-\epsilon\delta - \delta^{\frac{5}{4}} \cdot C(1+|\xi|) \left( E_{\sF_\tau}\left[ \int_0^{T_0}\big| \zeta^{\phi}_t\big|^4 dt   \right]\right)^{1/4}
,\quad \text{a.s.}
\end{equation}
Notice that 
$$Y_\tau^{2,\theta^\epsilon,\psi(\theta^\epsilon)}=E_{\sF_\tau}\left[\int_\tau^{\tau+\delta\wedge T_0}F(s,\xi,Y_s^{2,\theta^\epsilon,\psi(\theta^\epsilon)},Z_s^{2,\theta^\epsilon,\psi(\theta^\epsilon)},\theta^\epsilon,\psi(\theta^\epsilon))ds\right].$$
This together with the Lipschitz continuity of $F$, relation \eqref{F-ineq}, and Lemma \ref{estimateYZ} indicates that
\begin{align}
Y_\tau^{2,\theta^\epsilon,\psi(\theta^\epsilon)}\leq&\,E_{\sF_\tau}\left[\int_\tau^{\tau+\delta \wedge T_0 }\left(L|Y_s^{2,\theta^\epsilon,\psi(\theta^\epsilon)}|+L|Z_s^{2,\theta^\epsilon,\psi(\theta^\epsilon)}|+F(s,\xi,0,0,\theta^\epsilon,\psi(\theta^\epsilon)\right)ds\right]
\nonumber \\
\leq&\,-\frac{1}{2}k\delta
+ \delta^{\frac{5}{4}} \cdot C(1+|\xi|) \left( E_{\sF_\tau}\left[ \int_0^{T_0}\big| \zeta^{\phi}_t\big|^4 dt   \right]\right)^{1/4}
,\qquad \text{a.s.} \label{Y2leqCdelta}
\end{align}
Combining \eqref{Y2geqCdelta} and \eqref{Y2leqCdelta} and dividing both sides by $\delta$, we have
 \begin{align*}
 &-\delta^{\frac{1}{4}} \cdot C(1+|\xi|) \left( E_{\sF_\tau}\left[ \int_0^{T_0}\big| \zeta^{\phi}_t\big|^4 dt   \right]\right)^{1/4}  	-\epsilon\\
 & 
 \leq \delta^{\frac{1}{4}} \cdot C(1+|\xi|) \left( E_{\sF_\tau}\left[ \int_0^{T_0}\big| \zeta^{\phi}_t\big|^4 dt   \right]\right)^{1/4}
 	-\frac{1}{2}k.
 \end{align*}
 Letting $\delta\downarrow0$ and then $\epsilon\downarrow0$, we obtain $k\leq0$, which incurs a contradiction. 
 
 Hence, the value function $V$  is a viscosity solution of the stochastic HJBI equation \eqref{SHJBI-low}. Analogously, we prove that the value function $U$  is a viscosity solution of the stochastic HJBI equation  \eqref{SHJBI-up}.
 \end{proof}

\begin{rmk}
In the above proof we have actually shown that  the value function $V$ (resp., $U$) is indeed a viscosity $L$-solution of the stochastic HJBI equation \eqref{SHJBI-low} (resp., \eqref{SHJBI-up}). On the other hand,  in \eqref{compr-V-1} and \eqref{compr-V-2}, one may see that the sublinear functionals $\overline\cE^K$ and $\underline\cE^K$ are used to reach \eqref{compr-phi-1} and \eqref{compr-phi-2}; this explains why we employ the sublinear functionals $\overline\cE^K$ and $\underline\cE^K$ for test functions in $\underline{\mathcal{G}}u(\tau,\xi;\Omega_{\tau},K)$ and $\overline{\mathcal{G}}u(\tau,\xi;\Omega_{\tau},K)$ in the definition of viscosity solutions. 
\end{rmk}

\section{Uniqueness of the viscosity solution}

\subsection{A comparison theorem}
 Under Assumption $\textbf{(A1)}$, we have proven that the value function is a viscosity solution; however, due to the non-convexity of the games, the value function cannot be expected to be the minimal or maximal one among all the viscosity solution candidates  as in  \cite{qiu2017viscosity}. Instead, the proof herein is based on a comparison result, which is strategically different from  \cite{qiu2017viscosity}.

In Definition \ref{defn-viscosity}, we use the test functions from $\mathscr C_{\sF}^3$ that are logically finer than those in \cite{qiu2017viscosity}, to overcome the difficulties from the non-convexity  of the game and the nonlinear dependence of function $f$ on unknown variables. In principle, the stronger regularity of test functions increases the difficulties for the uniqueness. To reduce such an impact, we adopt some approximations and introduce the following function space.

\begin{defn}\label{defn-testfunc-2}
For $u\in \cS^{2} (C_b(\bR^d))$, 
we say $u\in \mathscr C_{\sF}^{2,\text{Lip}}$ if it holds that:\\
(i) there is $L^u>0$ such that a.s., $|u(t,x)-u(t,\bar x)|\leq L^u |x-\bar x|$ for all $(t,x,\bar x)\in [0,T]\times \bR^d\times \bR^d$;\\
(ii) there exist $\bar\alpha\in (0,1)$ and a finite partition $0=\underline t_0<\underline t_1<\ldots<\underline t_n=T$ for some integer $n\geq 1$ such that on each subinterval $[\underline t_j,\underline t_{j+1}]$ for $j=0,\dots,n-1$, we have 
$Du \in\cS^4_{\text{loc}}([\underline t_j,\underline t_{j+1});C_b^{1+\bar\alpha}(\bR^d))$, 
\begin{align*}
&\mathfrak{d}_tu \in 
\left\{h^1+h^2:\, h^1\in \cL^4_{\text{loc}}([\underline t_j,\underline t_{j+1}); \bR),\,h^2\in \cS^{\infty}_{\text{loc}}([\underline t_j,\underline t_{j+1});C_b^{\bar\alpha}(\bR^d)) \right\},
\\
&
\mathfrak{d}_{\omega}u \in 
	\left\{h^1+h^2:\, h^1\in \cL^{4,2}_{\text{loc}}([\underline t_j,\underline t_{j+1}); \bR),\,h^2\in \cS^{\infty}_{\text{loc}}([\underline t_j,\underline t_{j+1});C_b^{1+\bar\alpha}(\bR^d)) \right\},
\end{align*}
satisfying for each $x\in \bR^d$ and for all $0\leq r\leq T_0<\underline t_{j+1}$,
\begin{align*}
u(r,x)=u(T_0,x)-\int_r^{T_0} \mathfrak{d}_su(s,x)\,ds -\int_r^{T_0}\mathfrak{d}_{\omega}u(s,x)\,dW_s,\quad \text{a.s.}
\end{align*}
\end{defn}

\begin{thm}\label{weakcomparisonthm}
Let $\textbf{({A}1)}$ hold and $u$ be a viscosity subsolution (resp. supersolution) of BSPDE \eqref{SHJBI-low}  and $\phi\in\sC^{2,\text{Lip}}_{\sF}$ with 
$(u-\phi)^+\in\mathcal{S}^{\infty}(C_b(\mathbb{R}^d))$ (resp. $(\phi-u)^+\in\mathcal{S}^{\infty}(C_b(\mathbb{R}^d))$). Suppose   $\phi(T,x)\geq  (\text{resp.}\,\leq) \Phi(x)$, for all $x\in\bR^d$ a.s. and with probability 1,
\begin{align*}
&\text{ess}\liminf_{s\rightarrow t^+} \left\{-\mathfrak{d}_s\phi(s,y)
-\bH_{-}(D^2\phi,D\mathfrak{d}_\omega\phi,D\phi,\phi,\mathfrak{d}_\omega\phi)(s,y)\right\}\geq 0\\
(&\text{resp., }\text{ess}\limsup_{s\rightarrow t^+} \left\{-\mathfrak{d}_s\phi(s,y)-\bH_{-}(D^2\phi,D\mathfrak{d}_\omega\phi,D\phi,\phi,\mathfrak{d}_\omega\phi )(s,y)\right\}\leq 0),
\end{align*}
for all $(t,y)\in[0,T)\times\bR^d$. Then it holds that $u(t,x)\leq(resp.\ \geq)\phi(t,x)$, a.s. $\forall (t,x)\in[0,T]\times\bR^d$. 
\end{thm}
\begin{rmk}
Due to the  nonanticipativity constraints on the unknown variables, the classical variable-doubling techniques for deterministic nonlinear partial differential equations are inapplicable to the concerned BSPDEs \eqref{SHJBI-low} and \eqref{SHJBI-up}; in fact, when constructing auxiliary functions, we need not just take care of the regularity but also take into account the adaptedness/nonanticipativity.  For the regularity, notice that the function $\phi\in\sC^{2,\text{Lip}}_{\sF}$ has less regularity than the ones in $\sC^{3}_{\sF}$, so in order to construct a test function for viscosity semisolutions, the function $\phi$ will get mollified via identity  approximations.   
\end{rmk}
\begin{proof}[Proof of Theorem \ref{weakcomparisonthm}]
The proof by contradiction will be only given for the case of viscosity  $K$-subsolution for some $K\geq 0$, since for the case of viscosity supersolution, the proof will follow similarly.
Furthermore, letting $0=\underline t_0<\underline t_1<\ldots<\underline t_n=T$ be the partition associated to $\phi\in\sC^{2,\text{Lip}}_{\sF}$  as in Definition \ref{defn-testfunc-2}, we shall first prove the assertion on time subinterval $[\underline t_{n-1},T]$ and then recursively on the intervals $[\underline t_{n-2},\underline t_{n-1}],\cdots,[0,\underline t_1]$  we may complete the proof with similar arguments. In this way,  w.l.o.g. we may assume that $\Delta_n=T-\underline t_{n-1}$ is smaller than some desired strictly positive number $C(K)$   (depending only on $K$, see \eqref{Delta-n} below); otherwise, we may further divide $[\underline t_{n-1},T]$ into smaller subintervals.

The following proof is given on the subinterval $[\underline t_{n-1},T]$ and it is divided into two steps. 

\textbf{Step 1.} We figure out a test function in some set $\underline{\cG} u(\tau,\xi_{\tau};\Omega_{\tau},K)$ by assuming the opposite. 

Put
$$
\Lambda
	=\|  (u-\phi)^+ \|_{\mathcal S ^{\infty}([\underline t_{n-1},T];C_b(\mathbb{R}^d))} .$$
To the contrary, suppose $\Lambda >0$, i.e., with a positive probability, $u(t,\bar x)> \phi(t,\bar x)+\frac{2\Lambda}{3}$  at some point $(t,\bar x)\in [\underline t_{n-1},T)\times \bR^d$. Define $\phi_{\delta}$ as in \eqref{I-approx}:
\begin{align*}
\phi_{\delta}(s,x)
&= \int_{\bR^d}  \phi(s,z) \rho\left(\frac{x-z}{\delta}\right)\cdot \frac{1}{|\delta|^d}\,dz,\quad (s,x, \delta)\in [ \underline t_{n-1},T]\times \bR^d\times (0,\infty).
\end{align*}
Notice that a.s. $|\phi(s,x)-\phi_{\delta}(s, x)|\leq L^{\phi} \delta$ for all $(s,x, \delta)\in [\underline t_{n-1},T]\times \bR^d\times (0,\infty)$. 
 Accordingly, set
 $$\phi^{(\delta)}(s,x):=\phi_{\delta}(s,x) + \delta L^{\phi} ,\quad (s,x)\in[\underline t_{n-1},T]\times\bR^d.$$
 Then there exists $(\delta_0,\eps_0)\in (0,1)\times (0,\infty)$, such that $\mathbb P(u(t,\bar x)>\phi^{(\delta)}(t,\bar x))>\eps_0$ for all $\delta \in (0,\delta_0)$.    Moreover, there exists $\overline \Omega_t\in\sF_t$ such that $\bP( \overline\Omega_t)>0$ with $\overline\Omega_t \subset \{u(t,\bar x)-\phi^{(\delta)}(t,\bar x) >\frac{\Lambda}{2}  \}$ for all $\delta \in (0,\delta_0)$. Let $g$ be the differentiable nonnegative convex function defined in \eqref{g-defined}-\eqref{h-linear-growth}. Then  for each $\eps \in (0,1)$, there exists $\xi_t\in L^0(\overline\Omega_t,\sF_t;\bR^d)$ such that
$$
\alpha:=u(t,\xi_t)-\phi^{(\delta)}(t,\xi_t)-\eps g(\xi_t-\bar x )=\max_{x\in\bR^d} \{  u(t,x)-\phi^{(\delta)}(t,x) -\eps g(x-\bar x )\}\geq \frac{\Lambda}{2}\text{ a.e. on } \overline\Omega_t,
$$ 
where the existence and the measurability of $\xi_t$ follow from the measurable selection (see Theorem \ref{thm-MS}), the linear growth of function $g$ (see \eqref{h-linear-growth}) and the fact  that $(u-\phi^{(\delta)})^+\in\mathcal{S}^{\infty}([\underline t_{n-1},T];C_b(\mathbb{R}^d))$ (because $(u-\phi)^+\in\mathcal{S}^{\infty}(C_b(\mathbb{R}^d))$). 

Note that $t$ and $\overline\Omega_t$ may be chosen to be independent of $(\eps,\delta)$, and that a.s. $\phi^{(\delta)} (r,x) \geq \phi(r,x)$ for all $(r,x)\in [\underline t_{n-1},T]\times \bR^d$ with $\phi^{(\delta)}\in \mathscr C_{\sF}^3([\underline t_{n-1},T])$. W.l.o.g, we take $\overline\Omega_t=\Omega$ in what follows. It is obvious that $\alpha \leq \|  (u-\phi^{(\delta)})^+ \|_{\mathcal S ^{\infty}([\underline t_{n-1},T];C_b(\mathbb{R}^d))} \leq \Lambda $.

For each $s\in(t,T]$, choose an $\sF_s$-measurable variable $\xi_s$ such that  
\begin{align}
\left( u(s,\xi_s)-\phi^{(\delta)}(s,\xi_s) -\eps g( \xi_s-\bar x )\right)^+
&=\max_{x\in\bR^d}   \left(u(s,x)-\phi^{(\delta)}(s,x)-\eps g(x-\bar x )\right)^+ \notag\\
&=\max_{x\in\bR^d}   \left((u(s,x)-\phi^{(\delta)}(s,x))^+-\eps g(x-\bar x )\right)^+  . \label{eq-maxima}
\end{align}  
Set
\begin{align*}
Y_s
&=
	(u(s,\xi_s) -\phi^{(\delta)}(s,\xi_s)-\eps g(\xi_s-\bar x ))^++\frac{\alpha e^{-K(T-t)}(s-t)}{(2+\kappa)(T-t)};\\
Z_s
&= 
	\esssup_{\tau\in\cT^s} \overline\cE^{K}_{s,\tau} [Y_{\tau}],
\end{align*}
where $\kappa\geq 0$ is to be determined later. Then for $s\geq t$,
$$ Y_s\leq 2\Lambda \quad \text{and}\quad  Z_s \leq \Lambda  \left(1+e^{\Delta_n}\right), \quad \text{a.s.}$$
 As $(u-\phi^{(\delta)})^+\in\mathcal{S}^{\infty}([\underline t_{n-1},T];C_b(\mathbb{R}^d))$, it follows obviously the time-continuity of  
$$
	\max_{x\in\bR^d}   \left((u(s,x)-\phi^{(\delta)}(s,x))^+-\eps g(x-\bar x )\right)^+
$$ 
and thus that of
$\left( u(s,\xi_s)-\phi^{(\delta)}(s,\xi_s) -\eps g(\xi_s-\bar x )\right)^+$. Therefore, the process $(Y_s)_{t\leq s \leq T}$ has continuous trajectories. 
Define $\tau=\inf\{s\geq t:\, Y_s=Z_s\}$. Obviously, we have $\bP(\tau\leq T)=1$; further, in view of the optimal stopping theory, observing that
$$
\overline \cE^K_{t,T}Y_{T}=\frac{\alpha }{2+\kappa}
	<\alpha=Y_t\leq Z_t=\overline\cE^K_{t,\tau}Y_{\tau} =\overline\cE^K_{t,\tau}Z_{\tau},
$$
we have $\bP(\tau<T)>0$.
Indeed, this together with the relation \eqref{est-sup-dominate} yields that
\begin{align*}
\frac{\alpha^2e^{-2K(T-t)}}{(2+\kappa)^2} + \left(2\Lambda  \right)^2  \bP(\tau<T) 
\geq E\left[ E_{\sF_t} [|Y_{\tau}|^2]\right]
\geq \frac{1}{e^{K(K+2)\Delta_n}}   E \left [ \left|\overline \cE^K_{t,\tau} \left[ Y_{\tau}\right] \right|^2 \right]
\geq \frac{\alpha^2}{e^{K(K+2)\Delta_n}}  .
\end{align*}
Setting
$$\kappa=2 e^{\frac{K(K+2) \Delta_n}{2}},$$
and noticing that $\frac{\Lambda}{2} \leq \alpha\leq \Lambda  $, 
we obtain an estimate independent of $(\delta,\eps)$:
 \begin{align}
 \bP(\tau<T) \geq   
 \frac{3}{64\,e^{K(K+2) \Delta_n}   } 
 >0.
 \label{est-tau-T}
 \end{align}

Notice that 
\begin{align}
(u(\tau,\xi_{\tau}) -\phi^{(\delta)}(\tau,\xi_{\tau})-\eps g(\xi_{\tau}-\bar x))^++\frac{\alpha e^{-K(T-t)} (\tau-t)}{(2+\kappa)(T-t)}
=Z_{\tau}\geq \overline \cE^K_{{\tau},T}[Y_{T}]= \frac{\alpha e^{-K(\tau-t)}}{2+\kappa}. \label{relation-tau}
\end{align}
 Define 
$$\hat\tau=\inf\{s\geq\tau:\, (u(s,\xi_{s}) -\phi^{(\delta)}(s,\xi_{s})-\eps g(\xi_{s}-\bar x))^+\leq 0\}.$$ Obviously, $\tau\leq\hat\tau\leq T$.
 Put $\Omega_{\tau}=\{\tau<\hat\tau\}$. Then $\Omega_{\tau}\in \sF_{\tau}$. In view of relation \eqref{relation-tau} and the definition of $\hat\tau$, we have further $\Omega_{\tau}=\{\tau<\hat\tau\}=\{\tau<T\}$ with $\bP(\Omega_{\tau}) >0$ as in \eqref{est-tau-T}.
 

Set 
$$\gamma^{(\delta)}(s,x)=\phi^{(\delta)}(s,x)+\eps g(x-\bar x) 
-\frac{\alpha e^{-K(T-t)}(s-t)}{(2+\kappa)(T-t)}+E_{\sF_s}\left[Y_{\tau} e^{K(s-\tau)} \right]
.$$ 
 For each $\bar\tau\in\cT^{\tau}$, \footnote{Recall that $\cT^{\tau}$ denotes the set of stopping times $\zeta$ satifying $\tau\leq \zeta\leq T$ as defined in Section 4.} we have for almost all $\omega\in\Omega_{\tau}$, 
\begin{align*}
\left( \gamma^{(\delta)}-u\right)(\tau,\xi_{\tau})
=0=Y_{\tau}-Z_{\tau}
\leq 
Y_{\tau}- \overline\cE^K_{{\tau}, \bar\tau\wedge \hat{\tau}}\left[Y_{\bar\tau\wedge \hat{\tau}} \right] 
&=
 \underline\cE^K_{{\tau}, \bar\tau\wedge \hat{\tau}}\left[ Y_{\tau} e^{K( \bar\tau\wedge \hat{\tau}-\tau)}\right]+ \underline\cE^K_{{\tau}, \bar\tau\wedge \hat{\tau}}\left[-Y_{\bar\tau\wedge \hat{\tau}} \right] 
\\
&
\leq
 \underline\cE^K_{{\tau}, \bar\tau\wedge \hat{\tau}}\left[ Y_{\tau} e^{K( \bar\tau\wedge \hat{\tau}-\tau)} -Y_{\bar\tau\wedge \hat{\tau}} \right] 
\\
&\leq \underline\cE^K_{{\tau}, \bar\tau\wedge \hat{\tau} } \left[ \inf_{y\in\bR^d} (\gamma^{(\delta)}-u)(\bar\tau\wedge\hat\tau,y)   \right].
\end{align*}
Together with the arbitrariness of $\bar\tau$ and the construction of $\gamma^{(\delta)}$, this implies by Lemma \ref{lem-linear-growth} that $\gamma^{(\delta)}$ admits a truncated version (denoted by itself) lying in $  \underline{\cG} u(\tau,\xi_{\tau};\Omega'_{\tau},K)$ for some $\Omega'_{\tau}\subset \Omega_{\tau}$ with 
\begin{align}
\bP(\Omega'_{\tau})>\frac{\bP(\Omega_{\tau})}{2}>0. 
\label{Omega-prime}
\end{align}

\textbf{Step 2.} We prove that the  assertion holds on time interval $[\underline t_{n-1},T]$. There is a constant $C=C(K)>0$ (only depending $K$) such that  whenever $\Delta_n=T-\underline t_{n-1} \leq C(K)$ it holds that 
\begin{align}
\frac{ e^{-K\Delta_n}}{2(2+\kappa)\Delta_n}
=\frac{ e^{-K\Delta_n}}{2(2+2 e^{\frac{K(K+2) \Delta_n}{2}})\Delta_n}
 \geq 4K. \label{Delta-n}
\end{align}
We take $\Delta_n=C(K)$ in what follows.

As discussed in Remark \ref{rmk-defn-timechange}, w.l.o.g., we may assume that $\mathbb{H}_-(t,x,A,B,p,y,z)$ is decreasing in $y$. Also, we notice that $2\Lambda \geq Y_{\tau}\geq \frac{\alpha e^{-K(T-t)(\tau-t)}}{(2+\kappa)(T-t)}$ a.e. on $\Omega'_{\tau}$.
 As $u$ is a viscosity $K_0$-subsolution, it holds that for almost all $\omega\in\Omega'_{\tau}$, 
\begin{align}
0
&\geq
	 \text{ess}\!\liminf_{s\rightarrow \tau^+}
	  \left\{ -\mathfrak{d}_{s}\gamma^{(\delta)}(s,\xi_{\tau})
	-\bH_{-}(D^2\gamma^{(\delta)},D\mathfrak{d}_\omega\gamma^{(\delta)},D\gamma^{(\delta)},\gamma^{(\delta)},\mathfrak{d}_\omega\gamma^{(\delta)})(s,\xi_{\tau})\right\}
	\notag \\
&=
	\frac{\alpha e^{-K(T-t)}}{(2+\kappa)(T-t)} -KY_{\tau}
\notag \\
&\quad
	+ \text{ess}\!\liminf_{s\rightarrow \tau^+}
	 \Big\{  
	-\mathfrak{d}_{s} \phi^{(\delta)}(s,\xi_{\tau})
	-\bH_{-}(D^2\phi^{(\delta)}+\eps D^2g,D\mathfrak{d}_\omega\phi^{(\delta)},D\phi ^{(\delta)}+\eps Dg,
\nonumber\\
&
	\quad\quad\quad\quad\quad\phi^{(\delta)}+\eps g  + Y_{\tau}e^{K(s-\tau)}-\frac{\alpha e^{-K(T-t)} (s-t)}{(2+\kappa)(T-t)},\mathfrak{d}_\omega\phi^{(\delta)}\Big)(s,\xi_{\tau})\bigg\}
\notag \\
&\geq
	\frac{\Lambda e^{-K\Delta_n}}{2(2+\kappa)\Delta_n} -2K \Lambda
\notag \\
&\quad
	+ \text{ess}\!\liminf_{s\rightarrow \tau^+}
	  \Big\{  
	-\mathfrak{d}_{s} \phi^{(\delta)}(s,\xi_{\tau})
	-\bH_{-}(D^2\phi^{(\delta)}+\eps D^2g,D\mathfrak{d}_\omega\phi^{(\delta)},D\phi^{(\delta)} +\eps Dg,\nonumber\\
	&\quad\quad\quad\quad\quad
		\phi^{(\delta)},\mathfrak{d}_\omega\phi^{(\delta)})(s,\xi_{\tau})\Big\}
\label{eq-appl-monotone}\\
&\geq
	 \text{ess}\!\liminf_{s\rightarrow \tau^+}
	  \left\{  
	-\mathfrak{d}_{s}\phi^{(\delta)}(s,\xi_{\tau})
	-\bH_{-}(D^2\phi^{(\delta)},D\mathfrak{d}_\omega\phi^{(\delta)},D\phi^{(\delta)},\phi^{(\delta)},\mathfrak{d}_\omega\phi^{(\delta)})(s,\xi_{\tau})\right\}
	\notag\\
	&
	\quad\quad\quad + \frac{\Lambda e^{-K\Delta_n}}{4(2+\kappa)\Delta_n} - \eps \, C(L)\alpha_0
\label{eq-app-lip} 
\\
&
\geq 
		 \text{ess}\!\liminf_{s\rightarrow \tau^+}
	 \left\{  
	-\mathfrak{d}_{s}\phi^{(\delta)}(s,\xi_{\tau})
	-\bH_{-}(D^2\phi^{(\delta)},D\mathfrak{d}_\omega\phi^{(\delta)},D\phi^{(\delta)},\phi^{(\delta)},\mathfrak{d}_\omega\phi^{(\delta)})(s,\xi_{\tau})\right\}
	\notag\\
	&
	\quad\quad\quad + \frac{\kappa e^{-K\Delta_n}}{8(2+\kappa)\Delta_n}, \label{inequ-phi-delta}
\end{align}
where we set $\eps=\frac{1}{2}\wedge \frac{\Lambda e^{-K\Delta_n}}{8(2+\kappa)\Delta_n C(L)\alpha_0}$ and we note that the relation \eqref{eq-app-lip}  is based on the Lipchitz-continuity of $\mathbb{H}_-(t,x,A,B,p,y,z)$ with respect to $A$ and $p$ while in \eqref{eq-appl-monotone} we use the monotonicity of $\mathbb{H}_-(t,x,A,B,p,y,z)$ in $y$ instead of the Lipchitz-continuity due to the unboundedness of function $g$. Here, the constant $\alpha_0$ is from \eqref{h-linear-growth}. 

Set
\begin{align*}
\eta_s(\delta)
&= 
	-\mathfrak{d}_{s}\phi^{(\delta)}(s,\xi_{\tau})
	-\bH_{-}(D^2\phi^{(\delta)},D\mathfrak{d}_\omega\phi^{(\delta)},D\phi^{(\delta)},\phi^{(\delta)},\mathfrak{d}_\omega\phi^{(\delta)})(s,\xi_{\tau}) ,
	\\
\eta_s
&=	   
	-\mathfrak{d}_{s}\phi (s,\xi_{\tau})
	-\bH_{-}(D^2\phi ,D\mathfrak{d}_\omega\phi ,D\phi ,\phi ,\mathfrak{d}_\omega\phi )(s,\xi_{\tau}) .
\end{align*}
%
In view of the identity approximations, we straightforwardly check that
$$
(D\phi_{\delta},D^2\phi_{\delta},\mathfrak{d}_{s}\phi_{\delta},\mathfrak{d}_\omega\phi_{\delta}, D\mathfrak{d}_\omega\phi_{\delta})
=(D\phi ,D^2\phi ,\mathfrak{d}_{s}\phi ,\mathfrak{d}_\omega\phi , D\mathfrak{d}_\omega\phi)_{\delta},
$$ where we use the notation in \eqref{I-approx}.
As $\phi\in\sC^{2,\text{Lip}}_{\sF}$, applying the Lipschitz-continuity of $\bH_-(t,x,A,B,p,y,z)$ in $(A,B,p,y,z)$ gives that a.e. on $\Omega'_{\tau}$ for each $\tau<s<T$, 
\begin{align*}
|\eta_s(\delta)-\eta_s|&\leq (L+1) \cdot 
\Big( |D^2\phi -(D^2\phi)_{\delta} | + |D\phi -(D\phi_{\delta}) |
+ |\phi -\phi_{\delta} |\\
&\quad\quad\quad\quad\quad\quad\quad
+ |D\mathfrak{d}_\omega\phi- (D\mathfrak{d}_\omega\phi)_{\delta}|
+ |\mathfrak{d}_\omega\phi- (\mathfrak{d}_\omega\phi)_{\delta}|
+|\mathfrak{d}_{s}\phi-( \mathfrak{d}_{s}\phi)_{\delta} |
\Big)(s,\xi_{\tau})
\notag\\
&
\leq
C(s,\phi,L) |\delta|^{\bar\alpha}
, 
\notag
\end{align*}
where $\bar\alpha$ is the H\"older exponent in (ii) of Definition \ref{defn-testfunc-2} for $\phi\in\sC^{2,\text{Lip}}_{\sF}$ and the constant $C(s,\phi,L)$ may be chosen to be increasing in $s$ and depends on $s,L$, and quantities related to $\phi$, being independent of $\delta$. Therefore, it holds that 
\begin{align}
 \lim_{\delta\rightarrow 0^+} \text{ess}\!\limsup_{s\rightarrow \tau^+} |\eta_s(\delta)-\eta_s| 1_{\Omega'_{\tau}}=0,\text{ a.s.}\label{limlimsup}
\end{align}
However, recalling $\Omega'_{\tau}\subset \Omega_{\tau}=\{\tau<T\}$, \eqref{inequ-phi-delta}, and \eqref{limlimsup}, we have
\begin{align*}
&-  \frac{\Lambda e^{-K\Delta_n}}{8(2+\kappa)\Delta_n}
\cdot \text{ess}\!\limsup_{\delta\rightarrow 0^+} 
1_{\Omega'_{\tau}} 
\\
\geq&\,
  \text{ess}\!\liminf_{\delta\rightarrow 0^+}\text{ess}\!\liminf_{s\rightarrow \tau^+}\eta_s(\delta) 1_{\Omega'_{\tau}} 
\\
\geq&\,
\text{ess}\!\liminf_{\delta\rightarrow 0^+} 
\text{ess}\!\liminf_{s\rightarrow \tau^+}\eta_s 1_{\Omega'_{\tau}} 
- \lim_{\delta\rightarrow 0^+} 
\text{ess}\!\limsup_{s\rightarrow \tau^+} |\eta_s(\delta)-\eta_s| 1_{\Omega'_{\tau}}  
\\
\geq&\, 0,
\end{align*}
which implies that $\text{ess}\!\limsup_{\delta\rightarrow 0^+}  1_{\Omega'_{\tau}} \leq 0$. 
This incurs a contradiction because by  \eqref{est-tau-T} and \eqref{Omega-prime} it holds that
\begin{align*}
E\left[ \text{ess}\!\limsup_{\delta\rightarrow 0^+}  1_{\Omega'_{\tau}}\right]
\geq  \text{ess}\!\limsup_{\delta\rightarrow 0^+} E\left[  1_{\Omega'_{\tau}}\right]
\geq  
 \frac{3 }{128  \cdot e^{K(K+2) \Delta_n}   }  
>0.
\end{align*}
Hence, it holds that $u(t,x)\leq \phi(t,x)$, a.s. $\forall (t,x)\in[\underline t_{n-1},T]\times\bR^d$.

 Recursively on the intervals $[\underline t_{n-2},\underline t_{n-1}],\cdots,[0,\underline t_1]$, we finally prove that with probability 1, $u(t,x)\leq \phi(t,x)$, $\forall (t,x)\in[0,T]\times\bR^d$. The proof for the case of viscosity supersolution follows analogously. 
\end{proof}

\begin{rmk}
Due to the lack of spatial (global) integrability of $u$ (or $u$ possibly being nonzero at infinity), we introduce a penalty function $g$ in the above proof to ensure the existence of extreme points (for instance, in \eqref{eq-maxima}).
\end{rmk}

\subsection{Uniqueness of the viscosity solution}
We first discuss the uniqueness under an assumption allowing for possibly degenerate diffusion coefficient $\sigma$: 
 \medskip
 \\
\textbf{(A2)} the \textit{diffusion} coefficient $\sigma:[0,T] \rightarrow\bR^{d\times m}$  does not depend on $(\omega,x,\theta,\gamma)\in \Omega\times \bR^d\times \Theta_0\times\Gamma_0$. 

\vspace{2mm}

We note that under assumptions \textbf{(A1)} and \textbf{(A2)}, we have for $(t,x,A,B,p,y,z)\in [0,T]\times\bR^d\times\bR^{d\times d}\times \bR^{m\times d} \times   \bR^d\times \bR \times \bR^m$,
\begin{align*}
\mathbb{H}_-(t,x,A,B,p,y,z)
=&\,\text{tr}\left(\frac{1}{2}\sigma (t)\sigma'(t) A+\sigma(t) B\right)  
\\
&+  \esssup_{\theta\in\Theta_0}\essinf_{\gamma\in \Gamma_0} \bigg\{
       b'(t,x,\theta,\gamma)p 
       +f(t,x,y,z+\sigma'(t)p,\theta,\gamma)
                \bigg\}, \\
 \mathbb{H}_+(t,x,A,B,p,y,z)
=&\,\text{tr}\left(\frac{1}{2}\sigma(t) \sigma'(t) A+\sigma(t) B\right) \\
&+ \essinf_{\gamma\in \Gamma_0}\esssup_{\theta\in\Theta_0} \bigg\{
       b'(t,x,\theta,\gamma)p 
       +f(t,x,y,z+\sigma'(t)p,\theta,\gamma)
                \bigg\}.               
\end{align*}


\begin{thm}
Let assumptions $\textbf{({A}1)}$ and $\textbf{({A}2)}$ hold. The viscosity solution to  BSPDE \eqref{SHJBI-low} is unique in $\mathcal{S}^\infty(C_b(\mathbb{R}^d))$. 
\end{thm}
\begin{proof}
Define
{\small
\begin{align*}
\overline{\mathscr V}=\Big\{\phi\in\mathscr C_{\sF}^{2,\text{Lip}} &: \phi^-\in\mathcal{S}^\infty(C_b(\mathbb{R}^d)), \phi(T,x) \geq \Phi(x) ,\ \forall x\in\mathbb{R}^d , \text{a.s., and for each } (t,y)\in [0,T)\times\mathbb{R}^d,
\\ 
& \text{ess}\!\liminf_{s\rightarrow t^+} \left\{-\mathfrak{d}_s\phi-\bH_{-}(D^2\phi,D\mathfrak{d}_\omega\phi,D\phi,\phi,\mathfrak{d}_\omega\phi)\right\}(s,y)\geq0,
\text{ a.s.} \Big\},\\
\underline{\mathscr V}=\Big\{\phi\in\mathscr C_{\sF}^{2,\text{Lip}} &: \phi^+\in\mathcal{S}^\infty(C_b(\mathbb{R}^d)),  \Phi(x)\geq \phi(T,x),\ \forall x\in\mathbb{R}^d, \text{a.s., and for each } (t,y)\in [0,T)\times\mathbb{R}^d, 
\\ & \text{ess}\!\limsup_{s\rightarrow t^+}\left\{-\mathfrak{d}_s\phi-\bH_{-}(D^2\phi,D\mathfrak{d}_\omega\phi,D\phi,\phi,\mathfrak{d}_\omega\phi)\right\}(s,y)\leq0, \text{ a.s.} \Big\},
\end{align*}
}
and set $$\overline{u}=\essinf_{\phi\in\overline{\mathscr V}}\phi,\quad\underline{u}=\esssup_{\phi\in\underline{\mathscr V}}\phi.$$
Notice that for each $(\overline \phi,\,\underline\phi)\in \overline{\mathscr V} \times  \underline{\mathscr V}$, we have $\overline \phi^-\in \mathcal{S}^\infty(C_b(\mathbb{R}^d))$ and $\underline \phi^+\in \mathcal{S}^\infty(C_b(\mathbb{R}^d))$. For  each viscosity solution $u\in\mathcal{S}^\infty(C_b(\mathbb{R}^d))$, we have $(u-\overline \phi)^+\in \mathcal{S}^\infty(C_b(\mathbb{R}^d))$ and $(\underline \phi-u)^+\in \mathcal{S}^\infty(C_b(\mathbb{R}^d))$, and Theorem \ref{weakcomparisonthm} indicates that $\underline{u}\leq u\leq\overline{u}$. Therefore, for the uniqueness of viscosity solution, it is sufficient to verify $\underline{u}= V= \overline{u}$.

For each fixed $\eps\in(0,1)$, select $(\Phi^{\eps},\,f^{\eps},\,b^{\eps})$ and $(\Phi_N,f_N,b_N)$ as in Lemma \ref{lem-approx}; we do not need the approximations for $\sigma$ here. 

Let $(\Omega',\sF',\{\sF'_t\}_{t\geq 0}, \bP')$ be another complete filtered probability space which carries a $d-$dimensional standard Brownian motion $B=\{B_t\, :\, t\geq 0\}$ with $\{\sF'_t\}_{t\geq 0}$ generated by $B$ and augmented by all the $\bP'$-null sets in $\sF'$. Set
$$
(\bar\Omega,\bar\sF,\{\bar\sF_t\}_{t\geq 0},\bar\bP)=
(\Omega\times\Omega',\sF \otimes\sF',\{\sF_t\otimes\sF'_t \}_{t\geq 0},\bP \otimes\bP').
$$
Then $B$ and $W$ are independent on $(\bar\Omega,\bar\sF,\{\bar\sF_t\}_{t\geq 0},\bar\bP)$ and it is easy to see that all the theory established in previous sections still hold on the enlarged probability space.

Recalling the standard theory of BSDEs (see \cite{Hu_2002} for instance), let the  pairs  $(Y^{\eps},Z^{\eps})\in \cS^2_{\sF}(\bR)\times \cL^2_{\sF}(\bR^{m})$ and $(y,z)\in  \cS^2_{\sF'}(\bR)\times \cL^2_{\sF'}(\bR^{d})$   be the solutions  of backward SDEs
$$
Y_s^{\eps}=\Phi^{\eps}+
	\int_s^T\left(f^{\eps}_t+M b^{\eps}_t\right)\,dt
		-\int_s^TZ^{\eps}_t\,d W_t,
$$
 and 
 $$
 y_s=
 |B_T|+\int_{s}^T |B_t| \,dt-\int_s^T z_t\,dB_t,
 $$
 respectively, and for each $(s,x)\in[0,T)\times\bR^d$,  set
\begin{align*}
{V}^{\eps}(s,x)
&=\essinf_{\mu\in\cM}\esssup_{\theta\in\Theta}  G_{s,T}^{N,s,x;\theta,\mu(\theta)} \bigg[
	\Phi_N\left( W_{t_1},\cdots, W_{t_N},X^{s,x;\theta,\mu(\theta),N}_T\right)
		\bigg].
\end{align*}
Here, the value of constant $M$ is to be determined, the process $X^{s,x;\theta,\mu(\theta),N}$ satisfies SDE
\begin{equation*}\label{state-proces-contrl}
\left\{
\begin{split}
&dX_t^{s,x;\theta,\mu(\theta),N}=b_N( W_{t_1\wedge t},\cdots, W_{t_N\wedge t},t,X_t^{s,x;\theta,\mu(\theta),N},\theta_t,\mu(\theta)(t))dt\\&\hspace{3cm}+\sigma(t)\,dW_t+\delta_N\,dB_t,\ 
\ t\in[s,T]; \\
& X_s^{s,x;\theta,\mu(\theta),N}=x,
\end{split}
\right.
\end{equation*}
with $\delta_N$ being a positive constant, and we adopt the (backward) semigroup
 \begin{align*}
G_{t,T}^{N,s,x;\theta,\gamma}[\eta]:=\overline Y_t^{s,x;\theta,\gamma},\quad t\in [s,T],
\end{align*}
with $\overline Y^{s,x;\theta,\gamma}$ together with $\overline Z^{s,x;\theta,\gamma}$ and $\tilde Z^{s,x;\theta,\gamma}$ satisfies the following BSDE:
\begin{equation}\label{BSDE-semigroup-N}
  \left\{
  \begin{split}
  -d\overline Y_t^{s,x;\theta,\gamma}&=\,f_N\left( W_{t_1\wedge t},\cdots, W_{t_N\wedge t},t,X^{s,x;\theta,\mu(\theta),N}_t,\overline Y_t^{s,x;\theta,\gamma},\overline Z_{t}^{s,x;\theta,\gamma},\theta_t,\gamma_t\right)\,dt 
  \\
  &\quad -\overline Z_{t}^{s,x;\theta,\gamma}\,dW_t
  -\tilde Z_{t}^{s,x;\theta,\gamma}\,dB_t,\quad t\in[s, T];\\
 \overline Y_{T}^{s,x;\theta,\gamma}&=\eta.
    \end{split}
  \right.
\end{equation}
 
%
The theory of stochastic differential games  (see \cite{buckdahnLi-2008-SDG-HJBI}) yields that when $s\in[t_{N-1},T)$, we have $V^{\eps}(s,x)=\tilde V^{\eps}(s,x,  W_{t_1},\cdots,  W_{t_{N-1}}, W_s)$ with
\begin{align*}
&\tilde{V}^{\eps}(s,x, W_{t_1},\cdots,  W_{t_{N-1}},y)\\
=&\essinf_{\mu\in\cM}\esssup_{\theta\in\Theta} 
G_{s,T}^{N,s,x;\theta,\mu(\theta)}\left[   \Phi_N\left(  W_{t_1},\cdots,  W_{t_N},X^{s,x;\theta,\mu(\theta),N}_T\right)   \right] \Big |_{W_s=y}
\end{align*}
satisfying the following semilinear superparabolic HJBI equation: 
\begin{equation}\label{markoviantype}
\left\{\begin{split} 
-D_tu(t,x,y)&=
\frac{1}{2}tr(D^2_{yy}u(t,x,y))+\frac{\delta_N^2}{2}tr(D^2_{xx}u(t,x,y))  
\\
&\quad+  tr\Big(\frac{1}{2}\sigma\sigma'(t)D^2_{xx}u(t,x,y)+{\sigma}(t)D^2_{xy}u(t,x,y)\Big)
\\
&\quad
+ \esssup_{\theta\in\Theta_0}\essinf_{\gamma\in\Gamma_0}\Big\{
b'_N( W_{t_1},\cdots,  W_{t_{N-1}},y,t,x,\theta,\gamma)D_xu(t,x,y)
\\
&\quad+
f_N( W_{t_1},\cdots,  W_{t_{N-1}},y,t, x,,u(t,x,y),(D_yu+\sigma'D_xu)(t,x,y), \theta,\gamma)\Big\},
\\
&\quad\quad (t,x,y)\in[t_{N-1},T)\times\mathbb{R}^d\times\mathbb{R}^{m};
\\u(T,x,y)&=\Phi_N( W_{t_1},\cdots, W_{t_{N-1}},y,x),\quad (x,y)\in\mathbb{R}^d\times\mathbb{R}^{m},
\end{split}\right.
\end{equation}
and the theory of parabolic PDEs gives
$$\tilde{V}^\eps(\cdot,\cdot, W_{t_1},\cdots,  W_{t_{N-1}},\cdot)\in L^\infty\left(\Omega,{\sF}_{t_{N-1}}; C_b^{1+\frac{\bar{\alpha}}{2},2+\bar{\alpha}}([t_{N-1},T)\times\mathbb{R}^d)\cap C_b([t_{N-1},T]\times\mathbb{R}^d)\right),$$
for some $\bar{\alpha}\in(0,1)$, where the time-space H\"{o}lder space $C_b^{1+\frac{\bar{\alpha}}{2},2+\bar{\alpha}}([t_{N-1},T)\times\mathbb{R}^d)$ is defined as usual. 
We can make similar arguments on time interval $[t_{N-2},t_{N-1})$ taking the obtained $V^\eps(t_{N-1},x)$ as the terminal value and recursively on the intervals $[t_{N-3},t_{N-2}),\cdots,[0,t_1)$. Then, applying the It\^o-Kunita formula to $\tilde{V}^{\eps}(s,x, W_{t_1},\cdots,  W_{t_{N-1}},y)$ on $[t_{N-1},T]$ yields that 
 \begin{equation}\label{SHJBI-N}
  \left\{\begin{array}{l}
  \begin{split}
  -dV^{\eps}(t,x-\delta_N B_t)&= tr\Big(\frac{1}{2}\sigma(t)\sigma'(t)D^2_{xx}V^{\eps}(t,x-\delta_N B_t)
  \\
  &\quad\quad+{\sigma}(t)D^2_{xy}  \tilde{V}^{\eps}(t,x-\delta_NB_t,  W_{t_1},\cdots, W_{t_{N-1}},W_t)       \Big)dt
\\
+
& \esssup_{\theta\in\Theta_0}\essinf_{\gamma\in\Gamma_0}\Big\{
b'_N( W_{t_1},\cdots,  W_{t_{N-1}},W_t,t,x-\delta_NB_t,\theta,\gamma)D_xV^{\eps}(t,x-\delta_N B_t)
\\
+&
f_N( W_{t_1},\cdots,  W_{t_{N-1}},W_t,t,x-\delta_N B_t,V^{\eps},(D_y\tilde V^{\eps}+\sigma'D_xV^{\eps}), \theta,\gamma)\Big\}dt
\\
-&
D_{y} \tilde{V}^{\eps}(t,x-\delta_N B_t,  W_{t_1},\cdots, W_{t_{N-1}},W_t) \,d W_t +\delta_N D_{x} V^{\eps}(t,x-\delta_N B_t)\,dB_t,
\\
V^{\eps}(T,x-\delta_N B_T)&=\Phi_N( W_{t_1},\cdots, W_{t_{N-1}},W_T,x-\delta_N B_T).
    \end{split}
  \end{array}\right.
\end{equation}

In view of the approximations in Lemma \ref{lem-approx} and with an analogy to Lemma \ref{reg-value-funct}, there exists $\tilde{L}>0$ such that 
$$
\max_{(t,x)\in[0,T]\times \bR^d}\left\{ |DV^{\eps}(t,x)|   \right\}
\leq \tilde L,\,\,\,\text{a.s.}
$$
with $\tilde L$ being independent of $\eps$ and $N$. Set $M=\tilde L$ and
\begin{align*}
\overline{V}^{\eps}(s,x)
&=
	V^{\eps}(s,x-\delta_N B_s)+Y^{\eps}_s+\delta_N \bar M y_s ,\\
\underline{V}^{\eps}(s,x)
&=
	V^{\eps}(s,x-\delta_N B_s)-Y^{\eps}_s-\delta_N \bar M y_s,
\end{align*}
with $\bar M=4L(\tilde L+1)$ and $L$ the constant in  $\textbf{({A}1)}$.

By Assumption  $\textbf{({A}1)}$, we have
\begin{align*}
|b(t,x,\theta,\gamma)-b(t,x-\delta_N B_t,\theta,\gamma)|
+|f(t,x,y,p,\theta,\gamma)-f(t,x-\delta_N B_t,y,p,\theta,\gamma)| 
& \leq 2\delta_N L |B_t|,\\
|G(x)-G(x-\delta_NB_T)|
&\leq \delta_N L |B_T|.
\end{align*}
Also, by Remark \ref{rmk-defn-timechange}, we may, w.l.o.g., assume that $\mathbb{H}_-(t,x,A,B,p,y,z)$ is decreasing in $y$.
Then for $\overline V^{\eps}$ on $[t_{N-1},T)$, omitting the inputs for some involved functions, we have
\begin{align}
&-\mathfrak{d}_{t}\overline V^{\eps}-\bH_{-}(D^2\overline V^{\eps},D\mathfrak{d}_\omega\overline V^{\eps},D\overline V^{\eps},\overline V^{\eps},\mathfrak{d}_\omega\overline V^{\eps})
\nonumber\\
=&\,
	-\mathfrak{d}_{t}\overline V^{\eps}
	-tr\Big(\frac{1}{2}\sigma\sigma'D^2_{xx}\overline V^{\eps} +{\sigma}D_{x}\mathfrak{d}_\omega \overline V^{\eps}\Big)
	\notag   \\
&\,
	-\esssup_{\theta\in\Theta_0}\essinf_{\gamma\in \Gamma_0} \bigg\{
  		     (b_N)'D\overline V^{\eps}
	  +f^N+f^{\eps}+ \tilde L  b^{\eps}\nonumber+\delta_N\bar M |B_t|\\
&
	  \,
	  +   \left(b-b_N\right)' D\overline V^{\eps}       -b^{\eps} \tilde L
	  +f-f^N-f^{\eps}-\delta_N \bar M |B_t|
	   \bigg\}
\nonumber\\
\geq&\,
	-\mathfrak{d}_{t}\overline V^{\eps}
	-tr\Big(\frac{1}{2}\sigma\sigma'D^2_{xx}\overline V^{\eps} +{\sigma}D_{x}\mathfrak{d}_\omega \overline V^{\eps}\Big)
	\notag \\
&\,
	-\esssup_{\theta\in\Theta_0}\essinf_{\gamma\in \Gamma_0} \bigg\{
  		     (b_N)'D\overline V^{\eps}
	  +f^N+f^{\eps}+ \tilde L  b^{\eps}\nonumber+\delta_N\bar M |B_t|
	   \bigg\}
	 \label{est-A4}    \\
=&\, 0,
\end{align}
and it follows similarly on intervals $[t_{N-2},t_{N-1})$, $\dots$, $[0,t_1)$ that
$$
-\mathfrak{d}_{t}\overline V^{\eps}-\bH_{-}(D^2\overline V^{\eps},D\mathfrak{d}_\omega\overline V^{\eps},D\overline V^{\eps},\overline V^{\eps},\mathfrak{d}_\omega\overline V^{\eps})
		 \geq 0,
$$
which together with the obvious relation  $\overline V^{\eps}(T)=\Phi_N+\Phi^{\eps}+ \delta\bar M |B_T|  \geq \Phi $ indicates that $\overline V^{\eps}\in \overline {\mathscr V}$. Analogously, $\underline V^{\eps}\in \underline {\mathscr V}$.

Now let us measure the distance between $\underline V^{\eps}$, $\overline V^{\eps}$ and $V$. By the estimates for solutions of BSDEs (see \cite[Proposition 3.2]{Hu_2002} for instance), we first have 
\begin{align*}
\|Y^{\eps}\|_{\cS^4(\bR)} + \|Z^{\eps}\|_{\cL^4(\bR^{m})}
&
	\leq  C \left(  \|\Phi^{\eps}\|_{L^4(\Omega,\sF_T;\bR)} + \|f^{\eps}+\tilde L b^{\eps}\|_{\cL^4(\bR)}  \right)
\\
&
	\leq  C(1+\tilde L)\eps,
\end{align*}
with the constant $C$ independent of $N$ and $\eps$.  Fix some $(s,x)\in[0,T)\times \bR^d$. In view of the approximation in Lemma \ref{lem-approx}, using It\^o's formula, Burkholder-Davis-Gundy's inequality, and Gronwall's inequality, we have through standard computations that for any $(\theta,\mu)\in\Theta\times\mathcal M$,
\begin{align*}
&
	E_{\sF_s} \left [  \sup_{s\leq t\leq T} \left|  X^{s,x;\theta,\mu(\theta),N}_t-X^{s,x;\theta,\mu(\theta)}_t   \right| ^2  \right]
\\
&\leq
	\tilde C\bigg(\delta_N ^2
	+ E_{\sF_s}\int_s^T\left| b_N\left(\tilde W_{t_1\wedge t},\cdots,\tilde W_{t_N\wedge t},t,X^{s,x;\theta,\mu(\theta),N}_t,\theta_t,\mu(\theta)(t)\right) \right.
	\\
&\qquad\qquad\qquad\qquad\quad
	\left.-b\left(t,X^{s,x;\theta,\mu(\theta),N}_t,\theta_t,\mu(\theta)(t)\right) \right|^2\,dt\bigg)
	\\
&\leq
\tilde C \left(\delta_N^2+  E_{\sF_s}\int_s^T   \left| b^{\eps}_t\right|^2\,dt\right) ,
\end{align*}
with $\tilde C$ being independent of $s,\, x,\,N$, $\eps$, and $(\theta,\mu)$. Then the standard estimates for solutions of BSDEs indicate that
\begin{align*}
&E \left[  \left| V^{\eps}(s,x)-V(s,x)\right|^2\right]
\\\leq&\, C
		E \bigg[ \esssup_{(\theta,\mu)\in\Theta\times\cM}
		E_{\sF_s}\bigg[ 
		\int_s^T\Big( |f^{\eps}_t|^2 +L^2 \Big| X^{s,x;\theta,\mu(\theta),N}_t - X^{s,x;\theta,\mu(\theta)}_t \Big|^2    \Big)\,dt
\\&\quad\quad\quad
		+\left|\Phi^{\eps}\right|^2 
			+L^2\Big| X^{s,x;\theta,\mu(\theta),N}_T -X^{s,x;\theta,\mu(\theta)}_T \Big|^2
		\bigg] \bigg]
\\\leq &\, C
		E \bigg[ \esssup_{(\theta,\mu)\in\Theta\times\cM}
		E_{\sF_s}\bigg[ 
		\int_s^T    \left( \left| f^{\eps}_t\right|^2+ \left| b^{\eps}_t\right|^2\right)\,dt
		+\left|\Phi^{\eps} \right|^2+|\delta_N|^2
		\bigg] \bigg]
	\\ \leq&\, C \left(     
	E\left[\int_s^T    \left( \left| f^{\eps}_t\right|^4+ \left| b^{\eps}_t\right|^4\right)\,dt
		+\left|\Phi^{\eps} \right|^4 \right]\right)^{1/2}
	 + C |\delta_N|^2 
\\\leq&\, K_0 \left(\eps^2+|\delta_N|^2\right),
\end{align*}
with the constant $K_0$ being independent of $N$, $\eps$, and $(s,x)$. Furthermore, in view of the definitions of $\overline V^{\eps}$ and $\underline V^{\eps}$, there exists some constant $C_1$ independent of $\eps$ and $N$ such that
\begin{align*}
E\left| \overline V^{\eps}(s,x)-V(s,x)\right|^2
+E\left| \underline V^{\eps}(s,x)-V(s,x)\right|^2
\leq C_1 \left(\eps^2+|\delta_N|^2\right), \quad \forall \, (s,x)\in [0,T]\times\bR^d.
\end{align*}
The arbitrariness of $(\eps,\delta_N)$ together with the relation $\overline V^{\eps}\geq V \geq \underline V^{\eps}$ finally implies that $\underline u=V= \overline u$.  
\end{proof}

\begin{rmk}\label{rmk-unique}
In view of the above proof, we may see that the condition \textbf{(A2)} is assumed because of the possible degenerateness and the lack of certain estimates for $D_{xx}^2V$ and $D_x\mathfrak{d}_\omega V$. The lack of such estimates and the superparabolicity prevents us from using the perturbations of $\sigma$, which is why $\sigma$ may not depend on $(\omega,x,\theta,\gamma)$ in \textbf{(A2)}. 
\end{rmk}

We may now consider the superparabolic cases. In a similar way to \cite{qiu2017viscosity}, we first decompose the Wiener process $W=(\tilde{W},\bar{W})$ with $\tilde{W}$ and $\bar{W}$ being two mutually independent and respectively, $m_0$ and $m_1$($=m-m_0$) dimensional  Wiener processes. We also adopt the decomposition $\sigma=(\tilde{\sigma},\bar{\sigma})$ with $\tilde{\sigma}$ and $\bar{\sigma}$ valued in $\bR^{d\times m_0}$ and $\bR^{d\times m_1}$ respectively for the controlled diffusion coefficient $\sigma$, and associated with $(\tilde{W}, \bar{W})$. 
Denote by $\{\tilde{\sF}_t\}_{t\geq0}$ the natural filtration generated by $\tilde{W}$ and augmented by all the
$\bP$-null sets. Then we say the \textit{superparabolicity} holds if
\medskip\\
 (i)  For each $(t,x,y,z,{\theta},\gamma)\in [0,T]\times\bR^d\times \bR\times \bR^m\times \Theta_0\times \Gamma_0$,  $\Phi(x)$ is $\tilde{\sF}_T$-measurable and for the  random variables $h=b^i(t,x,\theta,\gamma),f(t,x,y,z,\theta,\gamma)$, $i=1,\cdot, d$,
$$
  h:~\Omega\rightarrow \bR
\text{ is } \tilde{\sF}_t \text{-measurable;}
$$
 (ii) The \textit{diffusion} coefficient $\sigma=(\tilde \sigma,\bar\sigma):~[0,T]\times\bR^d\times \Theta_0\times \Gamma_0\rightarrow\bR^{d\times m}$ is continuous and  does not depend on $\omega$, and there exists $\lambda\in(0,\infty)$  such that
   \begin{align*}
        \sum_{i,j=1}^d\sum_{k=1}^{m_1} \bar\sigma^{ik}\bar\sigma^{jk}(t,x,\theta,\gamma)\xi^i\xi^j\geq \lambda |\xi|^2\quad \,\,\forall\, (t,x,\theta,\gamma,\xi)\in [0,T]\times\bR^d \times \Theta_0\times \Gamma_0\times\bR^d.
   \end{align*}

In fact, instead of \textbf{(A2)}, we may assume that there hold the superparabolicity and either of the following three conditions: 
 \\
\textbf{(A2$^*$)} the \textit{diffusion} coefficient $\sigma:[0,T]\times \bR^d \times \Theta_0 \rightarrow\bR^{d\times m}$  does not depend on $(\omega,\gamma)\in \Omega \times\Gamma_0$.
\\
\textbf{(A2$^{**}$)} the \textit{diffusion} coefficient $\sigma:[0,T]\times \bR^d \times \Gamma_0 \rightarrow\bR^{d\times m}$  does not depend on $(\omega,\theta)\in \Omega \times\Theta_0$.
\\ 
\textbf{(A2$^{***}$)} $d\leq 2$ and the \textit{diffusion} coefficient $\sigma:[0,T]\times \bR^d \times\Theta_0\times \Gamma_0 \rightarrow\bR^{d\times m}$  does not depend on $\omega\in \Omega$.
\\
In either of these three cases,  one does not need to enlarge the probability space to introduce another independent Wiener process $B$. Instead, for the cases of \textbf{(A2$^{*}$)} and \textbf{(A2$^{**}$)}, one may utilize the $C^{1+\frac{\bar\alpha}{2},2+\bar\alpha}-$estimate (for some $\bar\alpha \in (0,1)$) of viscosity solutions to deterministic HJB equations (see  \cite[Proposition 3.7]{wang1992-I} and \cite[Theorem 1.1]{krylov1982boundedly} for instance), while for the case of \textbf{(A2$^{***}$)}, one may use the $C^{1,2}-$estimate (that is actually sufficient) of viscosity solutions to deterministic  HJBI equations (see \cite[Lemma 6.5]{pham2014two} for instance). The proofs will then follow in a similar way to that of \cite[Theorem 5.6]{qiu2017viscosity}, and they are omitted.\\

 \begin{rmk}
Similar results on the uniqueness hold for BSPDE \eqref{SHJBI-up} (equivalently \eqref{SHJB-up-eqiv}). 
\end{rmk}

Finally, assume the Isaacs condition: 
\begin{equation}\label{Isaasccondition}
\mathbb{H}_-(t,x,A,B,p,y,z)=\mathbb{H}_+(t,x,A,B,p,y,z)=:\mathbb{H}(t,x,A,B,p,y,z), 
\end{equation}
for all $(t,x,A,B,p,y,z)\in [0,T]\times\bR^d\times\bR^{d\times d}\times \bR^{m\times d} \times   \bR^d\times \bR \times \bR^m$,
and consider the following BSPDE:
\begin{equation}\label{IsaacsBSPDE}
  \left\{
  \begin{split}
  -du(t,x)=\,& 
 \mathbb{H}(t,x,D^2u,D\zeta,Du,u,\zeta) 
 \,dt -\zeta(t,x)\, dW_{t}, \quad
                     (t,x)\in Q;\\
    u(T,x)=\, &\Phi(x), \quad x\in\bR^d.
    \end{split}
  \right.
\end{equation}

\begin{thm}\label{mainresultbis}
Let Assumption {\bf (A1)} and the Isaacs condition \eqref{Isaasccondition} hold. Assume further that the uniqueness of viscosity solution to BSPDE \eqref{IsaacsBSPDE} holds. Then the game value exists, i.e., $V=U=:u$ with $u$ being the unique viscosity solution to BSPDE \eqref{IsaacsBSPDE}. 
\end{thm}
\begin{proof}
By Theorem \ref{existence} and the uniqueness of viscosity solutions, we obtain immediately that $V=U$ and it is the unique viscosity solution to BSPDE \eqref{IsaacsBSPDE}. 
\end{proof}



\begin{appendix}
\section{Measurable selection theorem}
The following measurable selection theorem is referred to \cite{wagner1977survey}.
\begin{thm}\label{thm-MS}
Let $(\Lambda,\mathscr M)$ be a measurable space equipped with a nonnegative measure $\mu$ and let $(\cO,\cB(\cO))$ be a polish space. Suppose $F$ is a set-valued function from $\Lambda$ to $\cB(\cO)$ satisfying: (i) for $\mu$-a.e. $\lambda \in\Lambda$, $F(\lambda)$ is a closed nonempty subset of $\cO$; (ii) for any open set $O\subset \cO$, $\{\lambda:\, F(\lambda)\cap O\neq \emptyset\}\in \mathscr M$. Then there exists a  measurable function $f$: $(\Lambda,\mathscr M)\rightarrow (\cO,\cB(\cO))$ such that for $\mu$-a.e. $\lambda\in\Lambda$, $f(\lambda)\in F(\lambda)$. 
\end{thm}

\section{A result on BSDEs}
The following proposition summarizes the standard wellposedness and comparison principle of BSDEs, whose proofs can be found in \cite{Hu_2002,El-Karoui-Peng-Quenez-2001,ParPeng_90}.

\begin{prop}\label{prop-BSDE-comp}
Assume 
\\
(a) $g:\Omega\times[0,T]\times\bR\times\bR^m \rightarrow \bR$  {is}
$\sP\otimes\cB(\bR)\otimes\cB(\bR^m) \text{-measurable}$ with $g(\cdot,0,0)\in\cL^1(\bR)\cap L^2(\Omega; L^1(0,T;\bR))$;
\\
(b) there exists $\tilde L_0\geq 0$ such that for almost all $(\omega,t)\in \Omega\times [0,T]$, $y_1,y_2\in\bR$ and $z_1,z_2\in\bR^m$,
$$
|g(t,y_1,z_1)-g(t,y_2,z_2)|\leq \tilde L_0\left( |y_1-y_2| + |z_1-z_2|   \right).
$$

\noindent For the following BSDE:
\begin{align}
y_t=\xi+\int_t^T g(s,y_s,z_s)\,ds -\int_t^Tz_s\,dW_s, \quad t\in[0,T],  \label{BSDE-g}
\end{align}
we have the following assertions hold.
\\
(i) For each $\xi\in L^2(\Omega,\sF_T;\bR)$, BSDE \eqref{BSDE-g} admits a unique solution $(y,z)\in\cS^2(\bR)\times \cL^2(\bR^m)$ with
$$\|y\|^2_{\mathcal{S}^2(\mathbb{R})}+\|y\|^2_{\mathcal{L}^2(\mathbb{R})} +\|z\|^2_{\mathcal{L}^2(\mathbb{R}^m)}\leq C(T,\tilde L_0)E\left[|\xi|^2+\left(\int_0^T|g(t,0,0)|dt\right)^2\right].$$
\\
(ii) Given two coefficients $g_1$ and $g_2$ satisfying (a) and (b) and two terminal values $\xi_1,\xi_2\in L^2(\Omega,\sF_T;\bR)$, denote by $(y^1,z^1)$ and $(y^1,z^2)$ the solution of BSDE associated with data $(g_1,\xi_1)$ and $(g_2,\xi_2)$ respectively. It holds that:
if $\xi_1\leq \xi_2$ and $g_1(t,y^2_t,z^2_t)\leq g_2(t,y^2_t,z^2_t)$ a.s. for all $t\in[0,T]$, then $y^1_t\leq y^2_t$, a.s. for all $t\in[0,T]$; if we have further $\bP(\xi_1<\xi_2)>0$, then $\bP(y^1_t<y^2_t)>0$, for all $t\in[0,T]$.
\\
(iii) Let 
$$
g_i(t,y_t^i,z_t^i)=g(t,y_t^i,z_t^i) + h_i(t), \quad\text{for almost all }(\omega,t)\in \Omega\times [0,T],\quad i=1,2,
$$
with $h_i\in \cL^1(\bR)\cap L^2(\Omega; L^1(0,T;\bR))$, $i=1,2$. For $\xi_1,\xi_2\in L^2(\Omega,\sF_T;\bR) $, letting $(y^1,z^1)$ and $(y^2,z^2)$ be the solution of BSDE \eqref{BSDE-g} associated with $(g_1,\xi_1)$ and $(g_2,\xi_2)$ respectively, we have for all $t\in[0,T]$,
\begin{align*}
&\left|y^1_t-y^2_t \right|^2
	+ E_{\sF_t} \int_t^T \left( \left|y^1_s-y^2_s\right|^2+ \left|z^1_s-z^2_s\right|^2   \right)ds
\\
&\leq 
	C_0 \left\{
		E_{\sF_t}\left[  \left|\xi_1-\xi_2\right|^2    
			+\left(\int_t^T\left|h_1(s)-h_2(s)\right|\,ds\right)^2
		\right]
\right\}\quad \text{a.s.},
\end{align*}
with the constant $C_0$ depending only on $\tilde L_0$ and $T$.
\end{prop}

%

\section{Proofs of Lemmas \ref{lemdifferenceofY} and \ref{estimateYZ}}
\begin{proof}[Proof of Lemma \ref{lemdifferenceofY}]
Note that $(T_0,\delta) \in (0,T)\times (0,1)$. For all $p\geq 1$, Lemma \ref{lem-SDE} (ii) combined with the estimate
\begin{equation*}
\begin{split}
E_{\sF_\tau}\left[\sup_{\tau\leq s\leq \tau+\delta \wedge T_0}|X_s^{\tau,\xi;\theta,\gamma}-\xi|^p\right]\leq &\, 2^{p-1}E_{\sF_\tau}\left[\sup_{\tau\leq s\leq \tau+\delta \wedge T_0}\left|\int_\tau^sb(r,X_r^{\tau,\xi;\theta,\gamma},\theta_r,\gamma_r)dr\right|^p\right]\\&+2^{p-1}E_{\sF_\tau}\left[\sup_{\tau\leq s\leq \tau+\delta \wedge T_0}\left|\int_\tau^s\sigma(r,X_r^{\tau,\xi;\theta,\gamma},\theta_r,\gamma_r)dW_r\right|^p\right]
\end{split}
\end{equation*}
yields that
\begin{equation}\label{controlXminusxi}
E_{\sF_\tau}\left[\sup_{\tau\leq s\leq \tau+\delta \wedge T_0}|X_s^{\tau,\xi;\theta,\gamma}-\xi|^p\right]\leq C (1+ |\xi|^p) \delta^{\frac{p}{2}}, \quad \text{a.s.},
\end{equation}
uniformly in $\theta\in\Theta$ and $\gamma\in\Gamma$.  Using Proposition \ref{prop-BSDE-comp} combined with \eqref{controlXminusxi} and \eqref{Lip-const} further gives
\begin{equation*}\begin{split}
&E_{\sF_\tau}\left[\int_\tau^{\tau+\delta\wedge T_0}\left(|Y_s^{1,\theta,\gamma}-Y_s^{2,\theta,\gamma}|^2+|Z_s^{1,\theta,\gamma}-Z_s^{2,\theta,\gamma}|^2\right)ds\right]\\
 \leq&\, CE_{\sF_\tau}\left[\left( \int_\tau^{\tau+\delta\wedge T_0} \zeta^{\phi}_s|X_s^{\tau,\xi;\theta,\gamma}-\xi|ds \right)^2\right]\\
  \leq &\,C  \delta^{\frac{3}{2}}\cdot   \sqrt{E_{\sF_\tau}\left[ \int_\tau^{\tau+\delta\wedge T_0} \big| \zeta^{\phi}_s\big|^4   ds \right]\cdot  E_{\sF_\tau}\left[\sup_{\tau\leq s\leq \tau+\delta\wedge T_0}|X_s^{\tau,\xi;\theta,\gamma}-\xi|^4\right]}\\
   \leq &\,C(1+|\xi|^2)  \delta^{\frac{5}{2}}\cdot   \left( E_{\sF_\tau}\left[ \int_0^{T_0}\big| \zeta^{\phi}_s\big|^4 ds \right]\right)^{1/2},\quad \text{a.s.}
\end{split}\end{equation*}

\noindent Therefore, 
\begin{equation*}\begin{split}
&\big|Y_\tau^{1,\theta,\gamma}-Y_\tau^{2,\theta,\gamma}\big|=\Big|E_{\sF_\tau}\big[Y_\tau^{1,\theta,\gamma}-Y_\tau^{2,\theta,\gamma}\big]\Big|\\
=
&
	\left|E_{\sF_\tau}\Big[\int_\tau^{\tau+\delta \wedge T_0}(F(s,X_s^{\tau,\xi;\theta,\gamma},Y_s^{1,\theta,\gamma},Z_s^{1,\theta,\gamma},\theta_s,\gamma_s)-F(s,\xi,Y_s^{2,\theta,\gamma},Z_s^{2,\theta,\gamma},\theta_s,\gamma_s))ds\Big]\right|
\\
\leq
&\,
	CE_{\sF_\tau}\left[\int_\tau^{\tau+\delta \wedge T_0}\left[\zeta_s^{\phi}|X_s^{\tau,\xi;\theta,\gamma}-\xi|+\big|Y_s^{1,\theta,\gamma}-Y_s^{2,\theta,\gamma}\big|+\big|Z_s^{1,\theta,\gamma}-Z_s^{2,\theta,\gamma}\big|\right]ds\right]
\\
\leq&\,
	C\delta^{\frac{3}{4}}\cdot   \sqrt{E_{\sF_\tau}\left[ \Big( \int_\tau^{\tau+\delta\wedge T_0} \big| \zeta^{\phi}_s\big|^4   ds\Big)^{1/2} \right]\cdot  E_{\sF_\tau}\left[\sup_{\tau\leq s\leq \tau+\delta\wedge T_0}|X_s^{\tau,\xi;\theta,\gamma}-\xi|^2\right]}\\
	\\&+C\delta^{\frac{1}{2}} \left( E_{\sF_\tau}\left[\int_\tau^{\tau+\delta \wedge T_0} \Big(\big|Y_s^{1,\theta,\gamma}-Y_s^{2,\theta,\gamma}\big|^2 + \big|Z_s^{1,\theta,\gamma}-Z_s^{2,\theta,\gamma}\big|^2\Big) ds\right]\right)^{\frac{1}{2}}
\\
\leq &\,\left(\delta^{\frac{5}{4}} +\delta^{\frac{7}{4}}\right) C(1+|\xi|) 
 \left( E_{\sF_\tau}\left[\int_0^{T_0}\big| \zeta^{\phi}_s\big|^4 ds \right]\right)^{1/4},\quad \text{a.s.}
\end{split}\end{equation*}
This completes the proof as $0<\delta<1$. 
\end{proof}

\begin{proof}[Proof of Lemma \ref{estimateYZ}]
Since $F(s,x,\cdot,\cdot,\theta,\gamma)$ has a linear growth in $(y,z)$, uniformly in $(\theta,\gamma)$, Proposition \ref{prop-BSDE-comp} gives the following estimates: for any $s\in[\tau,\tau+\delta \wedge T_0]$, 
\begin{align*}
E_{\sF_s}\left[\int_s^{\tau+\delta \wedge T_0 }\big|Y_t^{2,\theta,\gamma}\big|^2 +   \big|Z_t^{2,\theta,\gamma}\big|^2dt\right]
&
\leq
C E_{\sF_s}\left[\left(\int_s^{\tau+\delta \wedge T_0 }\big|  \zeta^{\phi}_t    \big|dt  \right)^{2}\right]
\\
\text{(by H\"{o}lder's inequality)}\quad &
\leq  C\delta^{\frac{3}{2}} 
	\left(E_{\sF_s}\left[\int_s^{\tau+\delta \wedge T_0 }\big|  \zeta^{\phi}_t   \big|^4 dt \right] \right)^{\frac{1}{2}}
, \text{ a.s.}
\end{align*}
Then, the standard application of H\"{o}lder's inequality implies that 
\begin{equation*}
\begin{split}
&E_{\sF_\tau}\left[\int_\tau^{\tau+\delta  \wedge T_0}\big|Y_t^{2,\theta,\gamma}\big|dt\right]+E_{\sF_\tau}\left[\int_\tau^{\tau+\delta \wedge T_0}\big|Z_t^{2,\theta,\gamma}\big|dt\right]
\\
\leq&\,
\sqrt{\delta}\left(E_{\sF_{\tau}}\left[\int_{\tau}^{\tau+\delta \wedge T_0}\big|Y_t^{2,\theta,\gamma}\big|^2dt\right]\right)^{\frac{1}{2}}
+\sqrt{\delta}\left(E_{\sF_{\tau}}\left[\int_{\tau}^{\tau+\delta \wedge T_0}\big|Z_t^{2,\theta,\gamma}\big|^2dt\right]\right)^{\frac{1}{2}}
\\
\leq&\,C \delta^{\frac{5}{4}} \left( E_{\sF_\tau}\left[\int_0^{T_0}\big|  \zeta^{\phi}_t \big|^4 dt \right]\right) ^{1/4},\quad \text{a.s.}
\end{split}
\end{equation*}
Hence, the desired estimate is obtained and the proof for $(Y^0, Z^0)$ follows in a similar way. 
\end{proof}

\end{appendix}

\section*{Acknowledgements}
The authors are very grateful to the editor and the anonymous referees for their very valuable remarks and comments which have made it possible to improve our paper. J. Qiu would also thank Professor Hongjie Dong from Brown University and Jianfeng Zhang from the University of Southern California for helpful communications about classical solutions of deterministic Hamilton-Jacobi-Bellman-Isaacs equations.

\bibliographystyle{siam}

\end{document}